\documentclass[graybox]{svmult}
\usepackage{etex}
\usepackage{url}
\usepackage{amsmath, stmaryrd}
\usepackage{amssymb}
\usepackage{mathptmx}       % selects Times Roman as basic font
\usepackage{helvet}         % selects Helvetica as sans-serif font
\usepackage{courier}        % selects Courier as typewriter font
\usepackage{type1cm}        % activate if the above 3 fonts are
\usepackage{placeins}                           
\usepackage{makeidx}         % allows index generation
\usepackage{graphicx}        % standard LaTeX graphics tool
\usepackage{amscd}
\usepackage{multicol}        % used for the two-column index
\usepackage[bottom]{footmisc}% places footnotes at page bottom

\usepackage[all,cmtip]{xy}  %xypic
\usepackage{tikz}
\let\coprod=\undefined

\DeclareSymbolFont{cmlargesymbols}{OMX}{cmex}{m}{n}

\DeclareMathSymbol{\coprod}{\mathop}{cmlargesymbols}{"60}

\let\amalg=\undefined

\DeclareSymbolFont{cmsymbols}{OMS}{cmsy}{m}{n}

\DeclareMathSymbol{\amalg}{\mathbin}{cmsymbols}{"71}
\makeindex             % used for the subject index
\usepackage{amssymb}
\usepackage{graphicx}
\usepackage{float}
\usepackage[nooneline]{caption}
\usepackage{color}
\usetikzlibrary{snakes}

\newcommand {\shM}  {\mathcal{M}}

\newcommand {\shN}  {\mathcal{N}}
\newcommand {\shO}  {\mathcal{O}}

\newcommand {\shX}  {\mathcal{X}}

\renewcommand{\P}{\mathcal{P}}
\newcommand {\nnn} {\widetilde{\mathbb{N}}}
\newcommand {\ra} {\rightarrow}
\renewcommand{\AA}{\mathbb{A}}
\newcommand {\NN} {\mathbb{N}}
\newcommand {\ZZ} {\mathbb{Z}}
\newcommand {\DD} {\mathbb{D}}
\newcommand {\GG} {\mathbb{G}}
\newcommand{\PP}{\mathbb{P}}
\newcommand {\QQ} {\mathbb{Q}}
\newcommand {\RR} {\mathbb{R}}
\newcommand {\CC} {\mathbb{C}}
\newcommand {\Real} {\rm{Re}}
\newcommand{\Proj}{\rm{Proj\,}}
\newcommand {\Spec} {\rm{Spec\,}}
\newcommand {\Specf} {\rm{Specf\,}}
\newcommand {\Strata} {\operatorname{Strata}}
\newcommand {\stratum} {\operatorname{stratum}}

\newcommand {\Log} {\operatorname{Log}}
\newcommand {\val} {\operatorname{val}}
\newcommand {\Val} {\operatorname{Val}}
\newcommand {\Hom} {\rm{Hom}}
\newcommand {\ov} {\operatorname{ov}}
\newcommand {\op}[1] {\operatorname{#1}}
\newcommand {\trop}  {{\operatorname{trop}}}
\newcommand {\hol}  {\mathrm{hol}}
\newcommand {\Bl}  {\operatorname{Bl}}
\newcommand {\lra} {\longrightarrow}
\newcommand {\Map}  {\operatorname{Map}}
\newcommand {\Mult} {\operatorname{Mult}}
\newcommand {\coker}  {\operatorname{coker}}
\newcommand {\Tor}  {\operatorname{Tor}}

\newcommand{\IN}{\mathbb N}
\newcommand{\IC}{\mathbb C}

\newcommand{\mc}{\mathcal}
\newcommand{\aaa}{\mathbb{A}^1_k}
\newcommand{\aaaa}{\mathbb{A}^2_k}
\newcommand{\orig}{\left\{ 0 \right\}}

\DeclareMathOperator{\spec}{Spec}
\DeclareMathOperator{\diff}{d}
\DeclareMathOperator{\hhh}{H}
\DeclareMathOperator{\ddd}{D}
\DeclareMathOperator{\dddlog}{Dlog}

\makeatletter
\let\c@proposition\c@theorem
\let\c@corollary\c@theorem
\let\c@lemma\c@theorem
\let\c@definition\c@theorem
\let\c@example\c@theorem
\let\c@exercise\c@theorem
\makeatother
\numberwithin{theorem}{section}
\numberwithin{proposition}{section}
\numberwithin{corollary}{section}
\numberwithin{lemma}{section}
\numberwithin{definition}{section}
\numberwithin{example}{section}
\numberwithin{exercise}{section}
\numberwithin{figure}{section}

\begin{document}

\title*{Enumerative aspects of the Gross-Siebert program}
\author{Michel van Garrel
\and D. Peter Overholser
\and Helge Ruddat}
\institute{Michel van Garrel \at KIAS, 85 Hoegiro, Dongdaemun-gu, Seoul 130-722, Republic of Korea\\ \email{vangarrel@kias.re.kr} %Helge Ruddat \at Mathematisches Institut, JGU Mainz, Staudingerweg 9, D-55128 Mainz\\ 
%\email{ruddat@uni-mainz.de}
\and D. Peter Overholser \at KU Leuven, Celestijnenlaan 200b,
3001 Leuven\\ \email{douglas.overholser@wis.kuleuven.be}
\and Helge Ruddat \at Mathematisches Institut, JGU Mainz, Staudingerweg 9, D-55128 Mainz\\ 
\email{ruddat@uni-mainz.de}
}
\maketitle

\abstract{We present enumerative aspects of the Gross-Siebert program in this
introductory survey.  After sketching the program's main themes and goals, we review the basic
definitions and results of logarithmic and tropical geometry.  We give examples and a proof for counting algebraic curves via tropical
curves.  To illustrate an application of tropical geometry and the Gross-Siebert
program to mirror symmetry, we discuss the mirror symmetry of the
projective plane.}

\section{Introduction}
\label{sec:introduction}
We begin with a brief description of the motivations and major ideas of the Gross-Siebert program.  These will serve as the target about which the rest of this exposition is roughly clustered.
\subsection{The Strominger-Yau-Zaslow conjecture and Gross-Siebert program}
\label{SYZ-GS}
A duality of special Lagrangian torus fibrations $X\ra B\leftarrow \check X$ of a Calabi-Yau $X$ and its mirror dual $\check X$ was conjectured by Strominger-Yau-Zaslow (SYZ) to be the geometric principle underlying mirror symmetry \cite{SYZ96}. This intrinsic approach overcomes the need to embed Calabi-Yau threefolds in toric Fano varieties to study their mirror duals and allows patching local constructions.  
Hitchin \cite{Hi97} noticed that, given such a fibration, both the complex and symplectic structure of $X$ give a real affine structure outside of the discriminant locus $\Delta$ on $B$.  Furthermore, the two are related by a Legendre transform. In such a fibration, the roles of the affine structures are swapped for the mirror dual $\check X$, e.g. the complex structure of $X$ and the symplectic structure of $\check X$ yield the same affine structure.  The discriminant locus of the fibration $\Delta$ in $B$ coincides with the locus of real affine singularities of $B$.  On the other hand, given an affine manifold $B$ without singularities, one can construct both a K\"ahler and a complex manifold torically fibered over $B$, suggesting that the base may contain the information necessary to describe the mirror relationship.  We will call the process of constructing a manifold from the affine base \emph{reconstruction}.

In practice, it can be difficult to find even a single special Lagrangian torus, let alone a fibration.  Nevertheless, families of Calabi-Yau's were observed to collapse to the base of such a fibration near suitably bad (large complex structure limit) degenerations.  More precisely, 
in \cite{GW00} Gross and Wilson studied the K3 case by combining the SYZ picture with the \emph{Gromov-Hausdorff limit}, a metric limit where the fibres of the SYZ fibration shrink to points such that the limit coincides (as a metric space) with $B$.  If one can recover the base of our desired fibration in such a way, and the base holds the information needed for mirror symmetry, this suggests a plan of attack.  In particular, one may dream of starting with a family of manifolds, degenerating to the base, and reconstructing a mirror family.

This is precisely the motivating principle behind the Gross-Siebert program.  The general large complex structure limit degeneration is replaced by a maximally unipotent degeneration of the Calabi-Yau manifold called a \emph{toric degeneration}, where the central fiber is (roughly) glued from toric varieties along toric strata.  Gross and Siebert succeeded in combining the SYZ approach with such degenerations, giving a versatile algebro-geometric framework for the study of mirror symmetry. The affine manifold appears in their work as the dual intersection complex of the special fibre.

The key concept is to encode information about the degeneration entirely in $B$.  A toric degeneration gives additional data on $B$ beyond the affine structure, namely a polyhedral decomposition $\P$ and discrete Legendre potential $\varphi$.
At the level of degeneration data, mirror symmetry is realized by a discrete Legendre transform
$$(B,\P,\varphi) \leftrightarrow (\check B,\check \P,\check\varphi)$$
discretizing Hitchin's Legendre duality.

Kontsevich and Soibelman \cite{KS06} demonstrated how one could reconstruct a K3 surface from an affine structure with singularities on $S^2$.
Using logarithmic geometry, Gross and Siebert were able to solve the reconstruction problem \cite{GS11} in any dimension, obtaining a degenerating family of Calabi-Yau manifolds $\shX \ra \DD$ over a holomorphic disk from the information of $(B,\P,\varphi)$ and a log structure. 
Furthermore, this family is parametrized by a canonical coordinate (in the usual sense in mirror symmetry). The construction features wall-crossings and scatterings, structures that encode enumerative information linking symplectic with complex geometry via tropical geometry. As will be hinted at in this exposition, Gromov-Witten theory \cite{loggw} can also be incorporated in this framework.
\subsection{Toric conventions}
We assume familiarity with toric geometry.  The interested reader is referred to the excellent exposition of Fulton \cite{fulton}.  As the following story is closely tied to toric geometry, it is convenient to begin by making a few conventions regarding notation.

Set $M:=\mathbb{Z}^n,$  $M_\mathbb{R}:=M\otimes_{\mathbb{Z}}\mathbb{R}$, $N:=\Hom_\mathbb{Z} (M, \mathbb{Z})$,  $N_\mathbb{R}:=N\otimes_{\mathbb{Z}}\mathbb{R}$.  For $n \in N,$ set $\langle n, m\rangle$ to be the evaluation of $n$ on $m$.
Set a toric fan $\Sigma$ in $M_\mathbb{R}$.  Let $\Sigma^{[n]}$ signify the set of $n$ dimensional cones of $\Sigma$.  Let $X_\Sigma$ be the toric variety defined by $\Sigma$.  

Denote by $T_\Sigma$ the free abelian group generated by $\Sigma^{[1]}$.  For $\rho\in \Sigma^{[1]}$, denote by $v_\rho$ the corresponding generator in $T_\Sigma$.   We will need the map
\begin{eqnarray*}
r:T_\Sigma\rightarrow M_\mathbb{R}\\
v_\rho \mapsto \hat{\rho}.
\end{eqnarray*}
where $\hat{\rho}$ is the integral vector generating $\rho$, that is $\rho\cap M=\mathbb{Z}_{\geq 0} \hat{\rho}$.

\subsection{Toric degenerations}\label{toriccon}
The object at the heart of the Gross-Siebert program is the \emph{toric degeneration}.  These are meant to be the algebro-geometric analogues of the large complex structure limit discussed above.  Let $R$ be a discrete valuation ring over an algebraically closed field $k$.
\begin{definition}
\label{toricdegeneration}
A \emph{toric degeneration} is a normal algebraic space $\mathcal{X}$ flat over $\spec R$
\[\xymatrixcolsep{0.5pc}\xymatrix{
\mathcal{X} \ar[d] & \supset & \mathcal{X}_0 \ar[d] \\
\spec R & \ni & 0
}
\]
such that:
\begin{enumerate}
\item The general fiber is irreducible and normal.
\item If $\nu :\widetilde{\mathcal{X}}_0\ra \mathcal{X}_0$ is the normalization, $\widetilde{\mathcal{X}}_0$ is a disjoint union $\coprod X_i$ of toric varieties that are glued along toric strata to form $\mathcal{X}_0$.  Furthermore, the conductor locus $C\subseteq \mathcal{X}_0$ is reduced, and the map $C\ra \nu (C)$ is unramified and generically two-to-one.  The square
$$\begin{CD}
C    @>>>  \widetilde{\mathcal{X}}_0\\
@VVV        @VV\nu V\\
\nu(C)     @>>>  \mathcal{X}_0
\end{CD}$$
is Cartesian and co-Cartesian.
\item $\mathcal{X}_0$ is a reduced Gorenstein space and $C$ restricted to each irreducible component of $\widetilde{\mathcal{X}}_0$ is the union of all toric Weil divisors of that component.
\item There exists a closed subset $Z \subseteq \mathcal{X}$ of relative codimension 2 such that it does not contain the image under $\nu$ of any toric stratum of $\widetilde{\mathcal{X}}_0$.  Furthermore, outside of $Z$, all points $x$ of $\mathcal{X}$ have a local toric model. More precisely, we require the existence of a monoid $M_x \supseteq \NN$ and an open set $U_x$ satisfying:
\[
\xymatrixcolsep{0.8pc}\xymatrix{
\spec k[M_x] \ar[dd] & & \ar[ll] U_x \ar[ld]|-{\text{smooth}} \ar@{^{(}->}[r] \ar[dd] & \mathcal{X} \ar[ddl]^f \\
 & \spec k[M_x]\times_{k[\mathbb{N}]}\spec R \ar[rd] \ar[lu] \ar@{}[ld]_>>>>>{\square}\\
\spec k[\mathbb{N}] & & \spec R \ar[ll]
}
\]
Furthermore the map $U_x\ra\spec k[M_x]$ identifies $X_0\cap U_x$ with the toric boundary divisor in $\spec k[M_x]$ near the origin.
\end{enumerate}

\begin{remark} 
Note that item 4 of the definition can be rephrased by just saying that $f:\shX\setminus Z\ra \spec R$ is log smooth, cf. \cite{Kato_log_struct}. See Section \ref{sec:loggeom} for more on log structures.
\end{remark}

\end{definition}
Let $j:\mathcal{X}\setminus \mathcal{X}_0 \hookrightarrow \mathcal{X}$ be the inclusion.  The monoid sheaf 
$$\mathcal{M}_{\mathcal{X}, \mathcal{X}_0}:= \mathcal{O}_\mathcal{X}\cap j_*\mathcal{O}^\times_{\mathcal{X}\setminus \mathcal{X}_0}$$ 
gives a log structure on $\mathcal{X}$ and, by pulling back, one on $\mathcal{X}_0$. See Section \ref{sec:loggeom}. 
We will spend much of our energies analyzing the affine structure derived from the combinatorial data of a degeneration, so we give a name for objects obtained in this fashion.   
\begin{definition}
\label{def-log-CY-space}
A \emph{toric log Calabi Yau space} is the type of log space $(\mathcal{X}_0, \mathcal{M}_{\mathcal{X}, \mathcal{X}_0}|_{\mathcal{X}_0})$ that can appear in the previous definition as a central fiber.  
\end{definition}

To reassure the reader that these technical definitions are not vacuous, we provide a concrete example.

\begin{example}
Let $\mathcal{X}:=\{ tf + z_0z_1z_2z_3=0\}\subseteq \mathbb{P}^3 \times \mathbb{A}_t^1$, with $f_4$ a generic quartic.  Note that $\mathcal{X}$ is the blowup of $\mathbb{P}^3$ along the union of the hypersurface defined by $f_4$ and that defined by $z_0z_1z_2z_3=0$.  
The singular locus is given  by $\{t=f_4=0\}\cap \rm{Sing}(\mathcal{X}_0)$.   As $\mathcal{X}_0$ is the coordinate tetrahedron, we expect four points of intersection of $\{f_4=0\}$ with each edge, giving a total of 24 singular points.  Defining $Z=Sing(\mathcal{X})$, it's easy to see that this is an example of a toric degeneration.
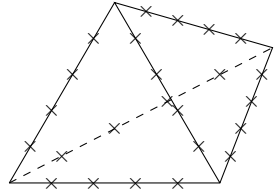
\begin{figure}
\centering
\[
\begin{tikzpicture}[xscale=0.7,yscale=0.6]
\draw (0,0) -- (4,0) -- (5,3) -- (2,4) -- (4,0);
\draw (0,0) -- (2,4);
\draw [dashed] (0,0) -- (5,3);
\node at (0.4,0.8) {$\times$};
\node at (0.8,1.6) {$\times$};
\node at (1.2,2.4) {$\times$};
\node at (1.6,3.2) {$\times$};
\node at (2.6,3.8) {$\times$};
\node at (3.2,3.6) {$\times$};
\node at (3.8,3.4) {$\times$};
\node at (4.4,3.2) {$\times$};
\node at (3.6,0.8) {$\times$};
\node at (3.2,1.6) {$\times$};
\node at (2.8,2.4) {$\times$};
\node at (2.4,3.2) {$\times$};
\node at (0.8,0) {$\times$};
\node at (1.6,0) {$\times$};
\node at (2.4,0) {$\times$};
\node at (3.2,0) {$\times$};
\node at (4.2,.6) {$\times$};
\node at (4.4,1.2) {$\times$};
\node at (4.6,1.8) {$\times$};
\node at (4.8,2.4) {$\times$};
\node at (1,.6) {$\times$};
\node at (2,1.2) {$\times$};
\node at (3,1.8) {$\times$};
\node at (4,2.4) {$\times$};
\end{tikzpicture}
\]
\caption{The set $Z\subseteq \mathcal{X}_0$ defined by the singularities of $\mathcal{X}$.}
\end{figure}
We set $Z=Sing(\mathcal{X})$.  Then $\mathcal{X}\rightarrow \mathbb{A}^1$ is a toric degeneration.  

Given $x\in \mathcal{X}_0\setminus Z$, what monoid $M_x$ is related to the local toric model?  Define $\stratum(x)\subset \Delta$ to be the manifestation of the toric stratum containing $x$ in the Newton polytope $\Delta$ of $\mathbb{P}^3$.

Define $\widehat{M}_x :=\mathbb{R}_{\geq 0}(\Delta - \stratum(x))\cap M$.  Then $M_x=\widehat{M}_x/\widehat{M}^\times_x.$  See Figure \ref{tgames}.
\begin{figure}
\centering
\[
\begin{tikzpicture}[xscale=1,yscale=1]
\draw[gray] (0,0.2) -- (1,0);
\draw (0.1,-0.2) -- (1,0) -- (0.5,1.3) -- (0,0.2) -- (0.1,-0.2) -- (0.5,1.3);
\draw [fill] (0.125,0.475) circle [radius=0.04];
\node [left] at (0.2,0.7) {$x$};
\draw [fill] (0.4,0.3) circle [radius=0.04];
\node at (0.6,0.35) {$0$};
\node at (1.5,0.5) {$\leadsto$};
\node at (2.2,0.5) {$\widehat{M}_x$};
\draw (2.7,1.5) -- (2.7,-0.5);
\draw [dashed] (2.7,-0.6) -- (2.7,-1.2);
\draw [dashed] (2.7,1.6) -- (2.7,2.2);
\draw [dashed] (2.7,-0.5) -- (3.8,-0.1);
\draw (2.7,-0.5) -- (3.9,-0.9);
\draw (2.7,1.5) -- (3.8,1.9);
\draw (2.7,1.5) -- (3.9,1.1);
\draw (3.9,1.1) -- (3.9,-0.9);
\draw [dashed] (3.8,1.05) -- (3.8,-0.1);
\draw (3.8,1.13) -- (3.8,1.9);
\draw [fill] (2.7,-0.5) circle [radius=0.04];
\draw [fill] (2.7,0.5) circle [radius=0.04];
\draw [fill] (2.7,1) circle [radius=0.04];
\draw [fill] (2.7,0) circle [radius=0.04];
\draw [fill] (3.3,0.6) circle [radius=0.04];
\node [right] at (3.3,0.6) {$0$};
\node at (0.5,-2) {$\leadsto$};
\node at (2,-2) {$\mathbb{N} \subseteq M_x=$};
\draw (3.2,-2) -- (4.45,-1.25);
\draw (3.2,-2) -- (4.45,-2.75);
\draw [dashed] (3.2,-2) -- (5.5,-2);
\draw [fill] (3.23,-2) circle [radius=0.04];
\draw [fill] (3.8,-2) circle [radius=0.04];
\node [below right] at (3.8,-2) {$0$};
\draw [fill] (4.4,-2) circle [radius=0.04];
\draw [fill] (5,-2) circle [radius=0.04];
\draw (3.7,-1.7) to [out=-54,in=54] (3.7,-2.3);
\draw (4.2,-1.4) to [out=-54,in=54] (4.2,-2.6);
\end{tikzpicture}
\]
\caption{The construction of $M_x$.}
\label{tgames}
\end{figure}
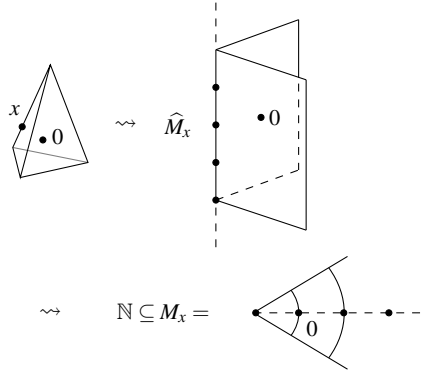
\end{example}

Toric degenerations are highly relevant to the theory of Batyrev-Borisov mirror duality \cite{BB96}, as evidenced by the following theorem of Gross \cite{G05}.  We will state it in the hypersurface case, though its generalization is true for complete intersections.
\begin{theorem}
Let $\mathcal{X}\subseteq \mathbb{P}_\Delta$ be a suitable one-parameter family of  Calabi-Yau hypersurfaces with $\mathcal{X}_0$ the toric boundary of $\mathbb{P}_\Delta$.  Then:
\begin{itemize}
\item $\mathcal{X}\rightarrow \mathbb{A}^1$ is a toric degeneration, with general fiber being a Calabi-Yau hypersurface in $\mathbb{P}_{\tilde{\Delta}} $, where $\pi: \mathbb{P}_{\tilde{\Delta}}\ra  \mathbb{P}_{{\Delta}}$ is a partial crepant projective resolution.
\item There exists a so-called maximal partial crepant projective (MPCP) resolution $\tilde{\mathbb{P}}_\Delta \rightarrow \mathbb{P}_\Delta$ such that the affine manifold determined by the degeneration (see Section \ref{affine}) is \emph{simple} (well behaved in a certain sense; see Section 1.5 of \cite{GS06}).   
\end{itemize}
\end{theorem}

\subsubsection{Reconstruction Theorem}
Now that we've seen the applicability of toric degenerations, one may wonder if it is possible to reconstruct a degeneration given the information of the special fiber.  
Due to work of Gross and Siebert \cite{GS11}, it is possible to answer this in the affirmative.  
\begin{theorem}
\label{recthm}
Let $(\mathcal{X}_0, \mathcal{M}_{\mathcal{X}_0})$ be a locally rigid (a technical condition weaker than simplicity) log Calabi-Yau space. Then there exists a canonical toric degeneration $\mathcal{X}\rightarrow \spec\mathbb{C}\llbracket t\rrbracket$, and $t$ is a canonical coordinate \cite{RS14}.
\end{theorem}

\subsection{Reduction to the affine manifold} \label{affine}
Now let us see how to construct an affine manifold from the data of log Calabi-Yau space.  There are two methods, related, as the reader may suspect, by mirror symmetry.
In what follows, Let $\nu:\widetilde{\mathcal{X}}_0\ra \mathcal{X}_0$ be the normalization of $\mathcal{X}_0$, $\widetilde{\mathcal{X}}_0=\coprod X_i$ with $X_i$ toric, and the strata of $\mathcal{X}_0$ defined by
$$\Strata(\mathcal{X}_0):=\{\nu(S) | S \text{ is a toric strum of } X_i \text{ for some } i \}$$

\subsubsection{The dual intersection complex or ``fan picture"}
Suppose $(\mathcal{X}_0, \mathcal{M}_{\mathcal{X}_0})$ is a log Calabi-Yau space.  Note that each component $X_i$ of $\mathcal{X}_0$ is a toric variety $X_\tau$ with a corresponding fan $\Sigma_\tau$ in $M$.  This data is used to construct an affine structure near strata of codimension greater than one in $B$.  Topologically, these fans are then glued along the identification of toric strata given by $\nu$.  This construction falls short, however, of giving us an affine structure; there is no way of identifying the structure on one fan with another.

Applying Definition \ref{toricdegeneration} (4), for each $\{x\}\in \Strata(\mathcal{X}_0)$, there exists $M_x$, a Gorenstein monoid.  Note that $\check{M}_x = \operatorname{cone}(\Delta_x)\cap N$ for some $\Delta_x$, so, in particular, each zero dimensional toric stratum is associated to a lattice polytope.  These lattice polytopes allow us to interpolate between the affine structure of different fans, yielding an affine structure.  However, as is easy to imagine, the affine structures arising from these constructions may not be sufficiently compatible to allow us stitch the topological manifold into an uninterrupted affine manifold.  Rather, we must introduce singularities along a codimension two discriminant locus compatible with the polyhedral decomposition.  This can be done canonically by using a barycentric subdivision.

The result of this construction is an affine manifold with singularities $B$ along with a polyhedral decomposition $\mathcal{P}$.  We will call the pair $(B,\mathcal{P})$ a tropical manifold.

If $(\mathcal{X}_0, \mathcal{M}_{\mathcal{X}_0})$ is polarized by an ample line bundle $\mathcal{L}$, we can nicely encode this as additional data on our tropical manifold.  In particular, each $\mathcal{L}|_{X_i}$ is an ample line bundle, giving a piecewise linear function on the fan $\Sigma_i$.  Globally, we can glue these into a multi-valued (because of monodromy) piecewise linear function $\varphi$.  We call the triple $(B,\mathcal{P}, \varphi)$ a \emph{polarized tropical manifold}.

\subsubsection{The intersection complex or ``cone picture"}
If the data of the polarization seemed extraneous in the fan picture, it is essential in the following ``cone picture."  Again, along each component $X_i$, $\mathcal{L}|_{X_i}$ an ample line bundle on a projective toric variety, with a corresponding polytope $\sigma_i$.  We can glue these polytopes along the identifications given by $\nu$.  This gives us a topological manifold $\check{B}$ as well as a polyhedral decomposition $\check{\mathcal{P}}$.  Just as before, we need a fan structure at the vertices to define an affine manifold structure to the topological gluing. 
% Each vertex $v$ corresponds to a zero-dimensional stratum $\{x\}$, and thus a maximal cell $\sigma_{v}$ in the dual intersection complex.  We can take the fan structure at $v$ to be given by the normal fan $\check{\Sigma_v}}$ to $\sigma_v$.  Equivalently, 
Recall that, by the Gorenstein assumption, a monoid of the form $M_v=\{(m,a)\in \mathbb{Z}^n\oplus \mathbb{Z} | \varphi(m)\geq a\}$ is associated to each vertex $v$.  The domains of linearity of $\check{\varphi}_v$ define a fan $\Sigma_v$ in $N_\mathbb{R}$.  We can again glue (with singularities) using the polytope and fan structure, giving a polarized tropical affine manifold $(\check{B}, \check{\mathcal{P}}, \check{\varphi})$.
\begin{figure}
\centering
\[
\begin{tikzpicture}[scale=0.7]
\draw (0,0) -- (1,-1) -- (3,-1) -- (4,0) -- (2,2) -- (0,0);
\draw (3,3) -- (5,1) -- (7,3) -- (5,5) -- (3,3);
\draw (6,0) -- (8,0) -- (8,2) -- (6,0);
\draw [thick,->] (5,-0.3) -- (6,0.7);
\draw [thick,->] (5,-0.3) -- (6,-1.3);
\draw [thick,->] (5,-0.3) -- (4,0.7);
\draw [thick,->] (5,-0.3) -- (4,-1.3);
\draw [thick,->] (5,-0.3) -- (5,-1.7);
\draw [thick,->] (8.3,3) -- (7.3,2);
\draw [thick,->] (8.3,3) -- (7.3,4);
%\draw [thick,->] (8.3,3) -- (7.3,4);
\draw [thick,->] (8.3,3) -- (10.7,2);
\draw [thick,->] (8.3,3) -- (8.3,1.6);
\draw [thick,->] (5.3,6) -- (6.3,5);
\draw [thick,->] (5.3,6) -- (3.7,6);
\draw [thick,->] (5.3,6) -- (5.3,7.4);
\draw [thick,->] (2,2.7) -- (3,1.7);
%\draw [thick,->] (2,2.7) -- (3,3.7);
%\draw [thick,->] (2,2.7) -- (1,1.7);
\draw [thick,->] (2,2.7) -- (2,4.1);
\draw [thick,->] (2,2.7) -- (0.6,2.7);
\end{tikzpicture}
\]

\caption{Patching an affine manifold from fans and polyhedra. Mismatches lead to singularities in the affine structure.}
\label{patching}
\end{figure}
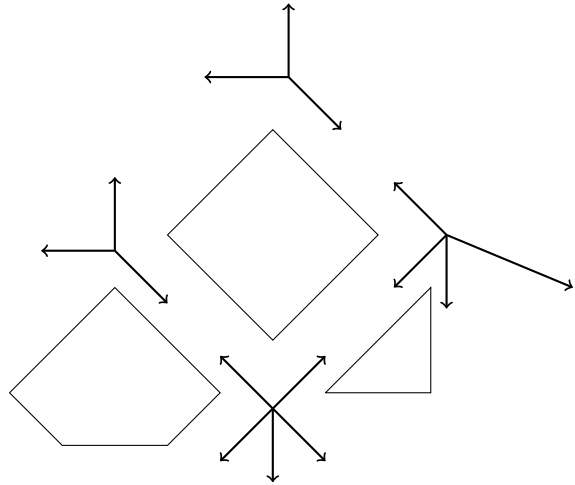

\begin{figure}
\centering
\[
\begin{tikzpicture}[scale=1.2]
\draw (0,0) -- (2,2) -- (3,1) -- (2,0) -- (1,1);
\draw (0,0) -- (0.5,-0.5) -- (1.5,-0.5) -- (2,0) -- (3,0) -- (3,1) -- (3.7,1);
%\draw (3,1) -- (3,1.7);
\draw (2,0) -- (2,-0.7);
%\draw (1,1) -- (1,1.7);
%\draw (1,1) -- (0.3,1);
\draw (2,2) -- (2,2.7);
\node at (1.5,1.5) {$+$};
\node at (2.5,0.5) {$+$};
\node at (1.5,0.5) {$+$};
\node at (2.5,1.5) {$+$};
\end{tikzpicture}
\]
\caption{An unavoidably misleading (flat paper provides an affine manifold without singularities!) representation of an affine manifold with singularities resulting from the identification in Figure \ref{patching}.}
\label{stitchedmfld}
\end{figure}
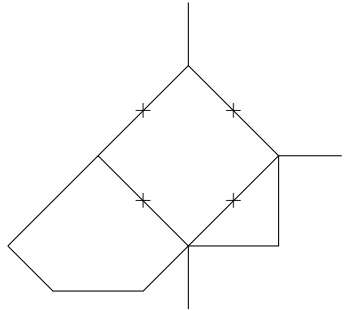

\subsubsection{The discrete Legendre transform}
The definitions above beg for an explicit connection.  The basic toric geometry correspondence between a polytope and a fan along with a piecewise linear function can be extended to a duality of polarized tropical manifolds taking $(B, \mathcal{P},\varphi)$ to $(\check{B}, \check{\mathcal{P}}, \check{\varphi})$ called the discrete Legendre transform.   This is the appropriate discretized version of the original relationship noticed by Hitchin between the complex and K\"ahler affine structures on the base of an SYZ fibration.  Significantly, we have the following result.
\begin{lemma} 
For a given log Calabi-Yau space, the discrete Legendre transform interchanges the dual intersection complex with the intersection complex.  
\end{lemma}
\subsection{Reconstruction of $\mathcal{X}_0$ from $(B,\mathcal{P}, \varphi )$}
As we've seen, Theorem \ref{recthm} shows that one can recover a toric degeneration from a log Calabi-Yau space.  Can one recover a log Calabi-Yau space from an affine manifold?  Consider the map 
$$\{\mathcal{X}_0, \mathcal{M}_{\mathcal{X}_0}\}\ra \{(B,\mathcal{P}, \varphi)\}$$ 
from the set of LCY spaces to the set of polarized tropical affine manifolds given by operation of taking the intersection complex.  

Recall each maximal cells $\sigma_i$ of an affine manifold, if interpreted as an intersection complex, represents a projective toric variety $\mathbb{P}_{\sigma_i}$.  As there is an 1-to-1 inclusion preserving correspondence between the toric strata of $\mathbb{P}_{\sigma_i}$ and the polyhedral strata of $\sigma_i$, it's clear that we should glue $\mathbb{P}_{\sigma_1}$ and  $\mathbb{P}_{\sigma_2}$ along $\mathbb{P}_\tau$ if $\tau=\sigma_1\cap\sigma_2$.  For each identification, there is a whole family of possible equivariant gluings.  These choices are called \emph{closed gluing data}.  With a choice $s$ of closed gluing data, one can recover a scheme $\check{\mathcal{X}}_0(B,\mathcal{P}, \varphi)$.

Not all choices of $s$ result in something that can be the central fiber of a toric degeneration, because the gluing must carry a correct log structure.  In order to guarantee the existence of such a log structure, we must consider closed gluing data that are induced by \emph{open gluing data}.  Each vertex $v$ of $\mathcal{P}$ comes with a monoid $P_v:=\{(m,r)\in \mathbb{Z}^n\times \mathbb{Z}| r\geq \varphi_v(m)\}$, where $\varphi_v$ is a local representative of $\varphi$.  Setting
\begin{eqnarray*}
&U(v):=\spec\mathbb{C}[P_v]\\
&V(v):=\spec \mathbb{C}[P_v]/(z^{(0,1)})
\end{eqnarray*}
we obtain a local model.  As shown by Gross and Siebert in \cite{GS03}, a necessary condition for $\check{\mathcal{X}}_0(B,\mathcal{P}, \varphi)$ to be the central fiber of a toric degeneration is that it can be expressed as an (equivariant) gluing of $V(v)$ along Zariski open subsets.  These gluing choices are called \emph{open gluing data}.  Each $V(v)$ come with a divisorial log structure $\mathcal{M}_v$ obtained from $V(v)\subseteq U(v)$, and the corresponding ghost sheaves $\overline{\mathcal{M}}_v=\mathcal{M}_v/\mathcal{M}^{\times}_v$ (see Section \ref{definitions}) are identified by the gluings.  This gives us a ghost sheaf of monoids on $\check{\mathcal{X}}_0(B,\mathcal{P}, \varphi)$.  

The following theorem is a main result of \cite{GS03} 
\begin{theorem}
Given $(B,\mathcal{P}, \varphi)$ simple, the set of log Calabi-Yau spaces with intersection complex $(B, \mathcal{P}, \varphi)$ modulo isomorphism preserving $B$ is $H^1(B, i_*\check{\Lambda} \otimes k ^ \times)$.  An isomorphism is said to preserve $B$ if it induces the identity on the intersection complex.
\end{theorem}

Therefore,  the fiber over a given manifold $(B, \mathcal{P}, \varphi)$ is identified with $H^1(B, i_*\check{\Lambda} \otimes k ^ \times)$, where $i:B\setminus  \Delta \hookrightarrow  B$, $\Delta$ is the discriminant locus of $B$, and $\Lambda$ is the family of lattices locally defined by the flat affine integral vector fields on $B\setminus \Delta$.  The element $0\in H^1(B, i_* \Lambda \otimes k^\times)$ corresponds to an untwisted gluing.  
Hence we have a bijection
\[
\xymatrix{
\left\{\left( \mathcal{X}_0, \mathcal{M}_{\mathcal{X}_0} \right)\right\} \ar@{<->}[rr]^-{1:1}  & & \left\{\left(\left(B,\mathcal{P} \right) , s \right) \, | \; s\in H^1(B, i_* \Lambda \otimes k^\times) \right\} \\ 
\left\{ \text{ polarized } \right\} \ar[u] \ar@{<->}[rr] & & \ar[u]_{\text{forgetful map}} \left\{\left(\left(B,\mathcal{P}, \varphi \right) , s \right) \, | \; s\in H^1(B, i_* \Lambda \otimes k^\times)  \right\}.
}
\]
\subsection{Mirror symmetry via the Gross-Siebert program}

With these results in place, we can discuss an overall strategy of using these techniques to understand mirror symmetry.  One begins with a polarized  toric degeneration $\mathcal{X}\ra S$, which can be distilled to a LCY space.  By taking the dual intersection complex, we further reduce to a polarized tropical affine manifold $(B, \mathcal{P}, \varphi)$.  From here, we wish to apply the reconstruction theorem to construct a degeneration $\check{\mathcal{X}}\ra \spec k\llbracket t\rrbracket$ whose \emph{intersection complex} is $(B, \mathcal{P}, \varphi)$.  This degeneration should be dual (in the mirror sense) to the one we started with.  The idea can be summed up in the following diagram.
\[
\begin{tikzpicture}[scale=1]
\node at (0,10) {$\mathcal{X}$} ;
\node at (0,9) {$S$} ;
\draw [->] (0,9.8) -- (0,9.2);
\node[align=center,right] at (0.4,9.5) {Polarized\\ toric CY\\degeneration};
\draw [<->,dashed] (2.5,9.5) -- (5.5,9.5) ;
\node [above] at (4,9.5) {Mirror};
\node [below] at (4,9.5) {symmetry};
\node at (6.3,10) {$\mathcal{X}$} ;
\node at (6.3,9) {$\spec k\llbracket t\rrbracket$} ;
\draw [->] (6.3,9.8) -- (6.3,9.2);
\node[align=center,right] at (7,9.5) {Polarized\\ toric CY\\degeneration};
\draw [->,snake=snake,segment amplitude=.6mm, segment length = 4mm, line after snake = 0.4mm] (0.5,8.5) -- (0.5,7);
\node at (0.4,6.5) {$\left(\mathcal{X}_0,\mc{M}_{\mathcal{X}_0}\right)$} ;
\node[align=center,right] at (0.8,7.6) {Polarized\\ toric log CY\\ space};
\draw [->,snake=snake,segment amplitude=.6mm, segment length = 4mm, line after snake = 0.4mm] (6.3,7) -- (6.3,8.5);
\node at (6.3,6.5) {$\left(\check{\mathcal{X}_0},\mc{M}_{\check{\mathcal{X}_0}}\right)$} ;
\node[align=center,right] at (6.5,7.75) {Reconstruction \\ thm};
\draw [->,snake=snake,segment amplitude=.6mm, segment length = 4mm, line after snake = 0.4mm] (0.5,5.7) -- (0.5,3.7);
\node[align=center,left] at (0.3,4.7) {Fan};
\draw [->,snake=snake,segment amplitude=.6mm, segment length = 4mm, line after snake = 0.4mm] (6.3,3.7) -- (6.3,5.7);
\node[align=center,right] at (6.5,4.7) {Fan};
\node at (0.4,3) {$\left(\left(B,\mathcal{P}, \varphi \right) , s \right)$};
\node at (6.3,3) {$\left(\left(\check{B},\check{\mathcal{P}}, \check{\varphi} \right) , \check{s} \right)$};
\draw [<->] (1.8,3) -- (4.9,3) ;
\node [above] at (3.35,3) {Discrete Legendre transform};
\draw [->,snake=snake,segment amplitude=.4mm, segment length = 4mm, line after snake = 0.4mm] (1.3,5.9) -- (5.8,3.6);
\draw [->,snake=snake,segment amplitude=.4mm, segment length = 4mm, line after snake = 0.4mm] (5.8,5.9) -- (1.3,3.6);
\node [align=center] at (1.8,5.3) {Cone};
\node [align=center] at (5.2,5.3) {Cone};
\node [align=center] at (4.2,2) {Pick this or\\work universally in $\check{s}$};
\draw [->] (5.7,2.1) to [out=-10,in=-90] (7.1,2.75);
\end{tikzpicture}
\]

The basic idea of mirror symmetry is to identify pairs of manifolds (or degenerations) for which the symplectic structure of one is closely related to the complex structure of the other.  Much of the early excitement over mirror symmetry resulted from the identification of certain enumerative invariants on one manifold with the results of period integrals on another.  One of the nice features of the above construction is that there is a combinatorial structure, the underlying affine manifold, which controls the symplectic structure of $\mathcal{X}$ and the complex structure of $\check{\mathcal{X}}$.  The natural geometry on tropical affine manifolds is \emph{tropical geometry}, which leads one to hope that mirror symmetry can be well described by identifying tropical structures that describe both the symplectic structure of $\mathcal{X}$ and the complex structure of $\check{\mathcal{X}}$. 

\subsubsection{Tropical data in the dual intersection complex}
The utility of tropical curves for the computation of Gromov-Witten invariants has been known for some time.  Please see Section \ref{section:NS} for more information about how these techniques fit into the overall structure of toric degenerations.  In keeping with the overall philosophy of the program, the goal is to develop the machinery to compute Gromov-Witten invariants of the general fiber from the combinatorial data of the central fiber. The current interpretation relies on something called the \emph{tropicalization functor} that uses log structures to produce polyhedral complexes; in particular, it recovers the dual intersection complex.  As Gross and Siebert have shown, it is possible to construct a nice moduli space of so-called \emph{log stable maps} for well behaved log spaces.  These techniques are not yet applicable to the general toric degeneration framework, as the log structure of the  central fiber fails to satisfy the requirements of the theorem on the points $Z$.  Nevertheless, the image of a log stable map under the tropicalization functor should be a tropical curve in the dual intersection complex, giving some motivation for the hope that curve counting can entirely be done on the combinatorics of the affine manifold.
\subsubsection{Tropical data in the intersection complex}
The tropical data relevant to the complex structure of a manifold reconstructed from an intersection complex are given by the rays of a scattering diagram.  In order to understand how this arises, we need to discuss the specifics of the reconstruction theorem.  In the absence of singularities in the affine manifold, the reconstruction process constructs the well-known Mumford degeneration.  Specifically, suppose that $\check{B}$ is a polytope $\Delta\subseteq \mathbb{R}^n$ and $\check{\mathcal{P}}$ is a polyhedral decomposition of $\Delta$ induced by the bending locus of a piecewise linear function $\check{\varphi}$.  Consider
$$\hat{\Delta}:=\{(m,a)\in \mathbb{R}^n\oplus \mathbb{R} | \check{\varphi}(m)\geq a\}.$$
Setting $\mathcal{X}:=\Proj k[\text{cone}(\hat{\Delta})\cap \mathbb{Z}^{n+2}]=\mathbb{P}_{\hat{\Delta}}$, we see that setting $t:=z^{(0,\ldots, 0,1,0)}$ gives us a degeneration $\mathcal{X}\ra k[t]$ which is a reconstruction of $(X_0, \mathcal{M}_{\mathcal{X}_0})$ (the LCY space achieved by a choice of ``vanilla" gluing data).  As you can see, this is just a gluing of the local models introduced in the discussion of the open gluing data.  The introduction of singularities, however, creates a great deal of complication.  
The effort to create a reconstruction process began with the work of Fukaya in \cite{F05}, who noted that perturbations of the complex structure (in dimension 2) should be concentrated along trees of gradient flow lines emanating from singular points of the affine manifold.  Kontsevich and Soibelman further studied the two-dimensional case in \cite{KS06}, showing that a tropical affine surface with 24 focus-focus singularities can be used to construct a rigid analytic K3 surface.  The key insight here was the use of  gluing automorphisms attached to gradient flow lines, giving a ``scattering diagram".  Gross and Siebert studied the problem using the dual affine structure in \cite{GS11}, where the gradient flows become straight lines.  The local models are then glued using the automorphisms carried by this \emph{scattering diagram}.  This allowed a difficult extension to the higher dimensional case, yielding the theorem referenced above.

In the case of dimension two, the gluing automorphisms propagate along straight lines, and these straight lines collide and glue to form structures reminiscent of so-called \emph{tropical disks} (see Section \ref{sec:tropics}).  Our guiding hope is that these are tropical manifestations of holomorphic disks.  Nishinou has shown that such a correspondence does indeed exist \cite{Nishinou}.  Furthermore, as Auroux has explained in \cite{A07}, one expects the complex structure on one side of the mirror to be controlled by holomorphic disks on the other side, lending further credence to this idea. 
\begin{figure}
\centering
\[
\begin{tikzpicture}[scale=1.7]
\draw (-0.7,0.7) -- (0,0) -- (2,0) -- (2.7,0.7);
\draw (0,0) -- (0,-1);
\draw (2,0) -- (2,-1);
\node at (1,0) {$\times$};
\draw [dashed] (1,0) -- (1,0.9);
\draw [dashed] (1,0) -- (1.85,0.85);
\draw [dashed] (1,0.2) -- (1.2,0.2);
\draw [dashed] (1,0.4) -- (1.4,0.4);
\draw [dashed] (1,0.6) -- (1.6,0.6);
\draw [dashed] (1,0.8) -- (1.8,0.8);
\node [right] at (0,-0.2) {$x$};
\draw [->] (0,0) -- (0,-0.4);
\node [above] at (0.2,0) {$w$};
\draw [->] (0,0) -- (0.4,0);
\node [below] at (-0.2,0.15) {$y$};
\draw [->] (0,0) -- (-0.28,0.28);
\node [left] at (2,-0.2) {$u$};
\draw [->] (2,0) -- (2,-0.4);
\node [below] at (2.2,0.15) {$v$};
\draw [->] (2,0) -- (1.6,0);
\node [above] at (1.85,0) {$w^{-1}$};
\draw [->] (2,0) -- (2.28,0.28);
\draw [->,thick] (0.8,0.66) to [out=45,in=135] (1.9,0.66);
\node at (1.3,1.25) {$y\mapsto w^{-1}v$};
\node at (1.3,1.05) {$x\mapsto w^{-1}u$};
\draw [->,thick] (0.6,-0.4) to [out=-45,in=-135] (1.4,-0.4);
\node at (1,-0.7) {$x\mapsto u$};
\node at (1.1,-0.9) {$y\mapsto v\,w^{-2}$};
\draw (0,0) circle [radius=0.5];
\draw (2,0) circle [radius=0.5];
\node at (-1.2,-0.8) {$xyw=t$};
\draw (-1.2,-0.8) ellipse (14pt and 10pt);
\draw [->] (-1.2,-0.35) to [out=75,in=180] (-0.6,0.1);
\node at (3.2,-0.8) {$uvw^{-1}=t$};
\draw (3.2,-0.8) ellipse (17pt and 12pt);
\draw [->] (3.2,-0.35) to [out=105,in=0] (2.6,0.1);
\end{tikzpicture}
\]
\caption{Monodromy introduces an ambiguity in the identification of local models near a singularity.  This difficulty is resolved by introducing gluing automorphisms along walls that are invariant under the monodromy induced by the singularity.  See \cite{GS11}.}
\label{gluing}
\end{figure}
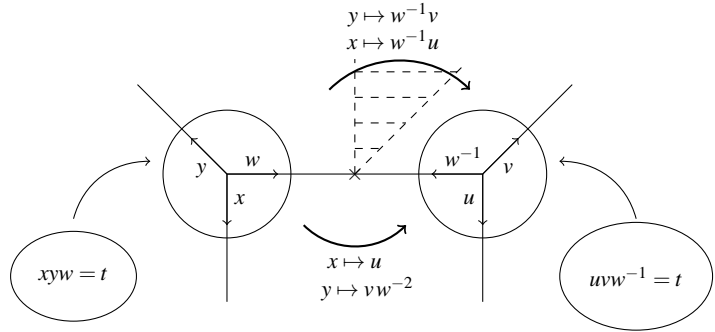
\subsection{Structure}
Having established a sketch of the main ideas of the Gross-Siebert program, we go on to explore some of major tools used in its study.  In Section \ref{sec:loggeom} we give an introduction to logarithmic geometry, an extremely important tool for the study of degenerating families.  Next, we introduce tropical geometry in Section \ref{sec:tropics}.  The application of tropical geometry to enumerative questions is introduced in Section \ref{sec:enum}, utilizing logarithmic techniques.  Finally,  these enumerative results are connected with certain period calculations on a Landau-Ginzburg model of $\mathbb{P}^2$ in a sketch of Gross's construction mirror symmetry.  This connection is achieved through an identification of tropical structures common to both the Landau-Ginzburg model and the tropical enumerative calculations.

\subsection{Acknowledgements}
We are indebted to the referee for a careful reading of the text and to Mark Gross for pointing out the application of the Welschinger invariant appearing in Section \ref{sec:enum}.  These notes are sprung from the Fields Institute's thematic program on Calabi-Yau varieties.  The authors would like to thank the Fields Institute for providing an excellent working and learning environment, and the program's organizers for their hard work and guidance, most of all Noriko Yui.

\section{Introduction to Logarithmic Geometry}
\label{sec:loggeom}

\abstract*{
}

\subsection{Introduction}

The first goal of this chapter is to familiarize the reader with log structures and to overview some basic properties of these. The second, more specific goal is to introduce the reader to notions used in other sections of this chapter. Namely, this includes the definition of log smoothness in section \ref{propr}, as well as the definition of torically transverse log curves in section \ref{tor-trans}. The third goal is to offer the reader an introduction to logarithmic Gromov-Witten theory. In order to do so, F. Kato's \cite{kato-log-curves} local description of log smooth curves is illustrated in Section \ref{log-sm-curves}. This is then used in Section \ref{log-gw} to sketch the starting point for logarithmic Gromov-Witten theory. In particular, we will describe why log smooth maps are a natural (and powerful!) candidate to generalize (relative) stable maps.

Log geometry was introduced by Illusie and Fontaine, see \cite{Illusie_log_spaces}, and by K. Kato, see \cite{Kato_log_struct}. Adding a log structure to certain singular schemes allows them to be treated as if they were smooth. The focus is on examples that illustrate this concept. The examples are taken from the book \cite{kan} by Gross. The interested reader is invited to consult that reference for a more thorough treatment of log geometry, as well as for more examples.

\subsection{Motivation}
Log structures are a vast abstraction of log differentials. Thus, to motivate log structures, we start by reviewing log differential. Let $X$ be a smooth quasi-projective variety contained in a projective variety $\overline{X}$. Denote by $i : X\hookrightarrow \overline{X}$ the inclusion and assume that the divisor $D=\overline{X}\backslash X$ is normal crossings. By definition, for a point $z\in D$ there is an affine open neighbourhood $U$ of $z$ in $\overline{X}$, and coordinates $x_1,\dots,x_n$ on $U$ such that $D\cap U$ is given by
$$
x_1\cdots x_p = 0,
$$
for some $p\leq n$.

\begin{definition} The \emph{sheaf of log differentials} $\Omega_{\overline{X}}^q\left( \log D \right)$ is a sheaf on $\overline{X}$, defined locally as a subsheaf
$$
\Omega_{\overline{X}}^q\left( \log D \right) \subseteq i_* \Omega_X^q,
$$
as follows. Assume $U\subset \overline{X}$ is affine open and has coordinates $x_1,\dots,x_p$ such that $D\cap U$ is given by $x_1 \cdots x_p = 0$ for some $p\leq n$. Define $\Omega_{\overline{X}}^q\left( \log D \right) \left( U \right)$ to be generated by
$$
\frac{\diff x_1}{x_1}, \cdots, \frac{\diff x_p}{x_p}, \diff x_{p+1}, \cdots, \diff x_n.
$$
\end{definition}
The sheaf of log differentials recovers for $X$ a number of properties that hold for projective varieties. For example, its hypercohomology calculates the cohomology of $X$:
$$
\mathbb{H}^q\left(\overline{X},\Omega_{\overline{X}}^\bullet \left( \log D \right) \right) \cong \hhh^q(X,\IC).
$$
Taking this as a starting point, Deligne developed his theory of mixed Hodge structures, which provides analogous results for $X$ as the Hodge structure does for $\overline{X}$. In mirror symmetry, this analogy is carried over to Yukawa couplings. Indeed, via variation of mixed Hodge structures, Konishi-Minabe in \cite{local_B_model} define the local $B$-model Yukawa coupling in the setting of local Calabi-Yau threefolds. Their result mirrors the properties of the Yukawa coupling for the compact Calabi-Yau threefold case. These examples show that the sheaf of log differentials extends results that are true for projective varieties to quasi-projective ones.

We proceed to consider the relative version of the sheaf of log differentials in a family. It illustrates how using the sheaf of log differentials recovers results that hold true for smooth varieties to singular ones. It is part of Steenbrink's construction of the limiting mixed Hodge structure, see \cite{lim_hodge_str}, for a normal crossings degeneration.

Consider a normal crossings degeneration. This consists of a one-dimensional flat family
$$
f : X \to S,
$$
such that $S$ is smooth and such that the fibers $X_s$ are smooth except for a closed point $0\in S$. Moreover, $f$ is assumed to be normal crossings. That means the following: For every $z\in X$, there is $U \ni z$ an affine open neighbourhood with coordinates $x_1,\dots,x_n$; there is an affine open neighbourhood $V$ of $S$ with coordinate $s$; $U$ and $V$ are such that $f|_U$ maps to $V$ and is given by
$$
(x_1,\dots,x_n) \mapsto s=x_1\cdots x_p,
$$
for some $p\leq n$. Define the \emph{sheaf of relative log $q$-forms} as the quotient
$$
\Omega^q_{X/S} \left( \log X_0 \right) := \Omega^q_{X} \left( \log X_0 \right) / \mc{F},
$$
where
$$
\mc{F}= f^*\Omega^1_{S} \left( \log 0 \right) \wedge \Omega^{q-1}_{X} \left( \log X_0 \right).
$$
Then $\Omega^q_{X/S} \left( \log X_0 \right)$ is a sheaf on $X$. To illustrate how it differs from $\Omega^q_{X} \left( \log X_0 \right)$, consider log $1$-forms. Since $f$ is normal crossings, in an affine open neighbourhood $U \subset X$ of $z\in X_0$ and in suitable coordinates, $X_0\cap U$ is given by $x_1\dots x_p=0$. Thus, as above, $\Omega^1_{X} \left( \log X_0 \right) \left( U \right)$ is generated by
$$
\frac{\diff x_1}{x_1}, \cdots, \frac{\diff x_p}{x_p}, \diff x_{p+1}, \cdots, \diff x_n.
$$
By definition $\Omega^1_{X/S} \left( \log X_0 \right)$ has the same set of generators. Pulling back the $1$-form $\diff s / s$ yields the additional relation
$$
\frac{\diff x_1}{x_1} + \cdots + \frac{\diff x_p}{x_p}=0.
$$
Consider the sheaf on $X_0$ obtained by the restriction to $X_0$,
$$
\Omega^q_{X_0^\dagger/S^\dagger}:=\Omega^q_{X/S} \left( \log X_0 \right)|_{X_0}.
$$
The sheaf $\Omega^q_{X_0^\dagger/S^\dagger}$ exhibits a lot of properties that would hold for $\Omega^q_{X_0}$ in case $X_0$ was smooth. For instance, $\Omega^q_{X_0^\dagger/S^\dagger}$ is locally free and the exterior derivative makes sense on $\Omega^q_{X_0^\dagger/S^\dagger}$. Moreover, it is shown in \cite{lim_hodge_str} that for $f$ proper and log smooth (see section \ref{def_log_sm} below), the higher direct image
$$
R^p f_* \Omega^q_{X/S} \left( \log X_0 \right)
$$
is locally free and furthermore imitates some of the properties that $R^p f_* \Omega^q_{X/S}$ enjoys in the smooth case. Namely, away from $X_0$, $R^p f_* \Omega^q_{X/S} \left( \log X_0 \right)$ is the sheaf of $q$-forms and so its fibers are the Dolbeault cohomology groups
$$
\hhh^p(X_s,\Omega^q_{X_s}),
$$
whenever $s\neq 0$. And its fiber at $0$ is
$$
\hhh^p( X_0,\Omega^q_{X_0^\dagger/S^\dagger}).
$$
Finally, these cohomology groups are used by Steenbrink in \cite{lim_hodge_str} to define the limiting mixed Hodge structure associated to this degeneration.

We hope that this last example convinces the reader that using the sheaf of relative log differentials allows to treat the central fiber $X_0$ as if it was smooth. Log structures, though more abstract, are a vast generalization of this idea. They have the advantage that they can be considered over any scheme. The notion of \emph{log smoothness}, see definition \ref{def_log_sm} below, applies much more generally than smoothness does, and exhibits many of the same properties than smoothness does.

\subsubsection{The \'etale topology}
\label{etale}

In order to talk about log structures, the Zariski topology is too coarse in general. Instead, we need to consider sheaves in the \'etale topology. We briefly overview what it means for a sheaf to be defined in the \'etale topology. We refer the interested reader for a more thorough treatment of the topic to the book \cite{milne_et_book} by Milne.

Let $X$ and $Y$ be schemes. Recall that a flat morphism of finite type\footnote{If we strove for maximal generality, we would assume $\pi$ to be flat and locally finitely presented.}
$$
\pi: X\to Y,
$$
is \'etale if and only if for any $q\in Y$, its preimage is written as a disjoint union
$$
\pi^{-1}(q) = \sqcup_i \spec K_i,
$$
where the $K_i$ are finite separable extensions of the residue field $k(q)$.

The \'etale topology adds more open subschemes to the Zariski topology. It is not a topology in the classical sense, but it exhibits the same properties. We do not provide a thorough overview of it, but rather describe what sheaves are in the \'etale topology and how they are used. Let $X$ a scheme. Open neighbourhoods in the \'etale topology are defined as \'etale morphisms
$$
U\to X.
$$
Let $\mc{F}$ be a sheaf of sets (or of groups or of any other algebraic structure) in the \'etale topology. Then $\mc{F}$ associates a set (or group etc.) $\mc{F}(U)$ to each \'etale map $U\to X$. Moreover, to each diagram of  \'etale maps
$$
\xymatrix{
U \ar[rr]^\phi \ar[dr] & & V \ar[dl] \\
& X, &
}
$$
$\mc{F}$ associates a restriction map of sets (or of groups etc.)
$$
\mc{F}(\phi) : \mc{F}(V) \to \mc{F}(U).
$$
These restriction maps are required to satisfy the usual sheaf axioms.

We review the definition of stalks in the context of the \'etale topology. Let $\overline{x} \to X$ be a geometric point. By definition, $\overline{x}=\spec(k)$, where $k$ is algebraically closed. Thus, choosing a geometric point amounts to choosing a point $x\in X$ and an inclusion $k(x) \subseteq k$ from the residue field $k(x)$ of $x$ to an algebraically closed field $k$. The stalk of $\mc{F}$ at $\overline{x}$ is defined as the direct limit
$$
\mc{F}_{\overline{x}} := \varinjlim \mc{F}(U),
$$
where the limit is taken over diagrams
$$
\xymatrix{
\overline{x} \ar[rd] \ar[r] & (U,u) \ar[d] \\
 & (X,x),
}
$$
for $(U,u)\to (X,x)$ pointed \'etale maps.

Throughout this section, we consider the schemes to be endowed with the \'etale topology, and the sheaves and stalks to be defined as above. For example, when we consider stalks of sheaves, we will always choose a geometric point.

\subsubsection{Basic definitions}
\label{definitions}
In this section, we introduce the terminology that is needed for the definition of log smoothness (Definition \ref{def_log_sm}). We are mainly concerned with sheaves of monoids, with the monoid operation usually given by multiplication, the notable exemption concerning the ghost sheaves. Let $X$ be a scheme and consider the sheaf of monoids $\mc{O}_X$ with the monoid structure given by multiplication. A \emph{pre-log structure} on $X$ consists of a sheaf of monoids $\mc{M}_X$ on $X$, in addition to a homomorphism of sheaves of monoids
$$
\alpha_X : \mc{M}_X \to \mc{O}_X.
$$
Then $\mc{M}_X$ is a \emph{log structure} if in addition the restriction
$$
\alpha_X|_{\alpha^{-1}_X\left( \mc{O}_X^{\times} \right)} : \alpha^{-1}_X\left( \mc{O}_X^{\times} \right) \to \mc{O}^\times_X
$$
is an isomorphism. Throughout this section, we use the notation $\mc{M}_X$ to denote a log structure on $X$. We write $X^\dagger = (X, \mc{M}_X)$ to indicate that the log structure is implicitly understood.

A morphism
\begin{align*}
f : X^\dagger & \to Y^\dagger
\end{align*}
of log structures consists of a morphism of the underlying schemes
\begin{align*}
f : X & \to Y,
\end{align*}
and a morphism of sheaves of monoids 
$$
f^{\#} : f^{-1}\mc{M}_Y \to \mc{M}_X,
$$
such that the diagram
\begin{equation}
\label{commdiag}
\xymatrix{
f^{-1} \mc{M}_Y \ar[d]_{f^{-1}\alpha_Y} \ar[r]^-{f^{\#}} & \mc{M}_X \ar[d]^{\alpha_X} \\
f^{-1} \mc{O}_Y \ar[r]^-{f^*} & \mc{O}_X
}
\end{equation}
commutes.

The \emph{ghost sheaf} $\overline{\mc{M}_X}$ is defined as the cokernel of $\alpha_X^{-1}$ restricted to $\mc{O}^\times_X$, yielding a short exact sequence
$$
1 \to \mc{O}_{X}^\times \xrightarrow{\alpha_X^{-1}} \mc{M}_X \to \overline{\mc{M}_X} \to 0.
$$
Note that the ghost sheaf is written \emph{additively}. As we will see in the examples of the next section, for the most important example of a log structure (the divisorial log structure), the ghost sheaf records the order of vanishing of regular functions. Since the order of vanishing of the product of two functions is the sum of the individual orders, this justifies the additive notation.

Assume that we have a map of log schemes $f : X^\dagger\to Y^\dagger$. Since the inverse image functor $f^{-1}$ is exact, $f^{-1}\overline{\mc{M}_Y}$ is the sheaf cokernel of
$$
f^{-1}\mc{O}^\times_Y \to f^{-1}\mc{M}_Y.
$$
Since \eqref{commdiag} commutes, $f^\#$ induces a map on the ghost sheaves
$$
\overline{f^\#} : f^{-1}\overline{\mc{M}_Y} \to \overline{\mc{M}_X}.
$$
For simplicity, we write $f^\# = \overline{f^\#}$ as well.

Let $\alpha : P_X \to \mc{O}_X$ be a pre-log structure on $X$. The \emph{log structure associated to} $P_X$ is the sheaf of monoids
$$
\mc{M}_X := \frac{P_X\oplus\mc{O}^\times_X}{\left\{ \left( p,\alpha(p)^{-1} \right) \; : \; p\in\alpha^{-1}\left( \mc{O}_X^\times \right) \right\}},
$$
in addition to the morphism of sheaves of monoids $\alpha_X : \mc{M}_X \to \mc{O}_X$ defined via
$$
\alpha_X(p,f) := \alpha(p)\cdot f.
$$
We show that this yields a log structure. Note that the map $\alpha_X$ is well-defined. Indeed, if $\left(p,\alpha(p)^{-1}\right)\in P_X\oplus \mc{O}_X^\times$ is such that $p\in \alpha^{-1}\left( \mc{O}_X^\times \right)$, then
$$
\alpha_X(p,\alpha(p)^{-1})= \alpha(p) \cdot \alpha(p)^{-1} = 1.
$$
We need to prove that the restriction of $\alpha_X$ to
$$
\alpha_X^{-1} \left( \mc{O}_X^\times \right) \to \mc{O}_X^\times
$$
yields an isomorphism. This map is surjective since if $f\in\mc{O}_X^\times$, then $\alpha_X(1,f)=f$. To show that it is injective, assume that $\alpha_X(p,f)=1$. Then $\alpha(p)\cdot f =1$, $f = \alpha(p)^{-1}$ and hence $(p,f)=(p,\alpha(p)^{-1})=1$.

Let $f : X \to Y$ be a morphism of schemes and assume that $Y$ is endowed with a log structure $\alpha_Y : \mc{M}_Y \to \mc{O}_Y$. The \emph{pull-back log structure} on $X$, denoted by $f^*\mc{M}_Y$, is the log structure associated to the pre-log structure defined by the composition
$$
f^{-1}\left( \mc{M}_Y \right) \xrightarrow{\alpha_Y} \alpha_Y^{-1} \left( \mc{O}_Y \right) \xrightarrow{f^*} \mc{O}_X.
$$
The pullback commutes with the ghost sheaf, in the sense that
$$
\overline{f^*\mc{M}_Y} = f^{-1}\overline{\mc{M}_Y}.
$$
For a proof of this statement, see \cite{kan}.

\subsection{Examples}
\label{examples}

Unless specified otherwise, the monoids below are written multiplicatively. The exception is for the monoid $\IN$, which is endowed with the operation of addition and which we assumed to contain $0$.

\begin{example}
The \emph{trivial log structure} on a scheme $X$ consists of the invertible functions: $\mc{M}_X = \mc{O}_X^\times$.
\end{example}

\begin{example} \label{stand-log-point}
Let $k$ denote a field. The \emph{standard log point} over $k$ is defined as
$$
\spec k^\dagger = \left( \spec k, \mc{M} = k^\times\oplus\IN  \right),
$$
where $\alpha : k^\times \oplus \IN \to k$ sends
$$
(y,n) \mapsto \left\{ \begin{array}{l l}
y & \quad \text{if } n=0, \\
0 & \quad \text{if } n\neq 0.
\end{array}
\right.
$$
Note that $\alpha^{-1}(k^\times)=k^\times \oplus \left\{ 0 \right\}$, hence $\overline{\mc{M}}=\IN$. In terms of the ghost sheaf, we can thus think of the standard log point to consist of a copy of $\IN$ on top of $\spec k$.
\end{example}

\begin{example}
\label{ex-div-log-str}
Next, we introduce the most important log structure, the \emph{divisorial log structure}. Let $X$ be a scheme and let $D \subset X$ be a closed subset of pure codimension $1$. Denote moreover by $j : X\backslash D \hookrightarrow X$ the inclusion. Then the \emph{divisorial log structure induced by} $D$ is the log structure $\mc{M}_{(X,D)}$ on $X$ defined by considering regular functions which are invertible away from $D$,
$$
\mc{M}_{(X,D)} := \left( j_* \mc{O}^\times_{X\backslash D} \right)\cap\mc{O}_X,
$$
and by taking
$$
\alpha_X : \mc{M}_{(X,D)} \hookrightarrow \mc{O}_X
$$
to be the inclusion.
\end{example}

\begin{example}
As a first example of divisorial log structure, consider the pair $(X,D)=(\mathbb{A}^1_k,\left\{ 0 \right\})$ and $\mc{M}= \mc{M}_{(X,D)}$. We show that the restriction of $\mc{M}$ to $\orig$ yields the standard log point, i.e. that the pull-back log structure $j^*\mc{M}$ is $\IN\oplus k^\times$. As above, consider the inclusion (of schemes)
$$
j : \left\{ 0 \right\} = \spec k \hookrightarrow \aaa.
$$
Consider the restriction (pullback via $j$) of $\mc{M}$ to $\orig$. $\mc{M}$ is the sheaf of regular functions on $\aaa$ that are invertible away from $\orig$. Moreover, $j^{-1}(\mc{M})$, its stalk at the origin, is the germ of functions on $\aaa$ that are invertible away from $\orig$. In other words,
$$
j^{-1}(\mc{M})=\left\{ \phi \cdot x^n \; | \; n \in \IN, \; \phi\in\mc{O}(U)^\times, \; U \text{ \'etale neighborhood of } \orig \right\}.
$$
 Furthermore, $\alpha_X^{-1}\left( \mc{O}_X \right)$ is the sheaf of invertible regular functions on $\aaa$, and the map
$$
j^* : \alpha_X^{-1} \left( \mc{O}_X \right) \to \mc{O}_{\orig}
$$
is the evaluation map. Putting this together, the composition
$$
\alpha : j^{-1}(\mc{M}) \xrightarrow{\alpha_X} \alpha_X^{-1}\left( \mc{O}_X \right) \xrightarrow{j^*} \mc{O}_{\orig}
$$
is the evaluation map and sends
$$
\phi \cdot x^n \mapsto \left\{ \begin{array}{l l}
\phi(0)\neq 0 & \quad \text{if } n=0, \\
0 & \quad \text{if } n \geq 1.
\end{array}
\right.
$$
We now take the log structure associated to $\alpha$. The set
$$
\left\{ \left(\phi \cdot x^n,\alpha(\phi \cdot x^n)^{-1}\right) \; : \; \phi \cdot x^n \in \alpha^{-1}\left( \mc{O}_X^\times \right) \right\}
$$
consists of the elements of the form $(\phi,\phi(0)^{-1})$. Therefore, the associated log structure is given by
$$
\mc{M}_{\orig} := j^*\mc{M}= \frac{\left\{\phi\cdot x^n \right\}\oplus k^\times}{\left\{ \left( \phi,\phi(0)^{-1} \right) \right\}} = \IN\oplus k^\times;
$$
\begin{align*}
\alpha_{\orig} :  \IN\oplus k^\times & \to \mc{O}_{\orig}; \\
(x^n,y) \; & \mapsto  \left\{ \begin{array}{l l}
y & \quad \text{if } n=0, \\
0 & \quad \text{if } n \geq 1.
\end{array}
\right.
\end{align*}
This indeed is the standard log point.

Continuing on the above example, there is only one map of schemes
$$
j : \orig \to \aaa.
$$
In terms of log schemes schemes however, and taking the same log structures as above, there are many maps
$$
\orig^\dagger \to \left(\aaa\right)^\dagger.
$$
Indeed, such a map corresponds to a choice of morphism between sheaves of monoids
$$
j^\# : j^{-1}\mc{M} \to \mc{M}_{\orig},
$$
making the diagram
$$
\xymatrix{
\left\{ \phi \cdot x^n \right\} = j^{-1} \mc{M} \ar[d] \ar[r]^-{j^{\#}} & \mc{M}_{\orig} = \IN \oplus k^\times \ar[d]\\
\left\{ \phi \right\} = j^{-1} \mc{O}_{\aaa} \ar[r]^-{j^*} & \mc{O}_{\orig} = k^\times
}
$$
commute. It follows that $j^{\#}$ is determined by two choices of morphisms of monoids
\begin{align}
\label{mapofmonoids}
\IN & \to \IN,\\
\IN & \to k^\times.
\end{align}
A geometric way of seeing this map is at the level of the ghost sheaf. The stalks of the ghost sheaf $\overline{\mc{M}}$ are trivial away from the origin, while its stalk at the origin is $\IN$. The ghost sheaf of $\mc{M}_{\orig}$ on the other hand is $\IN$. The map \eqref{mapofmonoids} is the map induced on ghost sheaves by $j$:
$$
\overline{j^\#} : \IN = j^{-1}\overline{\mc{M}} \to \overline{\mc{M}_{\orig}} = \IN.
$$
Choosing as map of monoids the identity map implies that the log structure $\mc{M}_{\orig}$ is induced by $\mc{M}$ via $j$.

The choice of the map  \eqref{mapofmonoids} is extra information that is not seen at the level of schemes. This data however carries geometric information as we will see in the examples below.
\end{example}

\begin{example}
Next, we consider the affine plane $\mathbb{A}^2=\spec k[x,y]$ with the divisorial log structure induced by the union of the coordinate axes $D = \left\{ xy=0 \right\}$. For simplicity, we again denote this log structure by $\mc{M}$. $\mc{M}$ is the sheaf consisting of regular functions on $\aaaa$ that are invertible away from the coordinate axes. Denote again by $j : D \hookrightarrow \aaaa$ the inclusion. Denote moreover by $D_1$ the $x$-axis and by $D_2$ the $y$-axis.

To illustrate what information is carried by it, we compute the ghost sheaf $\overline{\mc{M}}$, as well as the ghost sheaf $j^{-1}\overline{\mc{M}}$ of the restriction of $\mc{M}$ to $D$. Denote by $i_1 : D_1 \to \aaaa$, resp. by $i_2 : D_2 \to \aaaa$ the inclusion maps. Denote by $\nnn$ the constant sheaf of monoids determined by $\IN$ on $D_1$, resp. $D_2$. We have a map of sheaves on monoids
$$
\phi : \mc{M} \to i_{1,*}\nnn\oplus i_{2,*}\nnn,
$$
defined as follows. Let $u : U\to \aaaa$ be an \'etale morphism, and let $f$ be a regular function on $U$ that is invertible away from $u^{-1}(U)$. Then
$$
\phi(U)(f) := (n,m),
$$
where $n$, resp. $m$, is the order of vanishing of $f$ along $u^{-1}(D_1)$, resp. $u^{-1}(D_2)$. The map $\phi$ factors through $\overline{\mc{M}}$. Indeed, if $f$ and $g$ have the same order of vanishing along $u^{-1}(D_1)$ and $u^{-1}(D_2)$, then $f\cdot g^{-1}\in\mc{O}^\times_U$, so that $f=g$ in $\overline{\mc{M}}(U)$. In fact, the kernel of $\phi$ is $\mc{O}^\times_{\aaaa}$, so that we obtain an injection:
$$
\overline{\mc{M}} \hookrightarrow i_{1,_*}\nnn\oplus i_{2,_*}\nnn.
$$
Moreover, the functions $x^ny^m$ have orders of vanishing $(n,m)$ and thus the above map is surjective as well, thus an isomorphism. In particular, the stalk of $\overline{\mc{M}}$ at $x\in\aaaa$ is
$$
\left\{ \begin{array}{c l l}
\IN\oplus\IN & \quad \text{if } x=(0,0), \\
\IN & \quad \text{if } x\in D-\left\{(0,0)\right\}, \\
0   & \quad \text{otherwise}.
\end{array}
\right.
$$
By abuse of notation, denote by $i_1$, resp. by $i_2$, the inclusions $D_i \hookrightarrow D$. Recall that $j^{-1}\overline{\mc{M}}=\overline{j^*\mc{M}}$ as noted at the end of section \ref{definitions}. It follows that
$$
\overline{j^*\mc{M}}=i_{1,*}\nnn\oplus i_{2,*}\nnn.
$$
At the level of stalks, we can think of having a copy of $\IN$ on each component of $D$. In particular, this sheaf of monoids has nothing to do with functions on $D$, but rather remembers how $D$ is embedded into $\aaaa$ (it encodes the possible order of vanishing of functions).

\end{example}

\begin{example}
The previous example generalizes as follows. Let $X$ be a locally Noetherian normal scheme and let $D\subset X$ be a closed subset of pure codimension $1$. Take $\mc{M}$ to be the divisorial log structure associated to $D$. Let $\overline{x}\to X$ be a geometric point and let $r$ be the number of components of $D$ that meet $\overline{x}$. Then there is an injection
$$
\overline{\mc{M}}_{\overline{x}} \to \IN^r.
$$
The proof is analogous to the one given in the previous example. In particular, the above map is again induced by sending the germ of a regular function (invertible away from $D$) to its order of vanishing along the $r$ components. Encoding the possible orders of vanishing, the divisorial log structure can be thought of as describing geometric information about how $D$ is embedded into $X$.

\end{example}
In the last two examples, we computed the stalks of some ghost sheaves. A map of log schemes comes along with a pullback map of sheaves of monoids, and thus induces a pullback map on the stalks. These maps of monoids (or rather, of sheaves of monoids) can be thought of as extra combinatorial data. The next two examples explore the geometric information encoded by this data.

\begin{example}
We consider the case of a map
$$
f : X^\dagger \to \spec k^\dagger.
$$
from a log scheme to the standard log point over a field $k$. It follows from the map at the level of schemes that $X$ is defined over $k$. The pull back map fits into a commutative diagram
$$
\xymatrix{
f^{-1}\mc{M}_{\spec k^\dagger} = k^\times \oplus \IN \ar[r]^<<<<<{f^\#} \ar[d] & \mc{M}_X \ar[d] \\
f^{-1}\mc{O}_{\spec k} = k \ar[r]^<<<<<<<{f^*} & \mc{O}_X.
}
$$
It follows that $f^\#$ is determined by a map $\IN \to \mc{M}$. This in turn corresponds to a choice of section $\rho \in \Gamma(X,\mc{M}_X)$, forming a commutative diagram:
$$
\xymatrix{
(0,1) \ar@{|->}[d] \ar@{|->}[r] & \rho \ar@{|->}[d] \\
\alpha_{\spec k^\dagger}(0,1) = 0 \ar@{|->}[r] & \alpha_X(\rho) = 0.
}
$$
It follows that the extra data carried by $f$ is that of a section $\rho$ of $\mc{M}_X$ with the property that $\alpha_X(\rho)=0$.

\end{example}

\begin{example} \label{ex-blowup}
We now consider a map in the opposite direction of the previous example. Consider the affine plane $\left(\aaaa\right)^\dagger$ with log structure $\mc{M}$ induced by the divisor $D$ consisting of the union of the coordinate axes. Denote by $\spec k^\dagger$ the standard log structure on $\spec k$. Since we have not introduced toric geometry, for what follows we do not provide details - those can be found in \cite{kan}. Consider maps
$$
f : \spec k^\dagger \to \left(\mathbb{A}_k^2\right)^\dagger,
$$
mapping $\spec k$ to the origin. We explore the additional information carried by the pull back of sheaves of monoids. Denote by $\overline{0}$ a geometric point mapping to $0$. We have the pull back map
$$
f^\# : f^{-1}\mc{M} = \mc{M}_{\overline{0}} \longrightarrow \mc{M}_{\spec k^\dagger} = k^\times \oplus \IN,
$$
which fits into a commutative diagram
$$
\xymatrix{
\mc{M}_{\overline{0}} \ar[d] \ar[r] & k^\times \oplus \IN \ar[d] \\
\mc{O}_{\aaaa,\overline{0}} \ar[r] & k.
}
$$
Now, cf. \cite{kan}, the choice of pull back map $f^\#$ corresponds to a choice of toric blow up of $\mathbb{A}^2_k$ at the origin and a choice of point on the exceptional divisor (plus some minor extra information). In particular, the choice of $f^\#$ corresponds to a birational transformation on $\aaaa$. We discuss in section \ref{log-gw} how this insight is used to define log Gromov-Witten invariants.

\end{example}

\subsection{Properties}
\label{propr}

The goal of this section is the definition of log smoothness, Definition  \ref{def_log_sm} . Before stating it, we need to introduce some further conditions that guarantee the well-behavedness of log schemes and log maps. The first one was explored in the examples of the previous section:
\begin{definition}
Let $f : X^\dagger \to Y^\dagger$ be a morphism of log schemes. Then $f$ is said to be \emph{strict} if the map
$$
f^\# : f^{-1} \mc{M}_Y \to \mc{M}_X
$$
induces an isomorphism of log structures (that is, an isomorphism of sheaves of monoids) between the pull-back log structure $f^*\mc{M}_Y$ and $\mc{M}_X$.
\end{definition}
In the next definition, a log structure is said to be fine if \'etale locally it is realized as the log structure induced by a constant sheaf of monoids. The last section contained a number of examples of such log structures.

\begin{definition}
\label{def-fine}
Let $X^\dagger$ be a log scheme. Then $\mc{M}_X$ is said to be \emph{fine} if \'etale-locally the following conditions are satisfied: There is an \'etale open cover $\left\{ f_i : U_i \to X \right\}$ of $X$. For each $f_i$, there is a finitely generated monoid $P_i$ and a morphism of sheaves of monoids
$$
g_i : \widetilde{P} \to \mc{O}_U,
$$
where $\widetilde{P}$ denotes the constant sheaf of monoids on $U$ induced by $P$. Then, the log structure induced by $g_i$ is required to be isomorphic to the pull-back log structure $f_i^*\mc{M}_X$.
\end{definition}
We now come to the definition (by infinitesimal lifting criterion) of log smoothness for fine log schemes.

\begin{definition} \label{def_log_sm}
Let $f : X^\dagger \to Y^\dagger$ be a map of fine log schemes and assume that $f$ is of locally finite presentation. Then $f$ is said to be \emph{log smooth} if for each commutative diagram of fine log schemes
$$\xymatrix{
T^\dagger \ar@{^{(}->}[d]_\iota \ar[r] & \ar[d]^f X^\dagger \\
T'^\dagger \ar[r] & Y^\dagger,
}
$$
where $\iota$ is a strict closed log immersion and where $T$ is defined by a nilpotent ideal in $\mc{O}_{T'}$, there exists a unique log map $g : T'^\dagger \to X^\dagger$ making the diagram
$$\xymatrix{
T^\dagger \ar@{^{(}->}[d]_\iota \ar[r] & \ar[d]^f X^\dagger \\
T'^\dagger \ar@{.>}[ru]|-{g} \ar[r] & Y^\dagger
}
$$
commute.
\end{definition}
Note that unlike smooth morphisms, log smooth maps need not be flat, see \cite{kan} for an example.

We provide two examples of log smooth maps. The first example states that, with the appropriate choice of log structure, any toric variety is log smooth. The second example ties with the ideas outlined in the introduction. If $f : X \to \aaa$ is a smooth family of varieties, the fibers need not be smooth. The fibers will, however, be log smooth if $X$ is toric and if $f$ satisfies some properties. We do not provide the exact condition, as we haven't introduced toric varieties. The interested reader is referred to \cite{kan}.

\begin{example}
Let $X$ be toric variety and endow it with the divisorial log structure induced by the toric boundary. Then the structure map
$$
X^\dagger \to \spec k,
$$
where $\spec k$ is given the trivial log structure, is log smooth.
\end{example}
If $X$ is an affine toric variety over a field $k$, then there is a (toric) monoid $P$ such that $X=\spec k[P]$. The \emph{monoid ring} $k[P]$ is defined as the formal sum
$$
k[P] := \bigoplus_{p\in P} k \cdot z^p,
$$
with multiplication linearly induced by $z^p \cdot z^{p'} = z^{p+p'}$. See \cite{kan} for how monoid rings are related to toric varieties. Note that $\aaa=\spec k[\IN]$.

\begin{example}
Let $X=\spec k[P]$ be an affine toric variety. Let $f : X \to \aaa = \spec k[\IN]$ be a family induced by a non-zero map $\IN \to P$. Endow both $X$ and $\aaa$ with the divisorial log structure coming from their respective toric divisors. Then $f$ is log smooth. Furthermore, consider the fiber over $0$:
$$
\xymatrix{
X_0 \ar[d] \ar[r] & X \ar[d]^f \\
\spec k = \orig \ar[r] & \aaa.
}
$$
Endow $X_0$ with the pull-back log structure and $\spec k$ with the standard log structure (which is the pull-back log structure as we saw in section \ref{examples}). Then the map of fine log schemes
$$
X_0^\dagger \to \spec k^\dagger
$$
is log smooth (while it is not smooth).
\end{example}

We now introduce of the \emph{relative log tangent sheaf}, which will be used in \ref{log-curves-to-curves}.

\begin{definition}
\label{rel-log-tang}
Let $\pi : X^\dagger \to S^\dagger$ be a morphism of log schemes and let $\mathcal{E}$ be an $\mc{O}_X$-module. A \emph{log derivation} on $X^\dagger$ over $S^\dagger$ with values in $\mathcal{E}$ is a pair $(\ddd,\dddlog)$ as follows:
\begin{align*}
\ddd : \mc{O}_X & \longrightarrow \mathcal{E}
\end{align*}
is an ordinary derivation of $X$ over $S$.
\begin{align*}
\dddlog  : \mc{M}_X^{gp} & \longrightarrow \mc{E}
\end{align*}
is a morphism of abelian sheaves such that $\dddlog \circ \pi^{\#} = 0$. They are moreover required to satisfy the compatibility condition that for all $m\in\mc{M}_X$,
$$
\ddd\left( \alpha_X(m) \right) = \alpha_X(m) \cdot \dddlog(m).
$$
The resulting relative log tangent sheaf is denoted by $\Theta_{X^\dagger/S^\dagger}$.
\end{definition}

We end this section with some definitions needed in the next section.

\begin{definition}
A monoid $P$ is defined to be \emph{integral} if the cancellation law holds. That is, whenever $x+y=x'+y$ in $P$, then $x=x'$
\end{definition}

\begin{definition}
Let $P$ be a monoid with operation written additively and denote by $P^{gp}$ the Grothendieck group associated to $P$. Then $P$ is called \emph{saturated} if $P$ is integral and moreover if for all $p\in P^{gp}$, whenever there is $m\in\IN$ such that $mp\in P$, then $p\in P$ as well.
\end{definition}
The natural numbers are an example of a saturated monoid. Let $m\geq 2$ and consider the monoid
$$
P=\left\{ n\in\IN : n\geq m \right\}\cup\left\{ 0 \right\}.
$$
Then $P$ is not saturated.

Next comes a refinement of the property of being fine. Recall from section \ref{etale} that for sheaves defined in the \'etale topology, stalks are defined at geometric points.

\begin{definition}
Let $X^\dagger$ be a fine log scheme and use the same notation as for definition \ref{def-fine}. Then $\mc{M}_X$ is said to be \emph{fine saturated} if (in addition to being fine), at every geometric point $\overline{x} \to X$ of $X$, the stalk of the ghost sheaf $\overline{\mc{M}}_{X,\overline{x}}$ is saturated.
\end{definition}
The following couple definitions are motivated by the following (vaguely stated) fact: An integral homomorphism of monoids induces a flat map on the induced log schemes. See \cite{kan} for more details.

\begin{definition}
Let $P$ and $Q$ be integral monoids and let $h : Q \to P$ be a morphism of monoids. Then $h$ is called \emph{integral} if the following property holds. Assume there are $p_1,p_2\in P$ and $q_1,q_2 \in Q$ such that
$$
h(q_1)+p_1 = h(q_2) + p_2.
$$
Then there are $q_3,q_4\in Q$ and $p\in P$ such that
\begin{align*}
p_1 & = h(q_3) + p, \\
p_2 & = h(q_4) + p,  \\
q_1 + q_3 & = q_2 + q_4.
\end{align*}
\end{definition}

\begin{definition}
Let $f : X^\dagger \to Y^\dagger$ be a map of fine log schemes. Then $f$ is said to be \emph{integral} if the following holds. Let $\overline{x} \to X$ be a geometric point of $X$. Let $f(\overline{x}) \to Y$ be a geometric point such that $\overline{x} \to Y$ factors through $f(\overline{x})$. Then the induced morphism on the stalks of the ghost sheaves
$$
\overline{\mc{M}}_{Y,f(\overline{x})} \to \overline{\mc{M}}_{X,\overline{x}}
$$
is integral.
\end{definition}

\subsection{Torically transverse log curves}\label{tor-trans}
In this section, we introduce in definition \ref{toric-trans} and \ref{toric-trans-log-curve} below the notion of \emph{torically transverse (log) curve}, which is used for definition \ref{def-numbers-for-count} and in section \ref{section-match-counts}. This section assumes (conversational) knowledge of toric geometry and stable maps. Let $\Sigma \subseteq \mathbb{R}^n$ be a fan and denote by $X_\Sigma$ the associated toric variety. Denote by $\partial X_\Sigma$ the toric boundary (the union of the codimension 1 toric strata). Denote moreover by $\cup_{\tau \in \Sigma^{>1}} D_\tau$ the union of the toric strata of codimension two or higher. The reader versed in toric geometry will recognize the meaning of the notation.

\begin{definition}\label{toric-trans}
A curve $C\subseteq X_\Sigma$ is said to be \emph{torically transverse} if it is disjoint from $\cup_{\tau \in \Sigma^{>1}} D_\tau$.
\end{definition}
Note that it follows that a torically transverse curve has no irreducible component contained in a codimension 1 stratum (since then it would intersect $\cup_{\tau \in \Sigma^{>1}} D_\tau$).

\begin{definition}
A stable map $f : C\to X_\Sigma$ is called \emph{torically transverse} if its image $f(C)\subseteq X_\Sigma$ is torically transverse and no irreducible component of $C$ is mapped into $\partial X_\Sigma$.
\end{definition}
Consider now the following situation. Let $k$ be a field and let $\Sigma$ be a fan. Denote by $X$ the toric variety associated to $\Sigma$. Moreover, denote by $\Sigma(\mathbb{A}_k^1)$ the fan of $\mathbb{A}_k^1$. Endow both $X$ and $\mathbb{A}^1_k$ with the standard log structure, i.e. with the divisorial log structure associated to the toric boundary. Assume we are given a surjective map of fans $\Sigma \to \Sigma(\mathbb{A}_k^1)$. This yields a log smooth map
$$
\pi : X \to \mathbb{A}^1_k,
$$
which is a degeneration of toric varieties. Denote furthermore by $X_0 = \pi^{-1}(0)$ the central fibre. Endow $X_0$ with the log structure induced by the log structure of $X$. Restricting $\pi$ to the central fibre, we obtain a morphism of log schemes
$$
\pi_0 : X_0^\dagger \to \spec k^\dagger,
$$
where $\spec k^\dagger$ denotes the standard log point, as in Example \ref{stand-log-point}.

\begin{definition}\label{toric-trans-log-curve}
Assume the above setup and let $C^\dagger$ be a log curve with fine saturated log structure. Consider a log map $f : C^\dagger \to X_0^\dagger$, whose underlying scheme map is a stable map. Assume moreover that for each codimension 1 toric strata $D$ of $X_0$, the restriction $f^{-1}(D) \to D$ is torically transverse. Then, a \emph{torically transverse log curve in} $X_0^\dagger$ is given by a commutative diagram of log maps
$$
\xymatrix{
C^\dagger \ar[rd]^g \ar[rr]^f & & X^\dagger_0 \ar[dl]_{\pi_0} \\
& \spec k^\dagger. & 
}
$$

\end{definition}

\subsection{Log smooth curves}
\label{log-sm-curves}
The starting point of logarithmic Gromov-Witten theory, as discussed in the next section, is the realization that log smooth maps behave very much like stable maps, and that many of the geometric tricks needed for stable maps are already encoded by morphisms of log structures. The latter property was illustrated by the examples of maps of log schemes in section \ref{examples}. Here, we outline the local structure of log smooth curves, as established by F. Kato in \cite{kato-log-curves}.

Consider a morphism of log schemes
$$
f : C^\dagger \to W^\dagger
$$
satisfying the following list of conditions:
\begin{itemize}
\item The map $f$ is log smooth, integral and of relative dimension $1$;
\item As a scheme, $W=\spec A$, where $A$ is a complete local ring over an algebraically closed field $k$;
\item The log schemes $C^\dagger$ and $W^\dagger$ are fine saturated.
\end{itemize}
Denote by $0\in W$ the closed point. As $k$ is algebraically closed, $0$ is the only geometric point and it follows that any sheaf will be determined by its stalk at $0$. Analogously, any map of sheaves will be determined by its values on the stalk at $0$. Let $Q := \overline{\mc{M}}_{W,0}$. Then the log structure on $W$ is determined by a morphism
$$
\sigma : Q \to A.
$$
Denote by $C_0$ the fibre of $f$ over $0$ and let $x$ be a geometric point of $C_0$, in this case a $k$-valued point. The structure theorem by F. Kato then states that for a sufficiently small \'etale neighbourhood $U\to X$ of $x$, the log structure restricted to $U$ is isomorphic to one of the three following log schemes.

\paragraph{(1) Smooth point}

For the first case, $U=\spec A[u]$, $f$ is smooth (in the conventional sense) and the log structure on $U$ is induced by
\begin{align*}
Q & \to \mc{O}_U \\
q & \mapsto f^*\sigma(q).
\end{align*}
The log structure thus is just the pull back of the log structure on the base, and contains no additional information.

\paragraph{(2) Double point}

Let $m$ denote the maximal ideal of $A$. In the second case, there is $t\in m$ such that $U=\spec A[u,v]/(uv-t)$. Moreover, the log structure is as follows. There is $\alpha\in Q$ with  $\sigma(\alpha)=t$. Consider the diagonal map $\IN\to\IN^2$ and let $\IN\to Q$ be determined by $1\mapsto \alpha$. Denote by  $\IN^2\oplus_\IN Q$ the fibred sum determined by these maps. Then the log structure on $U$ is induced by the pre-log structure
\begin{align*}
\IN^2 \oplus Q & \to \mc{O}_U, \\
\left((a,b),q\right) & \mapsto u^av^bf^*\sigma(q).
\end{align*}
Here, $C_0$ is nodal.

\paragraph{(3) Log marked points}

For the third case, $U=\spec A[u]$ and the log structure is induced by the pre-log structure
\begin{align*}
\IN \oplus Q & \to \mc{O}_U, \\
\left(a,q\right) & \mapsto u^af^*\sigma(q).
\end{align*}
In this case, the point $u=0$ is the image of a section $W\to C$, which should be thought of as a marked point. Moreover, the log structure is the sum of on one hand the pull-back log structure from the base and on the other hand the divisorial log structure associated to the divisor $u=0$. In addition to simply choosing a point $u=0$, the ghost sheaf at $u=0$ has (compared to a smooth point) an additional copy of $\mathbb{N}$. Maps from $C$ to a log scheme will come with a pullback map at the level of the ghost sheaves. In the case explored in the next section, that pullback map carries some geometric information, as it encodes some intersection multiplicity.

\subsection{Towards logarithmic Gromov-Witten theory}
\label{log-gw}

With the goal of motivating logarithmic Gromov-Witten invariants, we briefly sketch its starting idea. The theory of logarithmic Gromov-Witten invariants was established by Gross-Siebert in \cite{loggw}, by Chen in \cite{Chen-DF} and by Abramovich-Chen in \cite{Chen-Abram-DF}.

The reader familiar with stable curves will recognize the similarities they have in common with log smooth curves. Stable curves are locally either smooth or nodal, and are endowed with marked points. The main difference is that a log marked point comes along with a ghost sheaf stalk isomorphic to $\IN$. This allows for much more flexibility when considering maps from smooth log curves. In one dimension lower, we saw in example \ref{ex-blowup} how mapping the log point to the plane corresponds (roughly) to a blow up of the plane and a choice of point on the exceptional divisor. In that example, the log map contained extra geometric information. Analogously, the log structure on a log smooth curve can be used to encode intersection multiplicities, as we explain now, by comparing log stable maps to relative stable maps.

Relative Gromov-Witten arise when the target variety $X$ degenerates to a variety given by the union of two smooth varieties $Y_1\cup_D Y_2$ glued along a smooth divisor $D$. In that situation, the degeneration formula applies. This formula, along with localization, is one of the most important tools in Gromov-Witten theory. Broadly speaking, the degeneration formula relates the Gromov-Witten invariants of $X$ to sums of gluings of relative invariants of $(Y_i,D)$. In practice, choosing a suitable degeneration, one hopes to computes the Gromov-Witten invariants of $X$ in terms of simpler to compute relative Gromov-Witten invariants.

The theory of relative Gromov-Witten invariants has two major disadvantages though. Firstly, it applies only when $D$ is a smooth divisor, limiting the range of degenerations that can be considered. Secondly, the definition of the relevant moduli space, though elegant, is somewhat unnatural and technically complicated to deal with. Indeed, in order to obtain a compact moduli space, the target variety is allowed to degenerate. More concretely, consider the situation of a smooth variety $X$ with smooth divisor $D$ (the situation in which relative Gromov-Witten invariants are defined). Let $C\to X$ be a relative stable map and assume that $C$ is not mapped into $D$. Then for each point of intersection of the image of $C$ with $D$, there is a well-defined intersection multiplicity and the (non-compactified) moduli of relative stable maps is stratified according to the different intersection multiplicities. However, if a component of $C$ degenerates into $D$ problems arise. For one, the intersection multiplicities are no longer well-defined. The solution developed by Li in \cite{li-stable-relative} is to allow the target to degenerate. If in the limit (a component of) the relative stable map limit is mapped into $D$, then $D$ is replaced by a $\mathbb{P}^1$-bundle on it. The relative condition is then considered at the divisor at $\infty$. This process is then repeated as necessary.

Logarithmic Gromov-Witten theory avoids these two shortcomings. Firstly, the divisorial log structure associated to $D$ exists whether $D$ is smooth or not. Secondly, the extra information carried by the log marked points is such that degenerations of the target variety are not needed. A marked point has a ghost sheaf of $\IN$ on top of it. With the divisorial log structure, a point on the divisor has a ghost sheaf of $\IN$ (in the case of $D$ smooth). The log map determines a map between these two copies of $\IN$. This map is the multiplication by a positive integer, which is the intersection multiplicity. When a component of the curve degenerates into $D$ then, the map on the ghost sheaves keeps track of the intersection multiplicity, which thus remains well-defined.

This is just a brief glimpse as to why log Gromov-Witten invariants are a suitable generalization of relative Gromov-Witten invariants. On one hand, they are simpler to work with. On the other hand, they allow for much more general degenerations.

\section{Tropical geometry}
\label{sec:tropics}
Tropical geometry can be roughly understood as a ``piecewise-linear" version of algebraic geometry.  It has flourished over the past few decades, quickly establishing itself as an important combinatorial and conceptual tool in the study of enumerative geometry.  The name ``tropical" was coined to honor Imre Simon, who pioneered many of the field's techniques.  Mikhalkin's demonstration of the equivalence of tropical and classical curve counting \cite{mik} was the inspiration for a number of results showing that a surprising amount of information can be naturally encoded in these piecewise-linear structures.  We will begin this section with some background on the field's connections to classical algebraic geometry and then proceed to rigorously define several tropical objects necessary in the following.  The motivational remarks owe a great deal to Mikhalkin's \cite{MikAm} and Gathmann's \cite{gat} excellent expositions, while the the latter definitions can be found in \cite{kan}.
\subsection{Motivation}
\label{subsec:motivation}
Throughout this chapter, tropical curves will manifest themselves as piecewise linear graphs in the plane.  The relationship of these objects with classical algebraic curves in $\mathbb{P}^2$ or $(\mathbb{C}^*)^2$ will be explored in this section.
\subsubsection{From amoebas to tropical curves}
Given a variety $V\subset(\mathbb{C}^*)^n$, one can examine the image under the map $\Log_e:(\mathbb{C}^*)^n\rightarrow \mathbb{R}^n$ defined by 
$$\Log_e(z_1,\ldots, z_n):=(-\log_e|z_1|, \ldots ,- \log_e|z_n|),$$
where $e=\ln(1)$.  The set $\Log_e(V)\subset \mathbb{R}^2$ is called the amoeba of $V$.  Note that this construction is quite widely applicable, as all toric varieties contain a copy of $(\mathbb{C}^*)^n$.  
\begin{figure}[h]
\centering
\begin{tabular}{cc}
\includegraphics[width=0.5\textwidth]{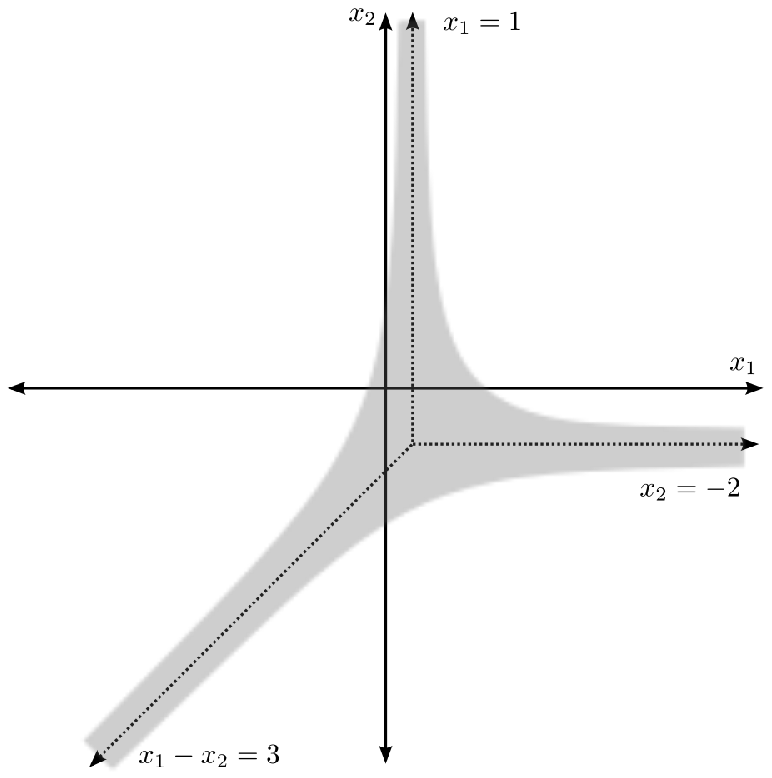}&
\includegraphics[width=0.5\textwidth]{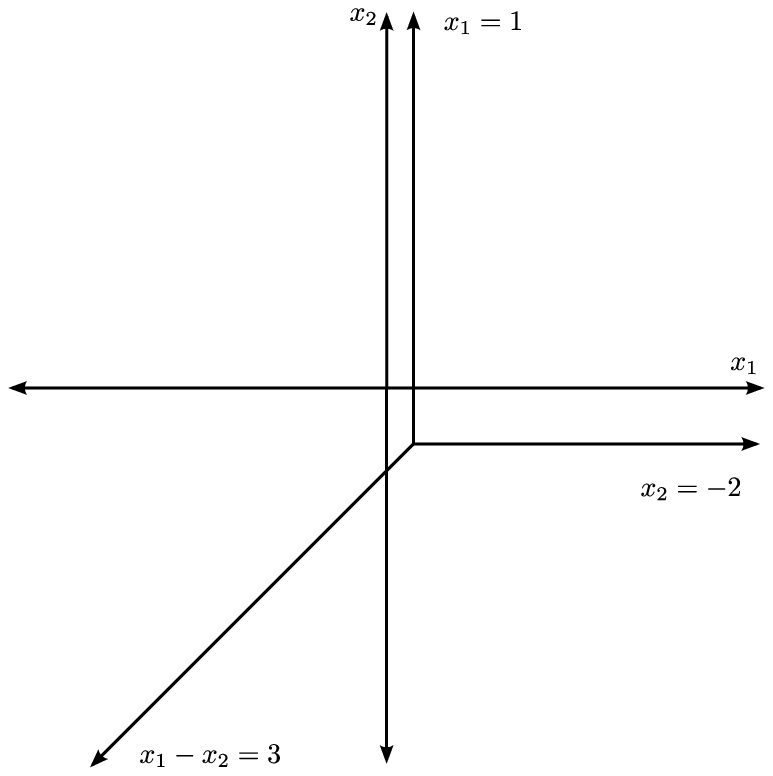}
\end{tabular}
\caption{An approximation of the image of $C = \{(z_1, z_2)|e^1z_1 + e^{-2}z_2 = 1\}$ under $Log_e$ is given on the left, while its ``spine" is given on the right. }
\label{amoeba}
\end{figure}

Upon an examination the amoebas of curves in $(\mathbb{C}^*)^2$ such as those in Figure \ref{amoeba}, one quickly sees that they share certain features.  One of these is the existence of ``arms" heading off to infinity; it is the resemblance of these features to the amoeba's pseudopods that earns these mathematical objects their name.  The ``fleshy" part of the picture can be considered extraneous, and one may wish to simplify the situation further, distilling the picture into the collection of piecewise linear components hinted at by the shape.  It's easy to see that one can roughly achieve this by zooming out on the graph until the pseudopods are very thin.  Mathematically, this could be achieved by defining
$$\Log_t(z_1,\ldots, z_n):=(\log_t|z_1|, \ldots , \log_t|z_n|)$$
and examining the amoeba given for very small $t$.  This process is unsatisfactory, however, because it would move vertex of the resulting graph to the origin.
A solution is found in replacing each coefficient  $a$ of the defining equations by $t^{\log_e a}$, thus defining a family of curves $V_t$ in $(\mathbb{C}^*)^2$.   Taking the limit as $t$ goes to $\infty$ of $\Log_t(V_t)$ gives us the piecewise linear graph we can intuitively see hiding in each of these amoebas.  Although biologically confusing,  this rigid structure is called the ``spine" of the amoeba, and the spine pictured on the right in Figure \ref{amoeba} give us our first example of a tropical curve.  

The Gross-Siebert program suggests that mirror symmetry is can be exhibited by an exchange of ``tropical" data on the shared base of a fibration.  The process described above is analogous to that of passing to the large complex structure limit of a family of varieties, suggesting that tropical objects may reasonably be expected to encode mirror symmetric data.

Although our strategy of degenerating amoebas to their spines is effective, it is a bit cumbersome. A shortcut is suggested by our replacement of the coefficients $a\in \mathbb{C}$ by $t^{\log_e a}$.   The field $K$ of Puiseux series over $\mathbb{C}$ is defined, roughly, to be the set of formal power series $\alpha=\sum_{k=k_0}^\infty c_n t^{k/n}$.  Therefore, instead of thinking of a family of curves $V_t$, we can instead consider a single curve over $(K^*)^2$.  How should we then interpret the map $\Log_t$?  

Suppose we have an element $f:=\sum_{k=k_0}^\infty c_n t^{k/n}\in K^*$ and $k_0\neq 0$.  For $0<r<1$, define $f(r)=\sum_{k=k_0}^\infty c_n r^{k/n}$.   It's then easy to see that $\lim_{r\rightarrow 0^+}\log_r f(r)=k_0/n$.  This assignment of $$\val:\sum_{k=k_0}^\infty c_n t^{k/n}\mapsto k_0$$ has some nice properties.  In fact, if we define $v(0)=\infty$ it's easy to see that 
\begin{eqnarray*}
\val(a)&=\infty \mbox{ if and only if } a=0\\
\val(ab)&=\val(a)+\val(b)\\
\val(a+b&)\geq \min\{\val(a), \val(b)\}
\end{eqnarray*}
which makes $\val$ into something known as a {\it non-Archimedean valuation}.  These properties will come into play shortly.  Continuing our intuitive construction, we should feel justified in making the following definition.
\begin{definition}
Let $V\subset (K^*)^n$ be an algebraic variety.  Define the tropicalization $V_{trop}$ of $V$ by $$V_{trop}:=\overline{\Val(V)},$$ where $\Val(k_1,\ldots, k_n):=(\val(k_1), \ldots \val(k_n))$.
\end{definition}
\subsubsection{The min-plus semiring and tropical varieties}
Because we wish to study the ``tropical" image of our varieties, we define an arithmetic on $\mathbb{R}$ corresponding to the non-Archimedean valuation. 
\begin{definition}
Let $a$, $b\in \mathbb{R}$.  Define:
\begin{eqnarray}
a\oplus b &= \min(a,b)\\
a\odot b &= a+b
\end{eqnarray}
where $+$ is standard addition on $\mathbb{R}$.
\end{definition}
Note that multiplicative inverses are given by subtraction, while there is no additive inverse.  The rough idea is that algebraic geometry in $\mathbb{R}^n$ with the min-plus arithmetic should have a correspondence to the tropicalization of algebraic geometry in $(K^*)^n$.  

Suppose we have a polynomial 
$$p(x_1,\ldots, x_n):=\sum_{i \in S} a_i x_1^{i_1}\cdots x_n^{i_n}$$
with $S \subseteq \mathbb{Z}^n$ a finite set, $i:=(i_1,\ldots, i_n)$, and $a_i\in K^*$.  The equation $p=0$ defines a variety $V$ in $(K^*)^n$, and thus defines a tropical curve $V_{trop}$.  Is there a way to recover $V_{trop}$ without passing through $(K^*)^n$?
Consider the tropical version of the above polynomial
\begin{eqnarray}
p_{trop}(z_1, \ldots, z_n)&:=\sum_{i \in S} \Val(a_i)\odot z_1^{i_1}\odot \cdots \odot z_n^{i_n} \label{eq1}\\
&=\min(\Val(a_i)+i_1z_1+\ldots + i_nz_n i\in S), 
\end{eqnarray}
where the sum in Equation \ref{eq1} is $\oplus$ and the $z_i$ are the standard coordinates on $\mathbb{R}^n$.  Note that $p_{trop}$ defines a piecewise linear map $\mathbb{R}^n\rightarrow \mathbb{R}$.  Suppose $p(r_1,\ldots r_n)=0$ for $r_i\in K^*$.  This means $\sum_{i \in S} a_i r_1^{i_1}\cdots r_n^{i_n}=0$.  

Define $m_i=\val(a_i r_1^{i_1}\cdots x_n^{r_n})$,  and let $l=\min(m_i)$.  The coefficient of $t^q$ in $p(r_1,\ldots r_n)$ must be zero for all values of $q\in \mathbb{Q}$, and thus $m_i=l$ for at least two values of $i\in S$.  Let the set of such $i\in S$ be given by $M\subseteq S$.  If we reinterpret this condition in terms of $p_{trop}$, we see that 
\begin{eqnarray}
p_{trop}(\val(r_1),\ldots, \val(r_n))&=\min(\val(a_i)+i_1z_1+\ldots + i_nz_n\mid i\in S)\\
&=\val(a_m)+m_1\val(r_1)+\ldots + m_n\val(r_n)
\end{eqnarray}
for any $m\in M$.  In particular, the minimum is simultaneously achieved by at least two monomials at $(\val(r_1), \ldots, \val(r_n))$.  Therefore, $V_{trop}$ must be contained in the locus of the non-smooth pieces of the function defined by $p_{trop}$.  
This motivates an alternate viewpoint of tropical curves as the so-called ``corner locus" of the piecewise linear functions defined by polynomials using the min-plus arithmetic.  Such objects are significantly easier to handle and have very nice combinatorial properties that allow further abstraction. 

We can think of $S\subseteq N$ and write
$$p_{trop}(m)=\min(\val(a_i)+\langle n, m \rangle \mbox{ for } n\in S)$$ as a function from $M_\mathbb{R}$ to $\mathbb{R}$.
 \begin{definition}
 Let $f(z)=\sum_{n\in S} a_n z^n = \min\{a_n+\langle n, z\rangle \mbox{ for } n\in S\}$.  As a set, we define the tropical hypersurface $V(f)$ associated to $f$ to be the set in $M$ defined by the corner locus of $f$.
 \end{definition}
 Thinking of $V(f)$ as a union of codimension one polyhedra of $M$, we associate a weight $w(e)$ to each polyhedron $e$.  This is a measure of the severity of the bend that occurs at $e$, and is defined to be the index of $n-n'$ in $N$, where $n'$ and $n$ define the behavior of $f$ on either side of $e$.
One important implication of the geometry behind this definition is the \emph{balancing condition}.  If $\dim M_\mathbb{R}=2$ so $V(f)$ is a piecewise linear graph in $\mathbb{R}^2$, we can formulate it in the following way.  Let $\tau$ be a vertex of $V(f)$ and $e_1,\ldots, e_n$ be edges connected to $\tau$ and $p_1, \ldots, p_n\in M$ be primitive vectors such that $p_i$ points away from $\tau$ in the direction of $e_i$.  Then
$$\sum_{i=1}^n p_i w(e_i)=0\in M.$$
This condition puts strong constraints on the nature of $V(f)$.

 Let's reexamine our example in this context.  The polynomial defining the amoeba in Figure \ref{amoeba} has the following counterpart in $K[x_1, x_2]$:
$$p(x_1, x_2)=t^{-1} x_1+t^{2} x_2-1$$
The tropicalized version is given by
$$p_{trop}(z_1, z_2)=-1\odot z_1\oplus 2 \odot z_2\oplus 0$$
and its graph is illustrated in Figure \ref{3D}.

\begin{figure}[h]
\centering
\begin{tabular}{cc}
\includegraphics[width=0.5\textwidth]{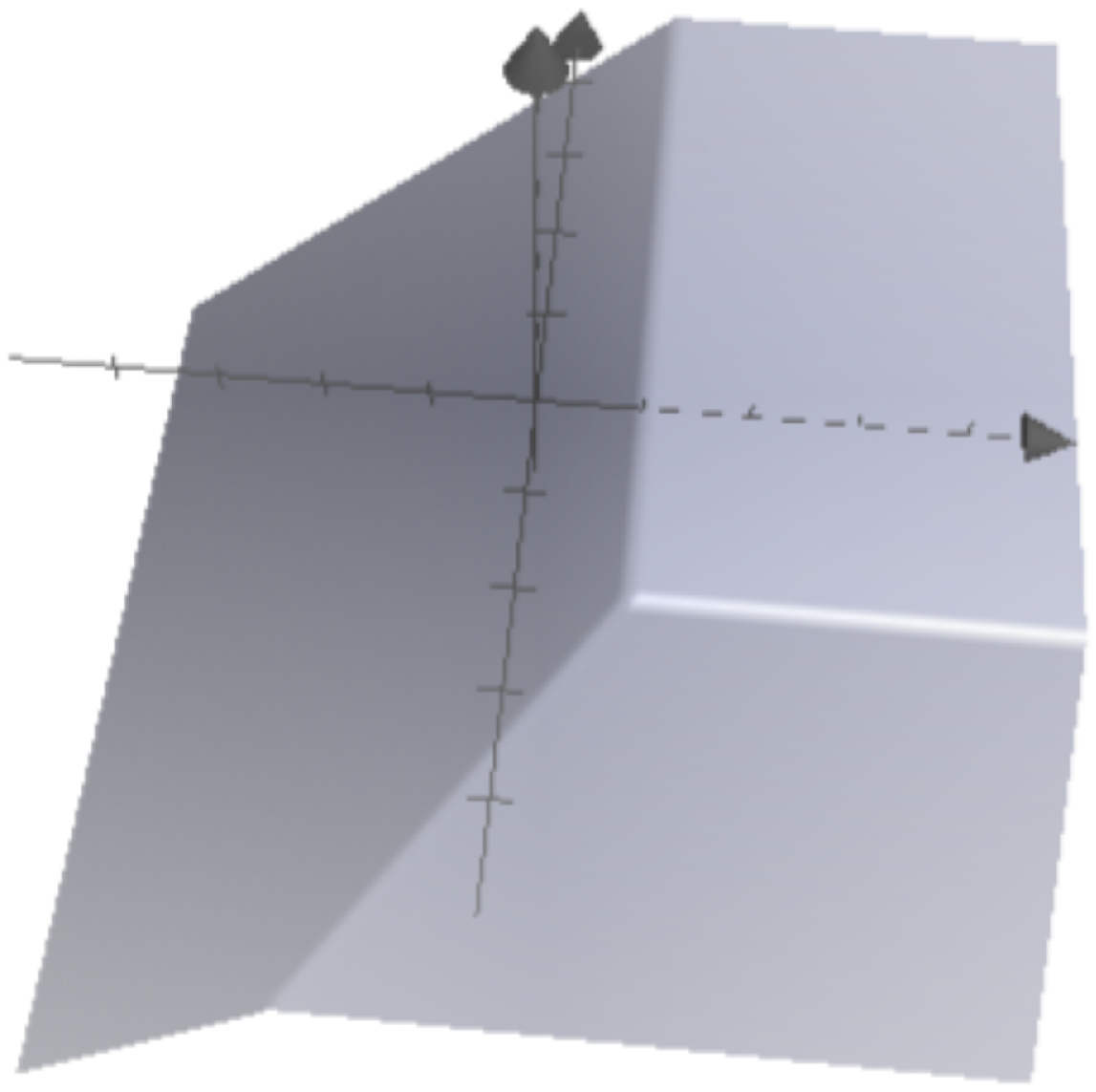}&
\includegraphics[width=0.5\textwidth]{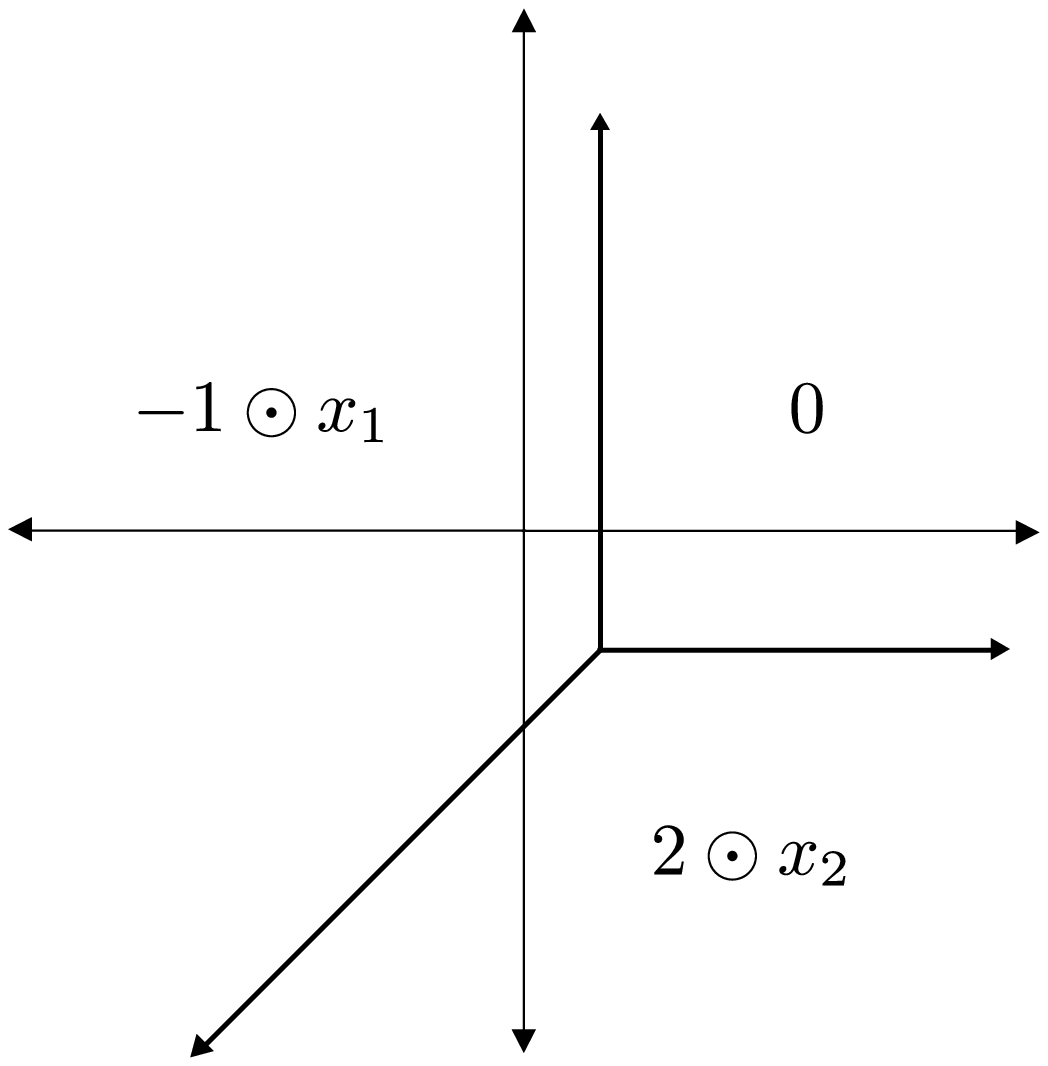}
\end{tabular}
\caption{On the left, a graph of $p_{trop}(z_1, z_2)$.  The diagram on the right indicates the monomial that determines the behavior of $p_{trop}(z_1, z_2)$ in each of the regions demarcated by the corner locus of the graph.  The weights of the edges of $V(p_{trop})$ are also indicated.}  
\label{3D}
\end{figure}

\subsection{Combinatorial objects}
For much of what follows, it is useful to abstract the definition of the tropical curve to a combinatorial formulation satisfying the properties explored above.
\subsubsection{Marked tropical curves}
\label{subsec:marked tropical curves}
For our purposes, it's most convenient to deal with strictly combinatorial objects incorporating the features we've discussed above. Let $\bar{\Gamma}$ be the topological realization of a graph with no bivalent vertices.  Let $\Gamma^{[1]}$ be the set of edges, $\Gamma^{[0]}$ the set of vertices.  Define $\Gamma$ to be $\bar{\Gamma}$ without its univalent vertices.  Note that $\Gamma$ generally will have non-compact edges, which we gather into a set $\Gamma^{[1]}_\infty$.  Assign a weight function $w: \Gamma^{[1]}\rightarrow \mathbb{Z}_{\geq 0}$ such that $w(\Gamma^{[1]}_\infty)\subseteq \{0,1\}$ and $w^{-1}(0)\subseteq \Gamma^{[1]}_\infty$.  Assign a label $x_i$ to each of the weight 0 edges using an inclusion
 \begin{eqnarray*}
& \{x_1,\ldots, x_n\}\hookrightarrow \Gamma^{[1]}_\infty\\
& x_i \mapsto E_{x_i}
\end{eqnarray*}
 
 The data $(\Gamma, x_1, \ldots, x_n)$ constitutes a \emph{marked graph}.  A marked graph can be given a geometric manifestation using the following definition.
 \begin{definition}[Marked parametrized tropical curve]
 \label{def2}
A \emph{marked parametrized tropical curve} [MPTC] is a continuous map $h:(\Gamma,x_1, \ldots, x_n)\rightarrow M_\mathbb{R}$ satisfying:
\begin{itemize}
\item If $E\in \Gamma^{[1]}_{\infty}$ and $w(E)=0$, then $h|_E$ is constant.  That is, $h$ collapses labeled edges.  On other edges, $h|_E$ is a proper embedding of $E$ into a line of rational slope in $M_\mathbb{R}$.
\item Let $V$ be a vertex of $\Gamma$, and $E_1, \ldots E_m$ be the edges adjacent to $V$.  Let $v(E_i)$ be a primitive vector pointing away from $h(V)$ along the direction of $h(E_i)$.  Then
$$\sum_{i=1}^mw(E_i)v(E_i)=0.$$
 \end{itemize}
In the following, we will conflate a collapsed edge with its label.  That is, if $$h:(\Gamma,x_1, \ldots, x_n)\rightarrow M_\mathbb{R}$$ is a marked parametrized tropical curve, we write $h(x_i)=h(E_{x_i})$.
\end{definition}
We say that two parametrized tropical curves $h:(\Gamma,x_1, \ldots, x_n)\rightarrow \mathbb{R}^n$ and $h':(\Gamma',x'_1, \ldots, x'_n)\rightarrow \mathbb{R}^n$ are \emph{equivalent} if there is a homeomorphism $\phi:\Gamma\rightarrow \Gamma'$ with $\phi(E_{x_i})=E_{x_i'}$ for each $i$ and $h=h'\circ \phi$.  We can then define a \emph{marked tropical curve} to be an equivalence class of parametrized marked tropical curves.

 We say a marked tropical curve $h$ is \emph{in} $X_\Sigma$ if, for each unmarked unbounded edge $E\in \Gamma_\infty^{[1]}$, $h(E)$ is a translate of some $\rho\in \Sigma^{[1]}$.  In this case we can define its degree.

\begin{definition}[Degree of a marked tropical curve]
If $h$ is a marked tropical curve in $X_\Sigma$, the \emph{degree} of $h$, notated $\Delta(h)$, is defined to be
$$\Delta(h):=\sum_{\rho\in \Sigma^{[1]}} d_\rho v_\rho\in T_\Sigma$$
where $d_\rho$ is the number of unbounded edges of $\Gamma$ that are mapped to translates of $\rho$ by $h$ and $T_\Sigma$ is as defined in Section \ref{toriccon}.
\end{definition} 
An unbounded edge of a tropical curve mapping in the direction of a ray $\rho\in \Sigma$ corresponds to an intersection of the corresponding classical curve with the toric divisor defined by $\rho$, justifying this naming convention.
\begin{definition}[Genus of a marked tropical curve]
If $h$ is a marked tropical curve in $X_\Sigma$, the \emph{genus} of $h$ is defined by
$$g(h):=b_1(\Gamma)$$.
\end{definition} 
As an exercise, convince yourself that $r(\Delta(h))=0$ for any marked tropical curve.  Given $\Delta(h)=\sum_{\rho\in \Sigma^{[1]}} d_\rho v_\rho\in T_\Sigma$, we define $|\Delta(h)|:=\sum_{\rho\in \Sigma^{[1]}} d_\rho$.

\begin{figure}[h!]
\centering
\begin{tabular}{cc}
\includegraphics[width=0.5\textwidth]{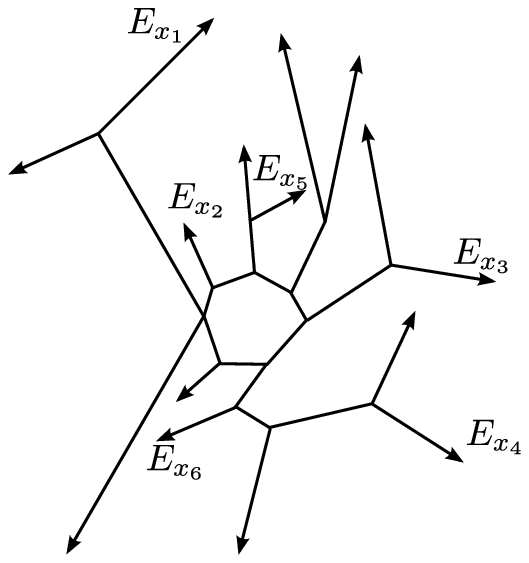} &
\includegraphics[width=0.5\textwidth]{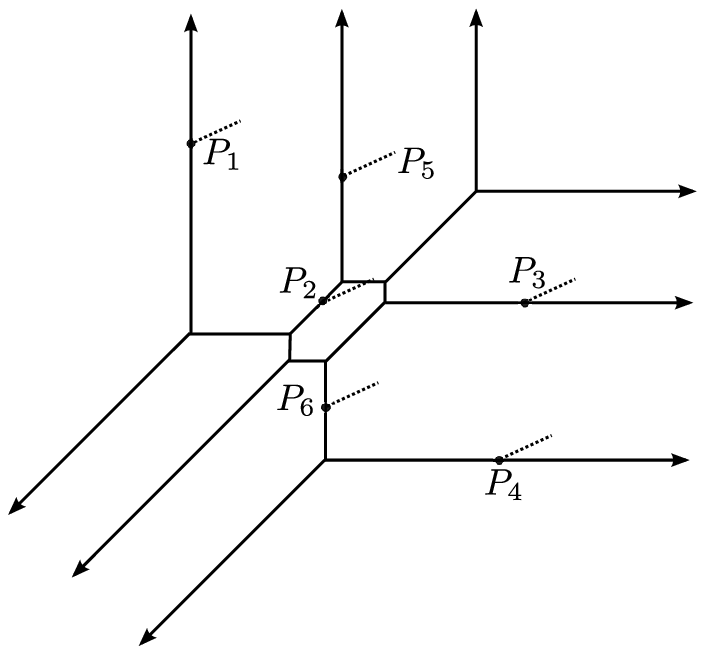} 
\end{tabular}
\caption{On the left, the graph $\Gamma$ underlying a marked parametrized tropical curve $h$ in $X_\Sigma$.  On the right, the image of $\Gamma$ under $h$ with $E_{x_i}$ mapping to $P_i$ in $M_\mathbb{R}$.  The dotted edges are of weight $0$, collapsed by $h$.  The genus of $h$ is 1, and the degree of $h$ is $3t_{\rho_0}+3t_{\rho_1}+3t_{\rho_2}$.  Note that there are an infinite number of inequivalent choices of maps $h$ given these particular choices of $\Gamma$, images of $E_{x_i}$ in the plane, and directions for the images of the unbounded edges of $\Gamma$.  That is, the image can be deformed while preserving these properties.}
\label{genus1}
\end{figure}

In order to use tropical curves for enumerative problems one must count them with a weighting known as the \emph{Mikhalkin multiplicity}.  See Section \ref{section:NS} for more on this.
\begin{definition}[$\Mult(h)$]
\label{def-mik-mult}
Let $h:\Gamma\rightarrow M_\mathbb{R}$ ($\dim M_\mathbb{R}=2$) be a trivalent marked tropical curve with no edges mapped on top of one another and weight one for all unbounded, unmarked edges.  For $V\in \Gamma^{[0]}$ with adjacent edges $E_1, E_2,$ and $E_3$, define
\begin{eqnarray*}
\Mult_V(h):=& w_1w_2|m_1\wedge m_2|\\
=& w_2w_3|m_2\wedge m_3|\\
=& w_3w_1|m_3\wedge m_1|
\end{eqnarray*}
if none of the $E_i$ are marked, and $\Mult_V(h)=1$ otherwise.  Here $w_i$ is the weight of $E_i$ and $m_i$ is a primitive (coprime entries) vector in $M$ pointing away from $V$ along the edge $E_i$.   Here we identify $M\wedge M$ with $\mathbb{Z}$ and sign ambiguity is absorbed by the absolute value.  Note that the equivalence of the statements is due to the balancing condition.  
Then we define
\begin{eqnarray}
\Mult(h):=\prod_{V\in \Gamma^{[0]}} \Mult_V(h).
\end{eqnarray}
\end{definition}
\begin{figure}[h]
\sidecaption[t]
\includegraphics[width=0.4\textwidth]{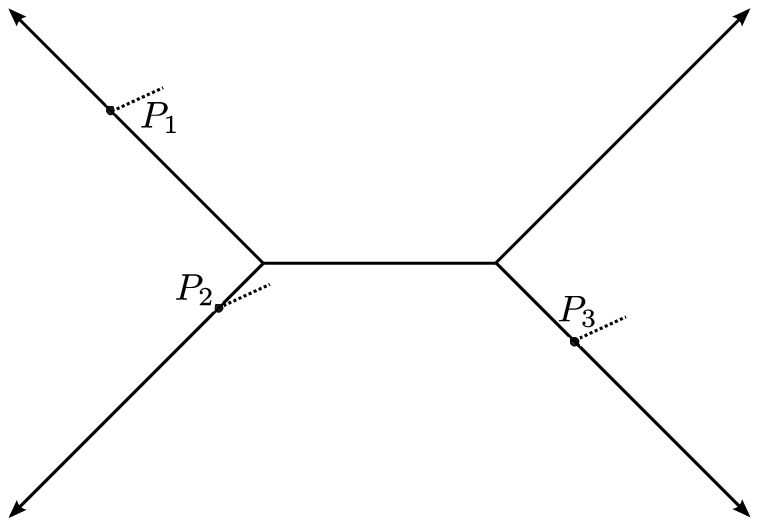}
\caption{The image of a marked, parametrized tropical curve, $h$.  Assume the outgoing edges are weight 1, pointing in the directions $(1,1)$, $(1,-1)$, $(-1,1)$, and $(-1,-1)$.  As an exercise, compute the Mikhalkin multiplicity of $h$.}
\label{mult}
\end{figure}

\begin{figure}[h]
\centering
\begin{tabular}{cc}
\includegraphics[scale=0.9]{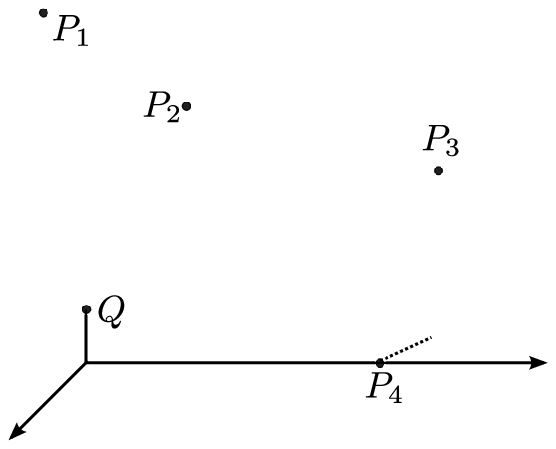}&
\includegraphics[scale=0.9]{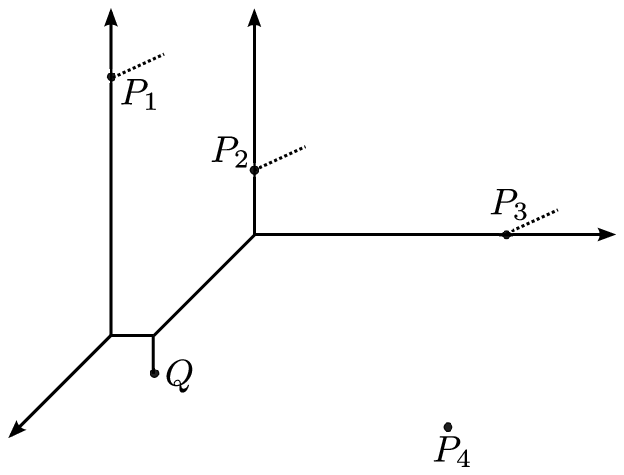}\\
\end{tabular}
\caption{The images of two tropical disks in $(X_\Sigma, P_1, \ldots,  P_5)$ with boundary $Q$.}
\label{diskfig1}
\end{figure}

\begin{figure}[h]
\centering
\begin{tabular}{cc}
\includegraphics[scale=0.7]{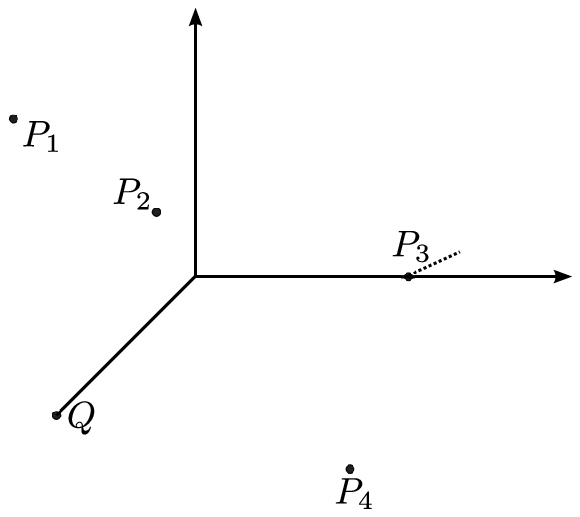}&
\includegraphics[scale=0.7]{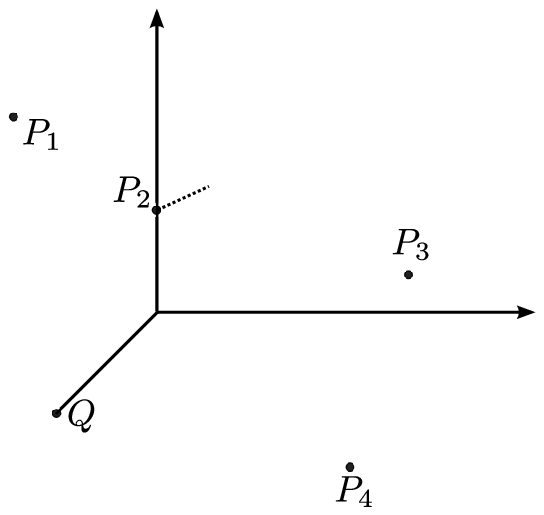}\\
\end{tabular}
\caption{The images of two more tropical disks in $(X_\Sigma, P_1, \ldots,  P_5)$ with boundary $Q$.}
\label{diskfig2}
\end{figure}

\begin{figure}[h]
\centering
\begin{tabular}{cc}
\includegraphics[scale=0.7]{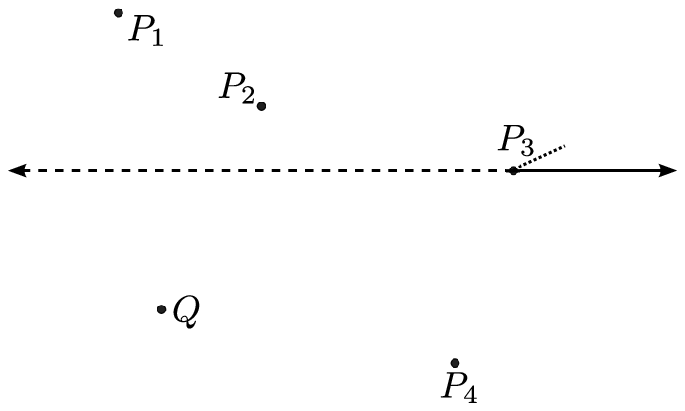}&
\includegraphics[scale=0.7]{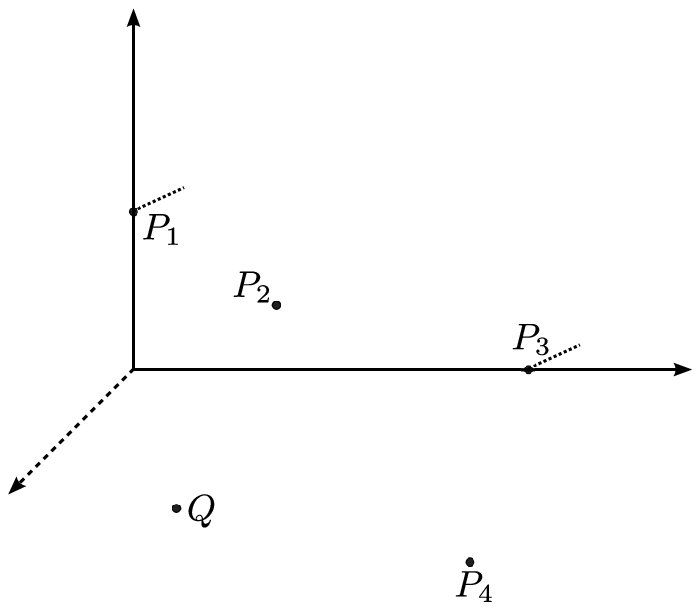}\\
\end{tabular}
\caption{Maslov index 0 tropical trees in $(X_\Sigma,  P_1, \ldots,  P_5)$.  The dashed edges are the distinguished outgoing edge.  Note that in the case of the tropical tree on the right, we could have selected any of the outgoing edges as the distinguished one.}
\end{figure}

\subsubsection{Tropical disks and trees}
In order to discuss the mirror symmetry relationship for $\mathbb{P}^2$, we'll need two objects which are closely related to tropical curves: tropical disks and trees.  Intuitively, tropical disks are fragments of a tropical curve broken at a vertex and are the tropical analogue of holomorphic disks, while tropical trees are tropical disks with the truncated edge extended to infinity.    

More formally, let $\overline{\Gamma}$ be a weighted, connected finite graph without bivalent vertices, with the additional choice of a univalent vertex $V_{out}$ adjacent to a unique edge $E_{out}$.  Let
$$\Gamma':=(\overline{\Gamma}\setminus\overline{\Gamma}_\infty^{[0]})\cup \{V_{out}\}\subseteq \overline{\Gamma}.$$
Suppose that $\Gamma'$ is a tree with one compact external edge and a number of non-compact external edges.  Then a \emph{parametrized d-pointed tropical disk} in $M_\mathbb{R}$ with domain $\Gamma'$ is: 
\begin{itemize}
\item A choice of inclusion $\{p_1,\ldots,p_d\}\hookrightarrow\Gamma^{[1]}_\infty \setminus \{E_{out}\}$, written $p_i\rightarrow E_{p_i}$.
\item A weight function $w:\Gamma'^{[1]}\rightarrow \mathbb{Z}_{\geq 0}$ with $w(E)=0$ if and only if $E=E_{p_i}$ for some $i$ and $w(E)=1$ for all other edges in $\Gamma'^{[1]}_\infty$.
\item A continuous map $h:\Gamma' \rightarrow M_\mathbb{R}$ satisfying the conditions for tropical curves, except that there is no balancing condition at the univalent vertex $V_{out}$.\\
\end{itemize}

An isomorphism of parametrized tropical disks between $h_1:(\Gamma'_1, p_1,\ldots, p_d)\rightarrow M_\mathbb{R}$ and $h_2:(\Gamma'_2, p_1,\ldots, p_d)\rightarrow M_\mathbb{R}$  is a homeomorphism $\Phi:\Gamma'_1\rightarrow \Gamma'_2$ respecting marked edges and weights, such that $h_1=h_2\circ\Phi$.  Just as with marked tropical curves, we refer to an equivalence class of parametrized marked tropical disks a \emph{marked tropical disk}.

\begin{definition}[Tropical disks in $(X_\Sigma, P_1, \ldots, P_k)$ with boundary $Q$]
A tropical disk in $(X_\Sigma,P_1, \ldots,  P_k)$ with boundary $Q$ is a $d$-pointed tropical disk $h:(\Gamma, p_1, \ldots, p_d)\rightarrow M_\mathbb{R}$ with $h(p_j)=P_{i_j}$ for some $1\leq i_1 <\ldots<i_d\leq k$, $h(V_{out})=Q$, and $h(E)$ is a translate of some $\rho\in \Sigma^{[1]}$ for each $E\in \Gamma_\infty^{[1]}$ with $w(E)=1$.  
\end{definition}
Multiplicity and degree can be defined for tropical disks as they were defined for tropical curves, neglecting the univalent vertex.
Continuing the analogy with holomorphic disks, given a d-pointed tropical disk $h$, we define its \emph{Maslov index} as
$$MI(h):=2(|\Delta(h)|-d).$$
There is a related tropical object of some importance, the \emph{tropical tree}.
Tropical trees are simply tropical disks where the outgoing edge $E_{out}$ is extended into unbounded edge.  The degree, multiplicity, and Maslov index are computed in the same way as was done with tropical disks, in each case ignoring the distinguished unbounded edge.  Tropical trees are important in this particular story because a Maslov index 2 tropical disk with boundary $Q$ can be decomposed as a ``stem" with truncated Maslov index 0 tropical trees sprouting out from it. This idea is the key to the relevance of so-called ``scattering diagrams" to the B-model of $\mathbb{P}^2$.  See Figure \ref{stems}.   
\begin{figure}
\sidecaption[t]
\includegraphics[width=0.4\textwidth]{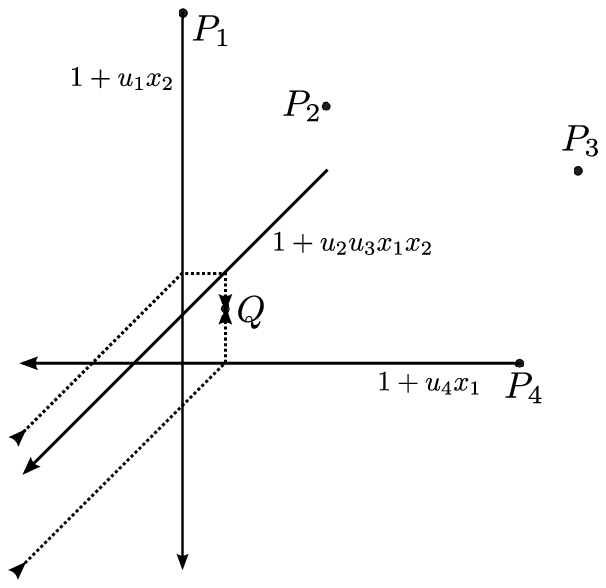}
\caption{``Stems" of Maslov index 2 tropical disks with boundary $Q$ along with the outgoing edges of their attached Maslov index 0 trees.  Find the Maslov index two tropical disks in Figures \ref{diskfig1} and \ref{diskfig2} corresponding to these stems.}
\label{stems}
\end{figure}

\section{Tropical curve counting}
\label{sec:enum}
\subsection{Moduli spaces of tropical curves}

\begin{definition} 
Given an element $\Delta\in T_\Sigma$, define 
$$\shM_{g,k}(\Sigma,\Delta)=\left.\left\{\begin{array}{l}
 \hbox{marked tropical curves in }X_\Sigma\hbox{ of genus }g\\
 \hbox{and degree }\Delta\hbox{ with }k\hbox{ markings}
\end{array} \right\}\right.$$
Furthermore if $\Delta=\sum_\rho d_\rho\rho$ we set 
$$|\Delta|=\sum_\rho d_\rho\in\ZZ.$$
\end{definition}

\begin{figure}[b]
\sidecaption
% For example, with the graphicx style use
\includegraphics[scale=.8]{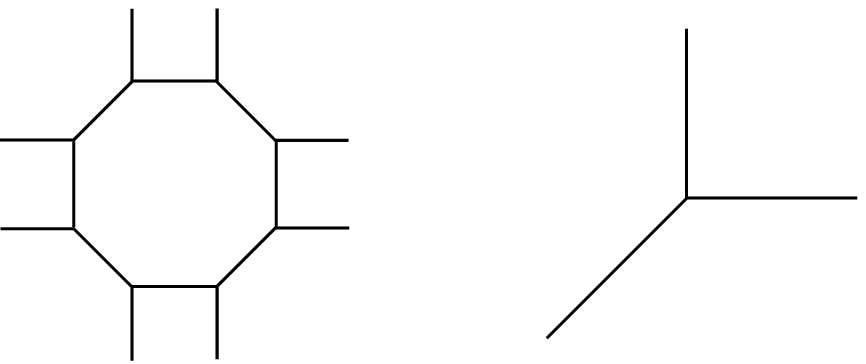}
%
% If no graphics program available, insert a blank space i.e. use
%\picplace{5cm}{2cm} % Give the correct figure height and width in cm
%
%\caption{If the width of the figure is less than 7.8 cm use the \texttt{sidecapion} command to flush the caption on the left side of the page. If the figure is positioned at the top of the page, align the sidecaption with the top of the figure -- to achieve this you simply need to use the optional argument \texttt{[t]} with the \texttt{sidecaption} command}
\caption{A tropical line in $\PP^2$ is uniquely determined by where its vertex is (right hand side). For the tropical curve of bi-degree $(2,2)$ in $\PP^1\times\PP^1$ of the combinatorial type shown on the left, there are, up to translation, $6$ further moduli by varying the lengths of the $8$ bounded edges.}
\label{fig:two-trop-curves}       % Give a unique label
\end{figure}

\begin{example}  \label{example-MPTCs}
\begin{enumerate}
\item
Let $\Sigma$ be the fan of $\PP^2$, so $M\cong\ZZ^2$. The rays are generated by $\rho_1=(1,0)$, $\rho_2=(0,1)$ and $\rho_3=(-1,-1)$. 
Let $\Delta=\rho_1+\rho_2+\rho_3$ then $\shM_{0,0}(\Sigma,\Delta)=M_\RR$ as the map $h$ is uniquely determined by where the trivalent vertex of $\Gamma$ goes and there is no restriction on where to map it.
In fact, in general for any $\Sigma,\Delta$ we have that $M_\RR$ acts freely on $\shM_{g,n}(\Sigma,\Delta)$ by translation.
\item
Let $\Sigma$ be the fan of $\PP^1\times\PP^1$. 
The rays are generated by $\rho_1=(1,0)$, $\rho_2=(-1,0)$, $\rho_3=(0,1)$ and $\rho_4=(0,-1)$. 
Set $\Delta=2\rho_1+2\rho_2+2\rho_3+2\rho_4$. 
Consider the tropical curve on the left in Fig~\ref{fig:two-trop-curves} (the graph $\Gamma$ is determined from the image of $h$ for given $P_i$).
Let us fix the combinatorial type of $h$, i.e. the weighted graph $\Gamma$ and the rational slopes of the edges of the image of $h$
and let $\shM^{[h]}_{1,0}(\Sigma,\Delta)$ denote the subset of $\shM_{1,0}(\Sigma,\Delta)$ of MPTCs of combinatorial type $h$.
Up to translation, a curve in $\shM^{[h]}_{1,0}(\Sigma,\Delta)$ is uniquely determined be the length of its compact edges of which there are $8$. 
However the lengths cannot vary freely because their union needs to be a closed cycle. This imposes two conditions, one for each coordinate of $M_\RR$.
Let $I=\{(1,1), (1,0), (1,-1), (0,-1), (-1,-1), (-1,0), (-1,1), (0,1)\}$ be the set of directions of the bounded edges.
We then find that the set $\shM_{1,0}(\Sigma,\Delta)$ can be identified with
$$M_\RR\times\left\{\phi\in \operatorname{Map}(I,\RR_{>0}) \left| 0=\sum_{v\in I} \phi(v)v\right.\right\}.$$
Note that $\shM_{1,0}(\Sigma,\Delta)$ is $8$-dimensional. 
This coincides with the (complex) dimension of the parameter-space of algebraic curves of bi-degree $(2,2)$ in $\PP^1\times\PP^1$ (these are elliptic curves). This is no coincidence as we will see later.
\end{enumerate}
\end{example}

\begin{lemma} 
\label{finiteness-of-comb-types}
When $\dim M_\RR=2$, the set of combinatorial types of tropical curves in $X_\Sigma$ of fixed genus, markings and degree is finite.
\end{lemma}
\begin{proof} 
It suffices to show that the set of combinatorial types of unmarked curves is finite as there is only a finite set of choices for placing the markings.
Given one such curve $h$, one can construct a piecewise linear convex function $M_\RR\ra\RR$ whose locus of non-linearity coincides with $h$. 
The bending at an edge $h(E)$ is $w(E)$ and the balancing condition guarantees that this gives a globally compatible function. 
This function thus determines a Newton polytope in the dual space of $M_\RR$ together with a triangulation. 
This is in fact a lattice polytope, so the set of lattice triangulations is finite. 
Furthermore, the Newton polytope only depends on the degree of $h$, so the set of combinatorial types of unmarked curves is identified with the set of triangulations of the Newton polytope and this is known to be finite.
\end{proof}

A priori $\shM_{g,k}(\Sigma,\Delta)$ is merely a set. However, the natural identifications in the following proposition furnish $\shM_{g,k}(\Sigma,\Delta)$ with a piecewise linear structure.
Given $h\in\shM_{g,k}(\Sigma,\Delta)$, let $\shM_{g,k}^{[h]}(\Sigma,\Delta)$ denote the subset of $\shM_{g,k}(\Sigma,\Delta)$ of all MPTC with the same \emph{combinatorial type} as $h$, i.e. the same weighted graph $\Gamma$ and the same rational slopes of $h(E)$ for each edge $E\subset\Gamma$ with $h(E)\neq 0$.
\begin{proposition}[shape of $\shM_{g,k}$]
\label{shapeMgk}
\begin{enumerate}
\item $\shM_{g,k}(\Sigma,\Delta)=\coprod_h \shM_{g,k}^{[h]}(\Sigma,\Delta)$ where the disjoint union is over all combinatorial types.
\item $\shM_{g,k}^{[h]}(\Sigma,\Delta)$ is naturally identified with the interior of a polyhedron.
\end{enumerate}
\end{proposition}
\begin{proof} The first statement is a tautology. The proof of second works along the lines of Example~\ref{example-MPTCs}-2., i.e. let $I$ denote the set of slope vectors of the bounded edges of $h(\Gamma)$. Up to translations by elements of $M_\RR$, we identify $\shM_{g,k}(\Sigma,\Delta)$ with the subset of $\operatorname{Map}(I,\RR_{>0})$ cut out by $m$ linear equations, one for each cycle in $\Gamma$.
\end{proof}

More can be said when we restrict to genus zero curves. 
Set $\Gamma^{[0]}=\{V\in\Gamma\hbox{ is a vertex }\}$. Since univalent vertices were removed and there are no bivalent vertices in $\Gamma$ each vertex of $\Gamma$ has valency at least three.
We define the overvalency of $\Gamma$ by 
$$\ov(\Gamma)=\sum_{V\in\Gamma\hbox{ \scriptsize is a vertex}} \operatorname{valency}(V)-3.$$
It vanishes if and only if each vertex of $\Gamma$ has valency three. 
\begin{definition}
A marked parametrized tropical curve $h$ is called \emph{simple} if $h$ is injective on vertices, unmarked unbounded edges have weight one and each vertex has non-zero multiplicity (in particular the overvalence vanishes).
\end{definition}

\begin{proposition}[shape of $\shM_{0,k}(\Sigma,\Delta)$]
\label{shapeM0k}
\begin{enumerate}
\item $\shM_{0,k}^{[h]}(\Sigma,\Delta)\cong M_\RR\times \RR_{>0}^{e+k-3-\ov(\Gamma)}$ where $e$ is the number of unbounded unmarked edges of $\Gamma$.
\item Assume now $n=2$, i.e. $M_\RR\cong\RR^2$. Given $P_1,...,P_{|\Delta|-1}\in M_\RR$ in general position, we have that
$$\{h\in\shM_{0,|\Delta|-1}(\Sigma,\Delta)\,|\, h(x_i)=P_i\}$$
is a finite set of simple curves of different combinatorial types.
\end{enumerate}
\end{proposition}
\begin{proof} 
By the proof of part 2. of Prop.~\ref{shapeMgk}, we need to show that the number of bounded edges coincides with 
$e+k-3-\ov(\Gamma)$. 
Set $\Gamma^{[0]}=\{V\in\Gamma\hbox{ is a vertex }\}$, we have that
\begin{equation}
\label{shapeM0k-eq1}
\begin{array}{rcl}
3|\Gamma^{[0]}|+\ov(\Gamma)&=&\sum_{V\in\Gamma^{[0]}} \operatorname{valency}(V)\\
 &=& 2\cdot\hbox{(number of bounded edges)}+\hbox{(number of unbounded edges)}
\end{array}
\end{equation}
On the other hand for the Euler characteristic of $\Gamma$ we find
\begin{equation}
\label{shapeM0k-eq2}
1-g=\chi(\Gamma)=|\Gamma^{[0]}|-\hbox{(number of bounded edges)}.
\end{equation}
Eliminating $|\Gamma^{[0]}|$ together with noting that $e+k$ is the number of unbounded edges yields
$$ \hbox{number of bounded edges} = e+k+3g-3-\ov(\Gamma).$$
Inserting $g=0$ gives the first assertion.
To prove the second assertion, note that each point imposes a 2-dimensional condition and all conditions are independent by the general position assumption. 
For $\shM_{0,k}$ to be non-empty, by a dimension count via the first assertion and $k=|\Delta|-1$, we need to have
$$2+ e+|\Delta|-4-\ov(\Gamma)-2(|\Delta|-1)\ge 0.$$
Note that $e\le|\Delta|$, so the inequality holds if and only if it is an equality and $\ov(\Gamma)=0$ and $e=|\Delta|$. In this case, $h$ is trivalent with all unbounded edges of weight one. By the general position assumption, $h$ is injective on vertices and if there was a vertex of multiplicity zero, all attached edges would be collinear and so one could move this vertex contradicting zero-dimensionality of the set of solutions. 
Thus, every curve is simple. 
They are of different types by part 1 of Prop.~\ref{shapeMgk}. 
The finiteness of the set of combinatorial types is Lemma~\ref{finiteness-of-comb-types}.
\end{proof}

In analogy to usual Gromov-Witten invariants, we may define the evaluation map
$$\op{ev}:\shM^{[h]}_{g,k}(\Sigma,\Delta)\ra M_\RR^k,\qquad h\mapsto (h(x_1),...,h(x_k))$$
which is in fact an affine linear map: it maps a set of polyhedra affine linearly to a vector space. 
The set of curves going through a set of points $P_1,...,P_k$ is then
$$\op{ev}^{-1}(P_1,...,P_k)=\{h\in\shM_{g,k}(\Sigma,\Delta)\,|\, h(x_i)=P_i\}$$
By the previous proposition, this set is finite for $k=|\Delta|-1$, $g=0$ and one may wonder how its size changes if one varies $P_1,...,P_k$.
If one counts weighted by the multiplicity of the combinatorial type, we will see later that the count is independent of the position of the points as long as the points are in general position. 
This means that if we take a path from one positioning of the $P_i$ to another positioning and at some point along the path one combinatorial type ceases to have a solution for the given points, another combinatorial type takes over!
Assuming this result, the following definition is well-defined (independent of the $P_i$).

\begin{definition} 
We define the number of rational tropical curves of degree $\Delta$ in $X_\Sigma$ as
$$N^{0,\trop}_{\Delta,\Sigma}=\sum_{\begin{array}{c}{h\in\shM_{0,|\Delta|-1}(\Sigma,\Delta)}\\{h(x_i)=P_i}\end{array}} \op{Mult}(h)$$
\end{definition}

\begin{definition} 
\label{def-numbers-for-count}
Similarly and classically, we define the number of rational holomorphic curves of degree $\Delta$ in $X_\Sigma$ as
$$N^{0,\hol}_{\Delta,\Sigma}=\left|\left\{f\in\overline\shM_{0,k}(X_\Sigma,\Delta) \left| \begin{array}{c} f:(C,x_1,...,x_k)\ra X_\Sigma \hbox{ is a torically transverse}\\ \hbox{algebraic curve with }f(x_i)=Q_i\end{array}\right.\right\}\right|$$
where $k=|\Delta|-1$ and $Q_1,...,Q_k$ are points in general position in $X_\Sigma$.
\end{definition}

The following result in particular gives the well-definedness of $N^{0,\trop}_{\Delta,\Sigma}$. 
\begin{theorem} 
\label{maintheorem}
If $\dim M_\RR=2$ and $g=0$ then
$$N^{g,\hol}_{\Delta,\Sigma}= N^{g,\trop}_{\Delta,\Sigma}$$
\end{theorem}
The theorem is the overlap of a result by Mikhalkin who proved the statement for any genus $g$ when $\dim M_\RR=2$ and Siebert-Nishinou \cite{NS06} who prove it for $g=0$ in any dimension.

\subsection{Finding all rational tropical curves through eight points in the plane}
We want to discuss in this section an extended example elucidating Thm.~\ref{maintheorem}. 
It is a famous fact that there are precisely $12$ rational curves of degree three going through $8$ generically placed points in the projective plane.
Dropping rationality, there is a one-parameter family of degree three curves going through $8$ points. The general member of this pencil is an elliptic curve but $12$ members are rational nodal curves.
So if $\Sigma_{\PP^2}$ is the fan of $\PP^2$ and we fix the degree as $\Delta_3=3\omega_1+3\omega_2+3\omega_3$ for $\omega_i$ the generators of the rays in the fan, then we have classically
$$N^{0,\hol}_{\Delta_3,\Sigma_{\PP^2}}=12$$
and by Thm.~\ref{maintheorem} we expect to find also $12$ tropical genus zero curves (counted with multiplicity) through $8$ general points in $\RR^2$.
We reduce the complexity of the problem by a slight modification. 
Pick any three of the eight points and consider the toric structure on $\PP^2$ where the open torus is the complement of the three lines going through pairs out of the three points.
The blow-up of $\PP^2$ in the three points can be realized torically, i.e. there is a subdivision $\Sigma$ of the fan $\Sigma_{\PP^2}$ where each of the three maximal cones is subdivided into two standard cones and the toric variety corresponding to the subdivision is the blow-up $X_\Sigma=\Bl_{3\hbox{\scriptsize pt}}\PP^2$. The resulting fan is shown in Fig.~\ref{fig:fan}.
\begin{figure}[t]
\sidecaption[t]
\includegraphics[width=0.63\textwidth]{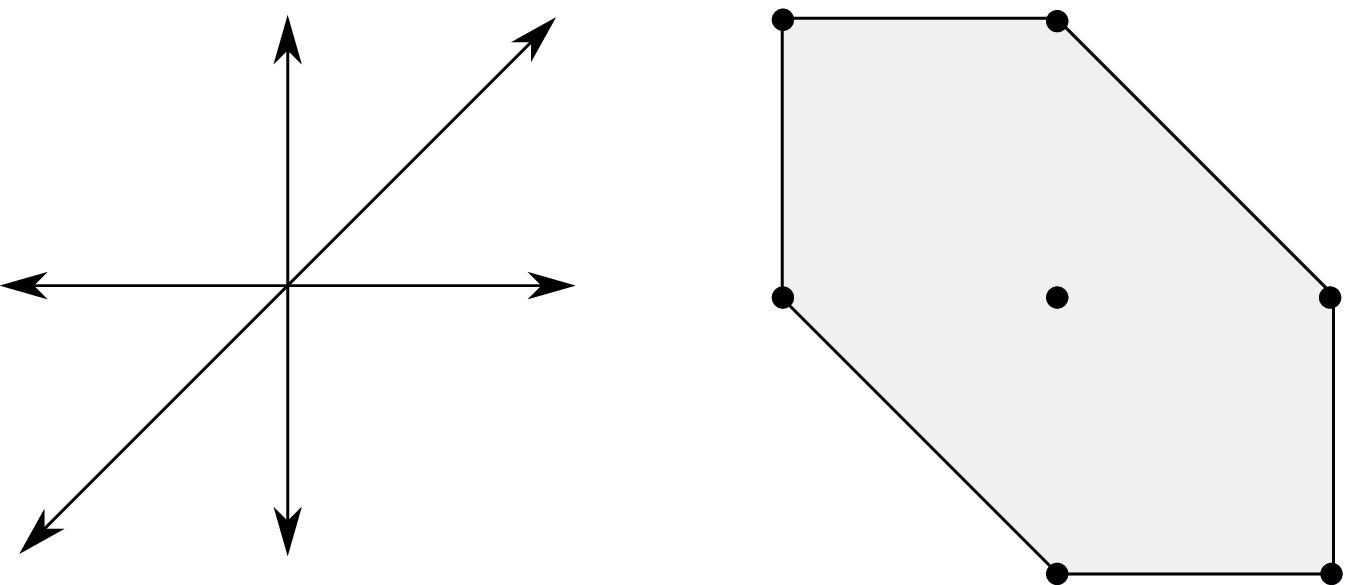}
\caption{Fan of the blow-up of $\PP^2$ in three points and the Newton polytope of its anti-canonical divisor.}
\label{fig:fan}      
\end{figure}
It is the normal fan to a hexagon (in the dual space) depicted on the right.
The anti-canonical degree of $X_\Sigma$ is
$$\Delta=\rho_1+...+\rho_6$$
where the $\rho_i$ denote the six generators of the rays in $\Sigma$.
The combinatorial problem is now to find all tropical genus zero curves through five general points in $\RR^2$ of degree $\Delta$. 
\begin{figure}[t]
\sidecaption[t]
\includegraphics[width=0.63\textwidth]{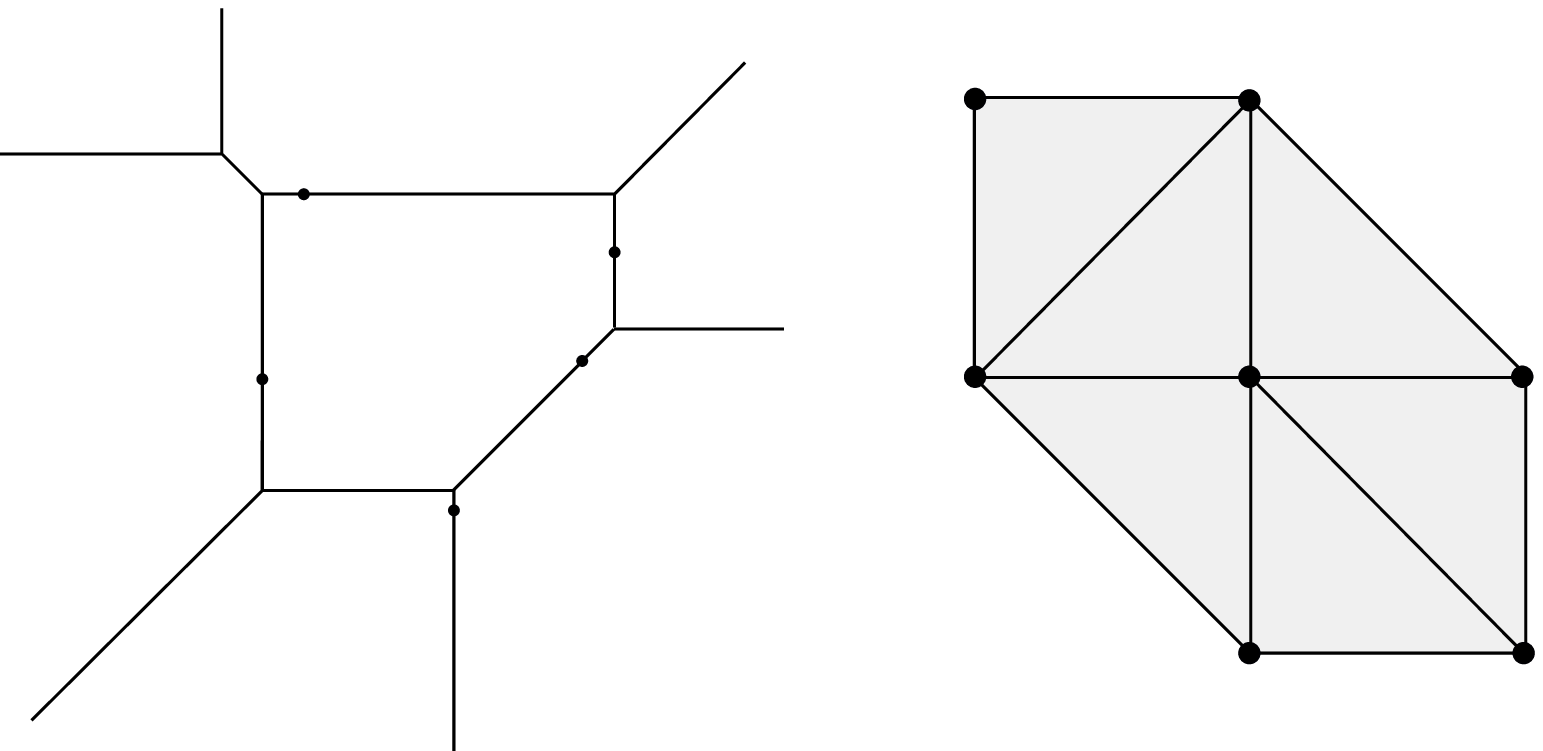}
\caption{A tropical elliptic curve in $X_\Sigma$ of degree $\Delta$ containing five given points and the subdivision of the Newton polytope corresponding to its combinatorial type.}
\label{fig:ell-curve}      
\end{figure}
Given any $5$ points, just by inspection it is quite hard to come up with just a single such tropical curve. It is easier though to find a genus one curve through these points as such tropical curves come in a one-parameter family just as their holomorphic analogues. Fig.~\ref{fig:ell-curve} depicts such a tropical genus one curve. The degree of freedom can be seen by the fact that the upper left branch is free to move out diagonally to the upper left.
There is actually a tropical version of the pencil of elliptic curves as the set of tropical genus one curves going through the five points. 
We are going to construct it in the following.

\subsubsection{A tropical pencil of elliptic curves}
\label{section-sweep-pencil}
A side effect of the construction of the pencil is going to be that we also obtain all rational curves going through the five points as those are members of the pencil, so we will find them on the way.
Note that a tropical curve of degree $\Delta$ is uniquely determined (up to adding a constant) by the piecewise linear convex function $\RR^2\ra\RR$ whose locus of non-linearity is the tropical curve.
Any such function has the following shape
$$\varphi:\RR^2\ra\RR,\quad v\mapsto \max\left\{\langle v,m\rangle+a_m\ |\ m \hbox{ is a lattice points in the Newton polytope}\right\}$$
for some coefficients $a_m\in\RR$. As there are seven coefficients, all piecewise linear convex functions naturally give a convex subset in $\RR^7$.
Requiring that the locus of non-linearity of such a function contains a certain point imposes a one-dimensional condition on the function, so by the general positioning of the five points, we expect that there is a two-dimensional subset of $\RR^7$ that gives the pencil. There is one excess dimension over the set of tropical curves as a function $\phi$ gives the same tropical curve as $\phi+a$ for any $a\in\RR$, so we could instead work in $\RR^7/\RR(1,...,1)\cong\RR^6$ to obtain the pencil as a piecewise linear one-dimensional subset. We will see that this subset in our example has the shape depicted in Fig.~\ref{fig:pencil}.

\begin{figure}[t]
\sidecaption[t]
\includegraphics[width=0.5\textwidth]{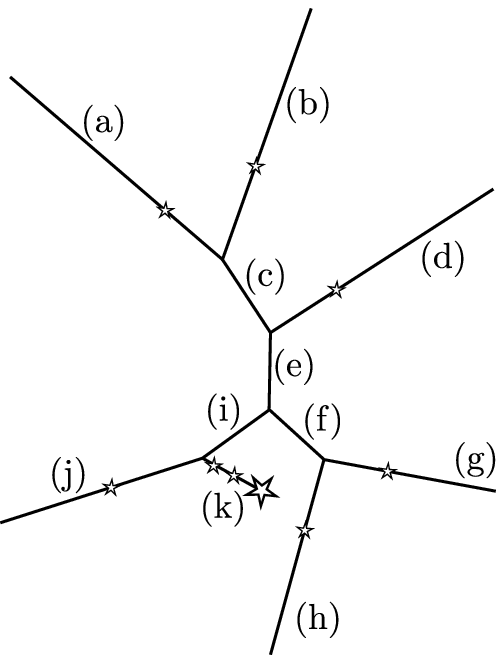}
\caption{Pencil of tropical anti-canonical curves containing $5$ general points in a del Pezzo surface of degree $6$ (blow-up of $\PP^2$ in three points). The rational nodal curves in this pencil are marked by a star. The large star is a genus zero curve of multiplicity four so that the sum of all rational curves with multiplicities adds up to $12$. The labels of the edges of the pencil refer to the labelling of the steps in the construction of family of tropical curves in Fig.~\ref{fig:pencil1}, Fig.~\ref{fig:pencil2}}.
\label{fig:pencil}    
\end{figure}

\begin{figure}[t]
\sidecaption[t]
\begin{tabular}{cc}
(a)& (b)\\
\includegraphics[width=0.45\textwidth]{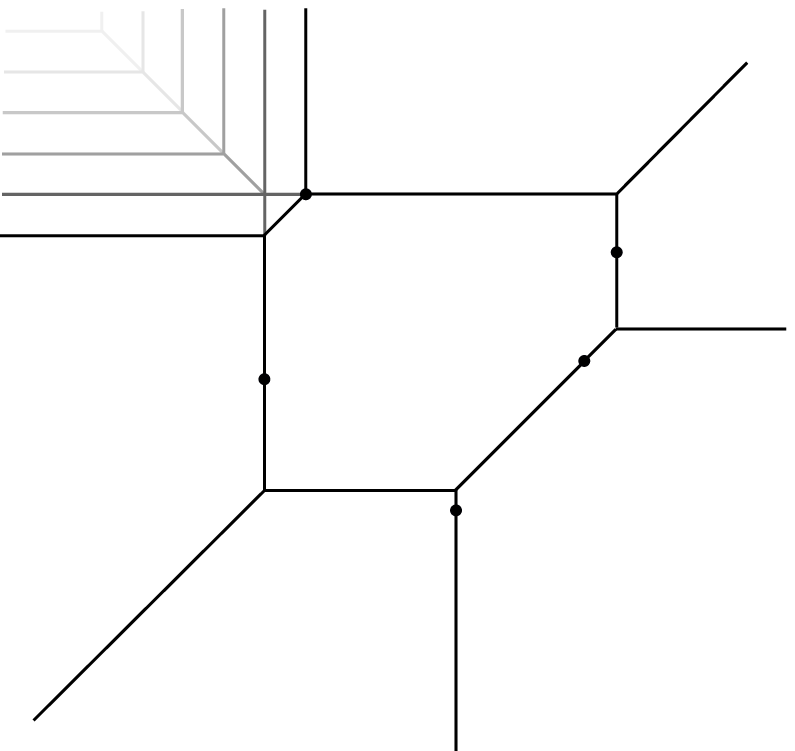} & \includegraphics[width=0.45\textwidth]{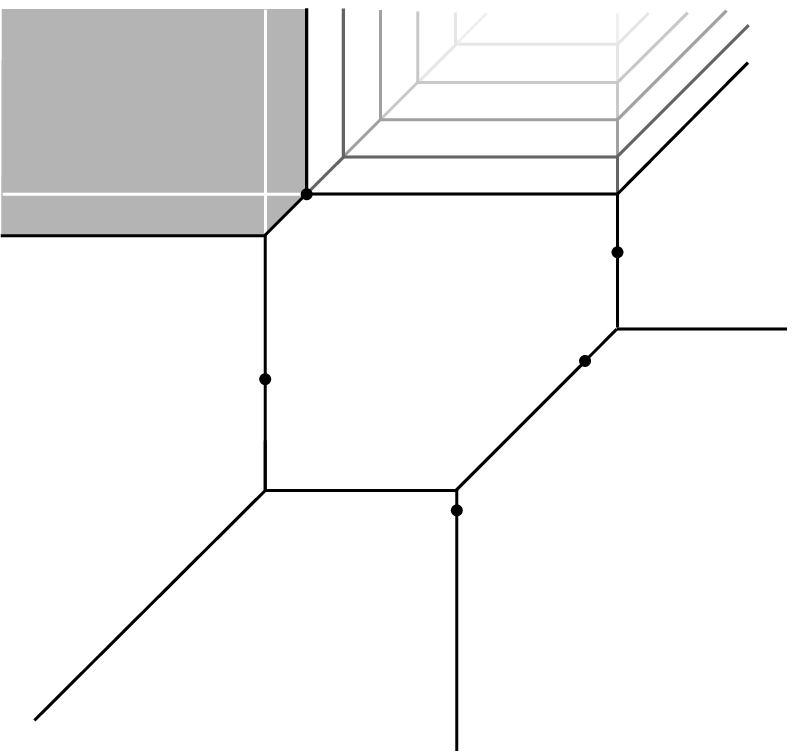}\\
(c)& (d)\\
\raisebox{-1\height}{\includegraphics[width=0.45\textwidth]{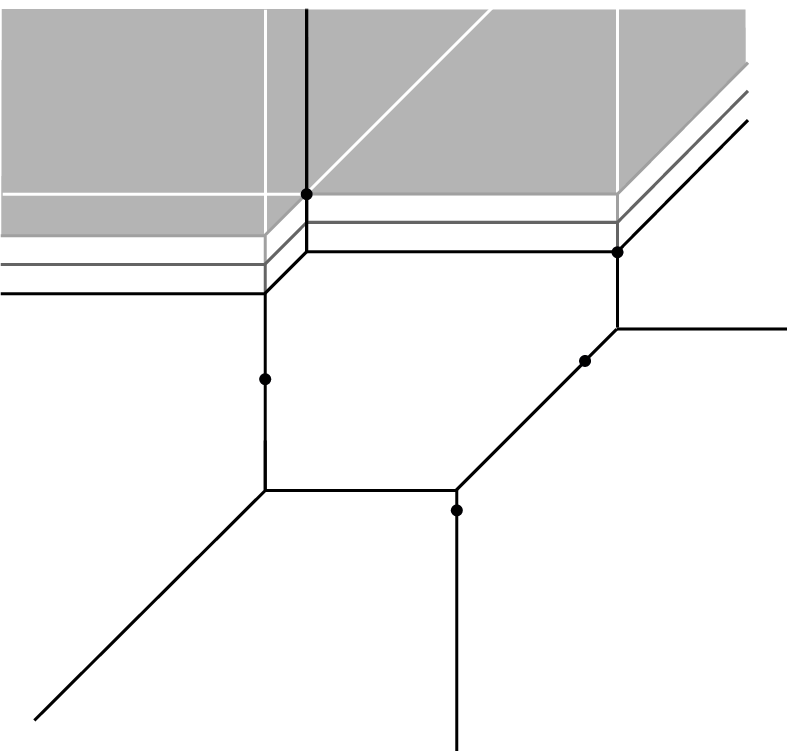}} & \raisebox{-1\height}{\includegraphics[width=0.45\textwidth]{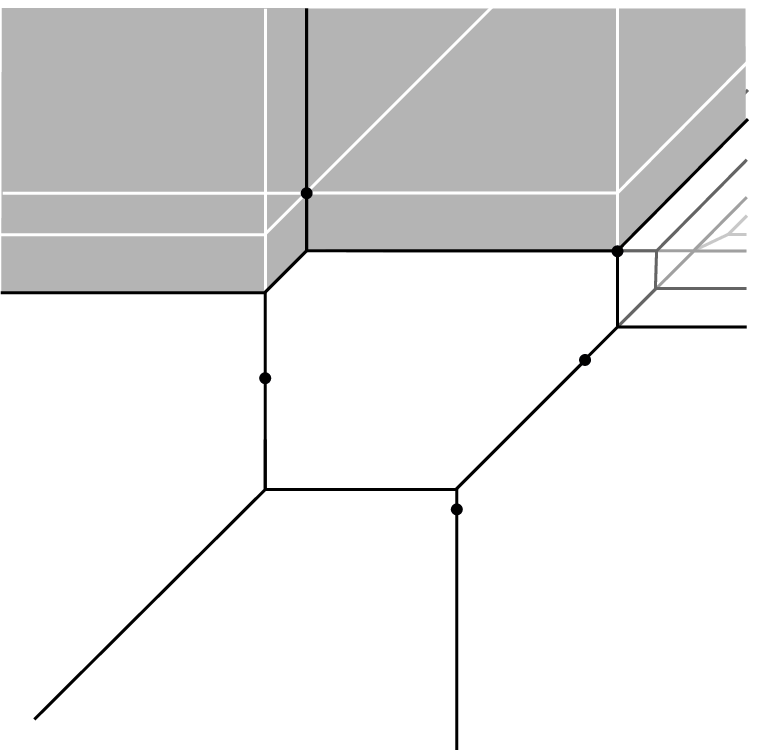}}\\
(e)& (f)\\
\raisebox{-1\height}{\includegraphics[width=0.45\textwidth]{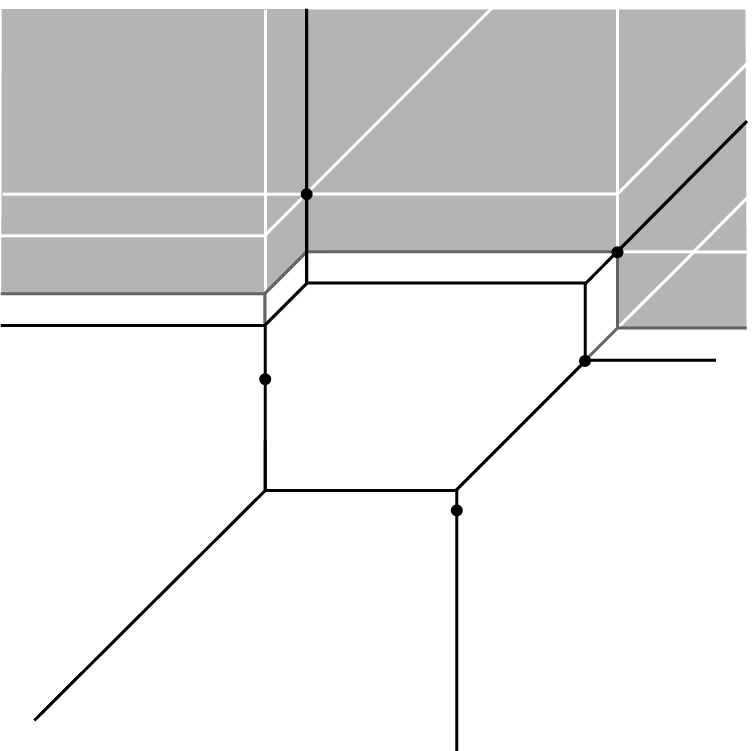}} & \raisebox{-1\height}{\includegraphics[width=0.45\textwidth]{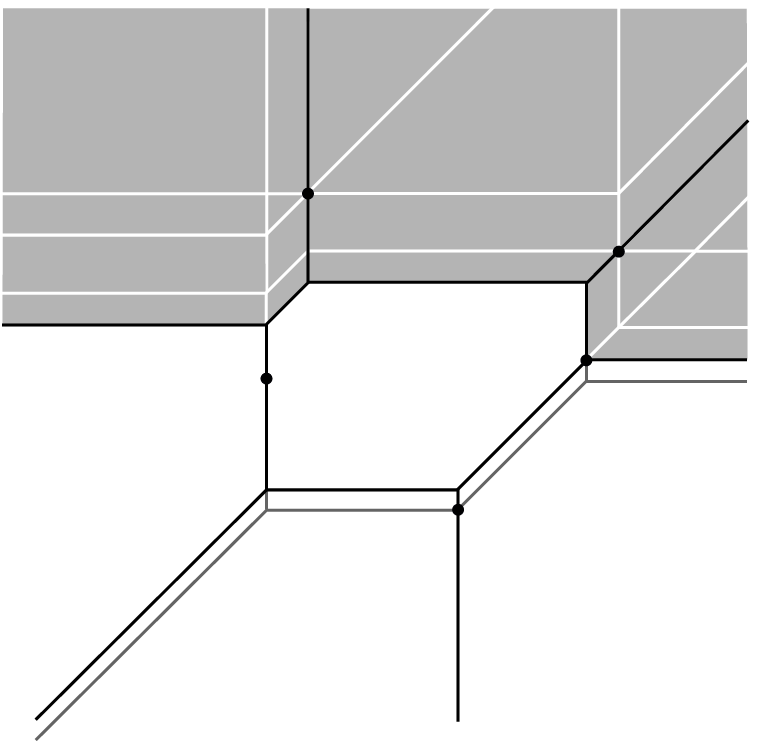}}
\end{tabular}
\caption{The pencil of elliptic curves sweeps the plane. Whenever a marked points becomes a vertex of the tropical curve, there are two possibilities to move on in the pencil leading to the various branches in Fig.~\ref{fig:pencil}. We depict here the tropical curves of the the first $6$ edges in the pencil}.
\label{fig:pencil1}    
\end{figure}

\begin{figure}[t]
\sidecaption[t]
\begin{tabular}{cc}
(g)& (h)\\
\includegraphics[width=0.45\textwidth]{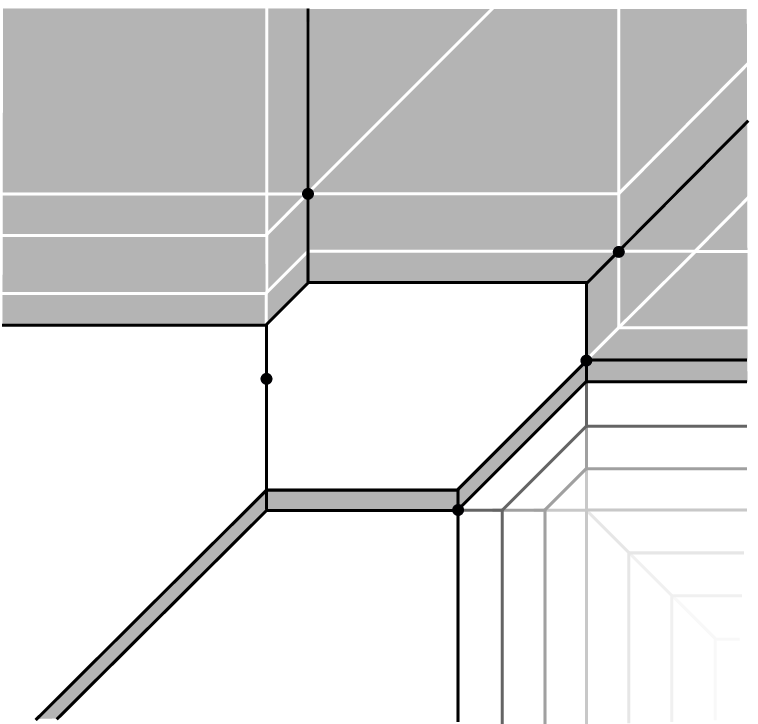} & \includegraphics[width=0.45\textwidth]{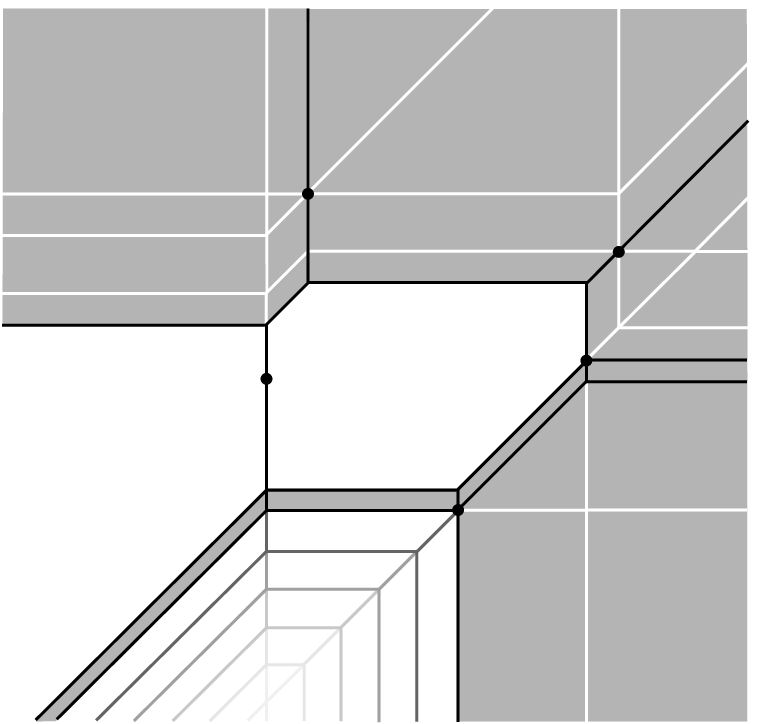}\\
(i)& (j)\\
\includegraphics[width=0.45\textwidth]{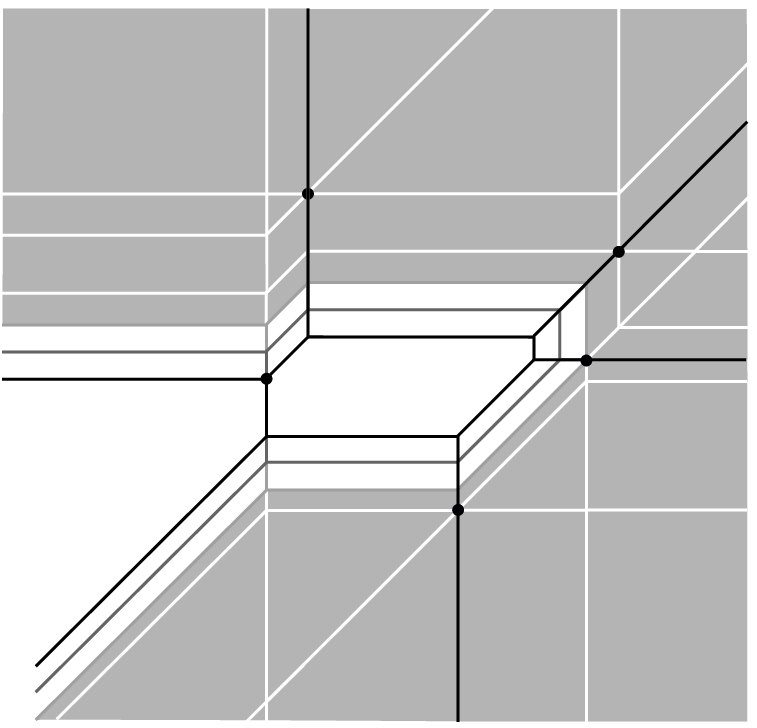} & \includegraphics[width=0.45\textwidth]{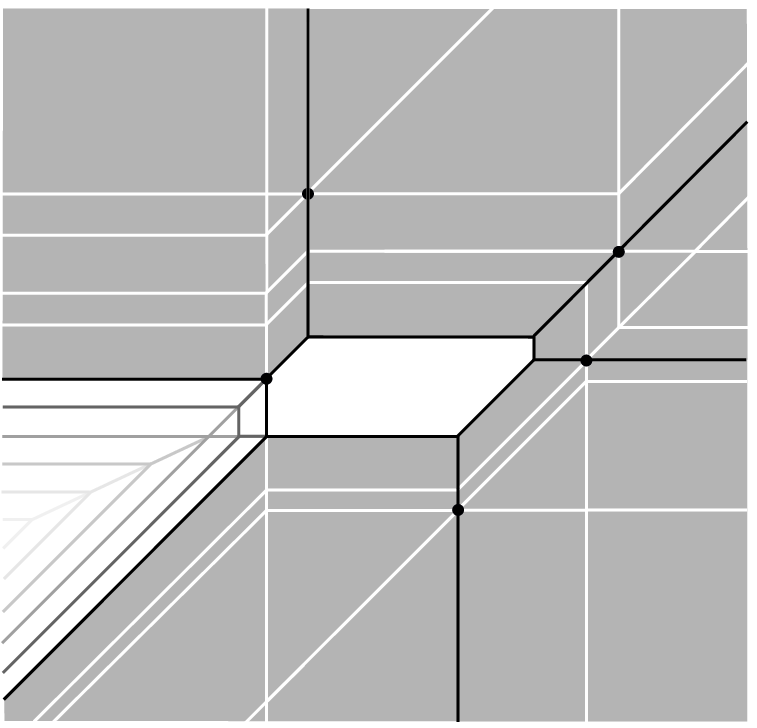}\\
(k)& (l)\\
\includegraphics[width=0.45\textwidth]{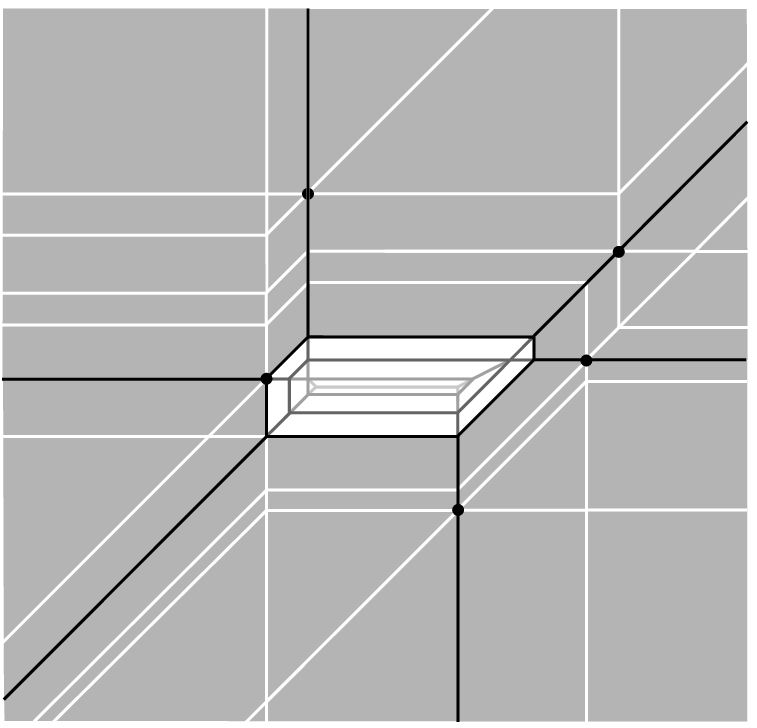} & \includegraphics[width=0.45\textwidth]{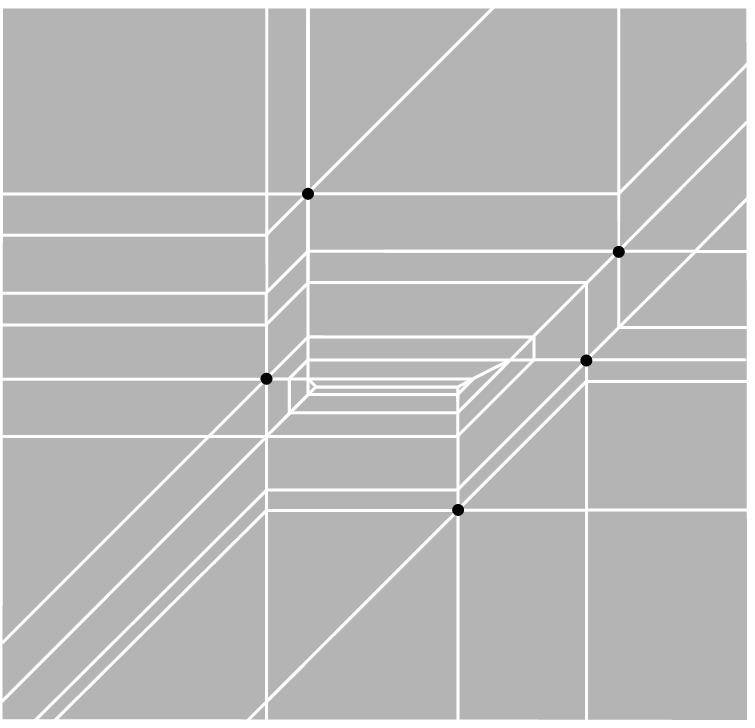}
\end{tabular}
\caption{Complementing Fig.~\ref{fig:pencil1}, we depict the tropical elliptic curves for the remaining edges in the pencil. Picture (l) shows the union of all rational curves in the pencil}.
\label{fig:pencil2}   
\end{figure}
Indeed, the movable upper left branch of our tropical elliptic curve of Fig.~\ref{fig:ell-curve} moves as shown in picture (a) of Fig.~\ref{fig:pencil1}.
It accommodates a nodal rational curve that shows as a tropical curve with a four-valent vertex. 
In fact as a marked parametrized tropical curve, the four-valent point is not actually a vertex, i.e. it is not the image of a vertex of the graph under the immersion $h$. 
The nodal curve is indicated by a star in Fig.~\ref{fig:pencil}.
Moving past the nodal curve, our elliptic tropical curve eventually attains the property that one of its vertices coincides with one of the $5$ fixed points. At this stage we have swept through the upper left section of $\RR^2$ with tropical curves parametrized by the branch of the pencil in Fig.~\ref{fig:pencil} marked by (a) and we reached a vertex of the pencil. 
From the vertex there are two directions to move on in the pencil corresponding to the two regions next to the marked point in the complement of the vertex-curve. In step (b), we move into the region to the upper right where we find another nodal curve. We carry on like this moving through further edges of the pencil. The steps (a)-(f) are depicted in Fig~\ref{fig:pencil1}, the steps (g)-(k) are depicted in Fig~\ref{fig:pencil2}. 
\FloatBarrier
The last step (k) in which the tropical curves sweep the central region is somewhat special: it gives the edge of the pencil with a univalent vertex. Not only does this edge contain two nodal curves in its interior, 
furthermore, the univalent vertex is also a rational curve of multiplicity four as it has two vertices each of multiplicity two. In total, we have found 8 nodal curves of multiplicity one and another rational curve of multiplicity $4$ adding up to the expected count:
$$N^{0,\trop}_{\Delta,\Sigma}=1+1+1+1+1+1+1+1+4=12.$$
Finally picture (l) in Fig~\ref{fig:pencil2} shows the union of all rational curves which gives a polyhedral subdivision of $\RR^2$ in which the fixed points are vertices.

\subsubsection{Is it possible to find twelve tropical curves of multiplicity one?}
One may wonder whether it is necessary to have a tropical curve of higher multiplicity in the pencil or whether there exists a configuration of $12$ multiplicity one curves going through some other positioning of the $5$ fixed points. 
From the experience of our construction of the pencil, one might get the impression that no matter where we place the $5$ points there should always be some region in the middle (in the cycle that gives the genus of the elliptic curve) that needs to be swept by the pencil leading to a univalent vertex of the pencil. This vertex is necessarily not an elliptic curve and most likely of higher multiplicity. 
While this is a hand-waving argument, there is a rigorous proof for the non-existence of a configuration of $12$ curves that has been known to real tropical geometers like Ilia Itenberg and Grigory Mikhalkin. It makes use of the Welschinger invariant. Recall the definition of the Mikhalkin multiplicity from Def.~\ref{def-mik-mult}.
We take from \cite[Def. 7.19]{mik} the following.
\begin{definition}[Welschinger multiplicity] 
\label{def-welsch-mult}
Let $h:(\Gamma,x_1, \ldots, x_n)\rightarrow M_\mathbb{R}$ be a simple marked parametrized tropical curve with $\dim M_\mathbb{R}=2$. 
For $V\in\Gamma$ a vertex, we define
$$
\Mult_V^{\RR,W}(h)=\left\{
\begin{array}{ccl} 
(-1)^{\frac{\Mult_V(h)-1}{2}}&\ \ &\hbox{if $\Mult_V(h)$ is odd}\\[2mm]
0 &&\hbox{otherwise}
\end{array}\right.
$$
and 
$$
\Mult^{\RR,W}(h):=\prod_{V\in \Gamma^{[0]}} \Mult^{\RR,W}_V(h).
$$
\end{definition}

\begin{definition}[Tropical Welschinger invariant]
\label{def-welsch-inv}
Let $\Delta$ be a degree for a smooth toric surface $\Sigma$, in particular $\dim M_\RR=2$. Set $k=|\Delta|-1$ and let $P_1,...,P_{k}\in M_\RR$ points in general position. We define the tropical Welschinger invariant
$$W^\trop(\Sigma, \Delta, P_1,...,P_k)=\sum_{h} \Mult^{\RR,W}(h)$$
where the sum ist over all rational tropical curves of degree $\Delta$ meeting the $P_i$, i.e. over
$$\{h\in\shM_{0,k}(\Sigma,\Delta)\,|\, h(x_i)=P_i\}.$$
\end{definition}
The tropical Welschinger invariant draws its significance from the following theorem.
\begin{theorem}{\bf (Mikhalkin \cite[Thm.~6]{mik}, Welschinger \cite{welsch}, cf. \cite[Thm. 3.1]{Sh06})}
Assume the setup of Def.~\ref{def-welsch-inv}. The number $W^\trop(\Sigma, \Delta)=W^\trop(\Sigma, \Delta, P_1,...,P_k)$ is independent of the position of $P_1,...,P_k$ and gives a lower bound on the number of real curves of degree $\Delta$ passing through $k$ real points in the corresponding toric surface over $\RR$.
\end{theorem}

%\begin{theorem} 
%\label{thm-shustin}
%Let $X_\Sigma$ be a toric Del Pezzo surface given by a fan $\Sigma$, let $\Delta\in T_\Sigma$ be the degree of the standard toric anti-canonical divisor and set $r=|\Delta|$.
%Let $x_1,..,x_r\in X_\Sigma$ be real points in general position, i.e. fixed under the anti-holomorphic involution. The Welschinger invariant $W_r(\Sigma, \Delta)$ satisfies the following equality
%$$W_r(\Sigma, \Delta) = \sum_A (-1)^{a}\prod_{\tau \hbox{ \scriptsize even}}\frac{|\tau|}2$$
%where the sum is over all rational tropical curves of degree $\Delta$ containing $x_1,...,x_r$ that do not have an edge of non-zero even weight.
%Each of these curves has a combinatorial type given by a subdivision of the Newton polytope and $a$ is the number of integral points that lie in the interior of a triangle. 
%The product is over all triangles $\tau$ of even lattice area.
%\end{theorem}

Most interesting for us is the property of the Welschinger invariant to be independent of the position of the points. Let us apply this to the toric del Pezzo of degree $6$ that we studied in the previous sections. We can readily compute the Welschinger invariant from our findings of rational curves via Def.~\ref{def-welsch-mult} and it yields
$$W^\trop(\Sigma, \Delta)= 1+1+1+1+1+1+1+1+0=8.$$
If there was another configuration of the $5$ points for which we had $12$ rational tropical curves of multiplicity one going through them, the calculation for the Welschinger invariant would read 
$$W^\trop(\Sigma, \Delta) = 1+1+1+1+1+1+1+1+1+1+1+1=12$$
however this would lead to a contradiction to the previous calculation as the invariant doesn't depend on the configuration of points we choose to compute it from. 
Knowing now that 12 curves are impossible, we can ask which other findings of curves would give the correct Welschinger invariant of $8$. 
\begin{exercise} 
\begin{enumerate}
\item By going through the possible regular triangulations of the Newton polytope, check that rational tropical curves of degree $\Delta$ can have Mikhalkin multiplicity $1,3,4$. (Note that there is a triangulation featuring only one area two triangle but this triangulation is not regular.)
\item Check that we have the following table on contributions of a rational tropical curve to the invariants. 
\begin{center}
\bgroup
\def\arraystretch{1.2}%  1 is the default, change whatever you need
\begin{tabular}{l|c|c|c|c}
Multiplicity (i.e. contribution to $N^{0,\trop}_{\Delta,\Sigma})$ \,&\,1\,&\,3\,&\,4\,\\
\hline
contribution to $W^\trop(\Sigma, \Delta)$ & 1 &-1&0
\end{tabular}
\egroup
\end{center}
\item Deduce that the conditions $N^\trop(\Delta,\Sigma)=12$ and $W^\trop(\Sigma, \Delta)=8$ allow for exactly one further possible configuration of rational tropical curves through $5$ points. It features $10$ curves and the multiplicities are respectively
$$1+1+1+1+1+1+1+1+1+3.$$
\item Verify the existence of this configuration by using the tropical pencil construction of the previous section: Start with the multiplicity three curve as the univalent vertex of the pencil and start sweeping from there.
\end{enumerate}
\end{exercise}

\section{From tropical curves to algebraic curves and back}
\label{section:NS}
We are going to sketch the proof of Theorem~\ref{maintheorem}. 
This will be similar to the exposition in \cite{kan}, while the original is \cite{NS06}.
The proof is a matching of the following sets
$$\left\{\hbox{tropical curves}\right\}\stackrel{1:\op{Mult}}{\longleftrightarrow} \left\{\substack{\hbox{torically transverse}\\ \hbox{log stable curves}}\right\}\stackrel{1:1}{\longleftrightarrow} \left\{\hbox{torically transverse curves}\right\}$$
and thus involves four steps constructing the maps in each direction. The main tool is a toric degeneration.

\subsection{Toric degenerations compatible with tropical curves}
\label{section-toric-degen-from-trop}
Let $X_\Sigma$ be a smooth toric surface given by a fan $\Sigma$ in $M_\RR$. 
This is the surface that we want to count rational curves in.
Let $\Delta\in T_\Sigma$ be a given degree, $s:=|\Delta|-1$ and $P_1,...,P_s\in M_\QQ=M\otimes_\ZZ\QQ$ points in general position.
By Prop.~\ref{shapeM0k}, the set $\shM_{0,s}(\Sigma,\Delta)$ is finite and consists of simple marked parametrized tropical curves $h_i:(\Gamma_i,x^i_1,...,x^i_s)\ra M_\RR$.
We are looking for a polyhedral decomposition of $M_\RR$ with the following properties
\begin{enumerate}
\item the tropical curves are contained in the 1-skeleton of $\P$, i.e. $$h_i(\Gamma_i)\subset\bigcup_{\tau\in\P,\dim\tau=1}\tau,$$
\item $P_1,...,P_s$ are vertices of $\P$,
\item the vertices in $\P$ have rational coordinates and the facets in $\P$ have rational slope,
\item each cell in $\P$ has at least one vertex,
\item for each $\tau\in\P$ we have $\lim_{r\ra 0}r\tau$ is a cone in $\Sigma$.
\end{enumerate}
This can be obtained as follows. Let $\P_i$ be the polyhedral decomposition of $M_\RR$ induced by $h_i(\Gamma_i)$. Consider their intersection
$$\P=\P_1\cap...\cap\P_s=\{\tau_1\cap...\cap\tau_s|\tau_i\in\P_i\}\setminus\{\emptyset\}.$$
It satisfies 1 and 3 but not necessarily 2,4, or 5. However if we further intersect with several translates of the subdivision $\Sigma$ moving the origin of $\Sigma$ to each of the $P_i$ we can make sure is also satisfies 2,4,5.
It might be unnecessary to add translates of $\Sigma$, e.g. in the example of section \ref{section-sweep-pencil} for which picture (l) of Fig.~\ref{fig:pencil2} shows the union of rational curves through $P_1,...,P_5$ we find properties 1-5 satisfied directly.
There situations however where it becomes necessary to add translates of $\Sigma$, e.g. when $N^{0,\trop}_{\Sigma,\Delta}=1$ than 2 is not satisfied. This happens for instance when $X_\Sigma=\PP^2$ and when $\Delta$ is the sum of the primitive generators of the rays. Also one should note that a tropical curve might just be a straight line, e.g. the tropical version of the rational curve $\PP^1\times\{0\}$ in $X_\Sigma=\PP^1\times\PP^1$.

We replace $M$ by $\frac1aM$ where $a$ is the common denominator of the coordinates of the vertices of $\P$. This doesn't change $N^{0,\trop}_{\Sigma,\Delta}$ and turns $\P$ into an integral subdivision.
Note that $(M_\RR,\P)$ is a fan picture (dual intersection complex) for a log Calabi-Yau space in the sense of Def.~\ref{def-log-CY-space}. One obtains a degenerating family $f:X\ra\AA^1$ as follows.
Let $\Sigma_\P$ be the \emph{fan over $\P$}, i.e.
$$\Sigma_\P=\{\overline{\op{Cone}(\sigma)}\mid\sigma\in\P\}\cup \{\overline{\op{Cone}(\sigma)}\cap (M_\RR\times\{0\})\mid\sigma\in\P\}$$
where 
$$\op{Cone}(\sigma)=\{(rm,r)\mid m\in\sigma,r\in\RR_{\ge 0}\}\subset M_\RR\oplus\RR$$ and $\overline{\op{Cone}(\sigma)}$ is its closure.
We have $X$ is the toric variety associated to $\Sigma_\P$, i.e. $X=X_{\Sigma_\P}$ and the map $X\ra\AA^1$ is given by the map of fans induced by the projection $M_\RR\oplus\RR\ra\RR$.
By property 5 of $\P$, we have that $\Sigma_\P$ has $\Sigma$ as the subfan living in $M_\RR\times\{0\}$. 
This means that the general fibre of $f$ is $X_\Sigma$.
Furthermore, $\P$ is the intersection of $\Sigma_\P$ with $M_\RR\times\{1\}$, i.e. geometrically $f$ is a toric degeneration of $X_\Sigma$ and $\P$ indeed gives the fan picture for the central fibre. See Fig.~\ref{fig:fan-of-deg-del-Pezzo} for an example.
\begin{figure}[t]
\sidecaption[t]
\includegraphics[width=0.63\textwidth]{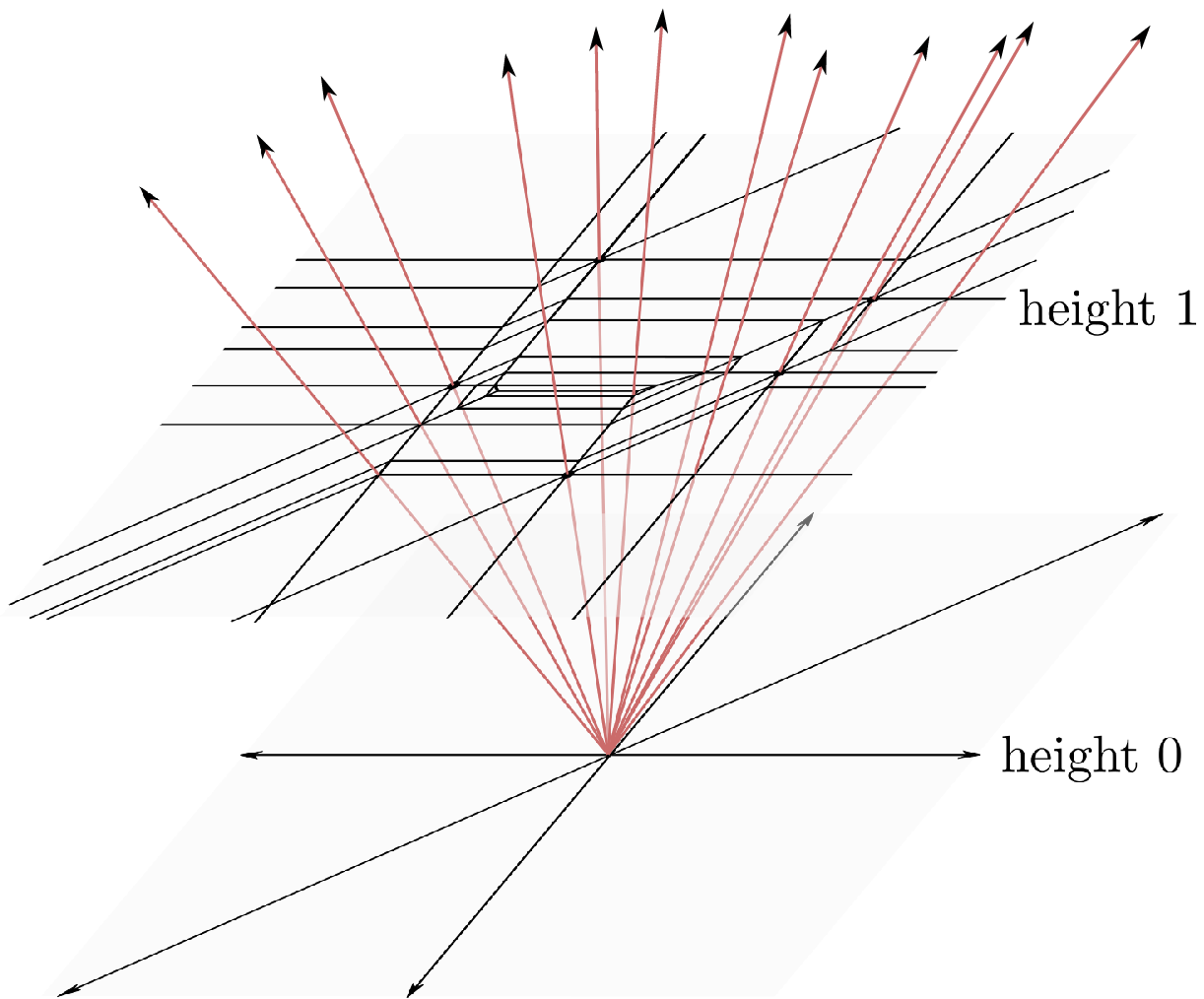}
\caption{The fan of the toric degeneration of a degree $6$ del Pezzo given by the polyhedral decomposition via the union of the rational tropical curves in picture (l) of Fig.~\ref{fig:pencil2}.}
\label{fig:fan-of-deg-del-Pezzo}
\end{figure}

\subsection{The different counts to be matched}
\label{section-match-counts}
Let $L_i$ be the rank one sublattice of $M\oplus\ZZ$ generated by $(P_i,1)$ and let $\GG(L_i)\subset\GG(M\oplus\ZZ)$ denote the corresponding one-dimensional subtorus of the open dense torus acting on $X$.
Choose general points $Q_1,...,Q_s\in \GG(M\oplus \ZZ)$ and consider $\overline{\GG(L_i).Q}$, closure of the $\GG(L_i)$-orbit of of $Q_i$ in $X$. The composition $\overline{\GG(L_i).Q}\subset X\stackrel{f}{\longrightarrow}\AA^1$ is an isomorphism, so each $\overline{\GG(L_i).Q}$ gives a section $\sigma_i:\AA^1\ra X$ of $f$.
\begin{center}
\centerline{
\xymatrix@C=30pt
{
X \ar^f[r]& \AA^1 
\ar^{\quad\sigma_1}@/^/[l] 
\ar^{\scriptsize\qquad \ddots}@{}/^0.8pc/
\ar^{\qquad\sigma_s}@(d,dr)[l] 
}}
\end{center}
Set $X_0=f^{-1}(0)$ and more generally $X_t=f^{-1}{t}$ for $t\in\AA^1$. We are next going to match the sets
\begin{enumerate} 
\item Marked parametrized rational tropical curves $(h,\Gamma,x_1,...,x_s)$ of degree $\Delta$ through $P_1,...,P_s$, i.e. the set $\shM_{0,s}(\Sigma,\Delta)$.
\item Torically transverse log stable genus zero curves $$g:C^\dagger\ra X_0^\dagger$$ going through $\sigma_1(0)$,...,$\sigma_s(0)$.
\item Torically transverse stable genus zero curves in $X_t$ going through $\sigma_1(t)$,...,$\sigma_s(t)$ for a general $t$.
\end{enumerate}
By what we said before, for any $t\neq 0$, $X_t\cong X_\Sigma$ and $\sigma_1(t)$,...,$\sigma_s(t)$ lie in general position for $t$ sufficiently general, so the count in 3. is independent of the choice of $t\neq 0$ by usual Gromov-Witten theory.
Let $K$ be the algebraic closure of $\CC((t))$, so we have inclusions 
$$\CC[t]\subset \CC((t)) \subset K$$
that gives the generic point $\eta:\Spec K\ra \AA^1$ of the base of $f$ and we may consider the fibre of $f$ over it which is
$$X_\eta = X\times_{\AA^1}\Spec K$$
and because the family $X$ is trivial outside of the central fibre, we have
$X_\eta=X_\Sigma\times_{\Spec\CC}\Spec K$
which is just the toric variety for the fan $\Sigma$ over the base field $K$. Furthermore the restriction of $\sigma$ to the point $\eta$, i.e. the composition
$$\Spec K\stackrel{\eta}{\lra} \AA^1\stackrel{\sigma_i}{\lra} X$$
gives a point $\sigma_i(\eta)\in X_\eta$.
We are going to replace the count in 3. by the following count at the generic fibre of $f$.
\begin{enumerate} \setcounter{enumi}{3}
\item Torically transverse stable genus zero curves in $X_\eta$ going through $\sigma_1(\eta)$, ..., $\sigma_s(\eta)$.
\end{enumerate}
The count in 4. coincides with that in 3. because Gromov-Witten invariants don't depend on the algebraically closed base field of characteristic zero that we define $X_\Sigma$ over.

\subsection{Turning log curves into tropical curves}
\label{section-log2trop}
Let us start with a log stable curve $g:C^\dagger\ra X_0^\dagger$ going through $\sigma_1(0)$,...,$\sigma_s(0)$. The central fibre $X_0$ is a union of closed toric strata $D_\tau$ for $\tau\in\P$ ($D_\tau$ is the closure of the torus orbit given by $\op{Cone}(\tau)\in \Sigma$). The components are actually $D_v$ for $v\in\P$ a vertex.
One checks that 
$$\sigma_i(0)\in D_{P_i},$$
in fact it lies in the dense torus of $D_{P_i}$. 
Here it becomes handy that $P_i$ are vertices of $\P$ which we ensured in section~\ref{section-toric-degen-from-trop}.
A component of $C_j$ of $C$ maps under $g$ into some toric surface $D_{v_j}$ for $v_j$ a vertex in $\P$. It doesn't map into the boundary divisor of $D_v$ by the toric transverseness assumption on $g$.

We build the tropical curve $h:(\Gamma,x_1,...,x_s)\ra  M_\RR$ corresponding to the log curve $g$ by first constructing its image $h(\Gamma)$.
The vertices of $h(\Gamma)$ will be 
$$\{v_j\mid C_j\subset C\hbox{ is a component}\}$$ and we connect two vertices by a straight line whenever the corresponding components of $C$ map to different components of $X_0$. 
It can happen that different $C_j$ map to the same $v_j$. This won't bother us.
We yet lack the rays shooting off to infinity for $h(\Gamma)$.
We add a ray $\rho\in\Sigma$ at the vertex $v_j$ for every point of intersection of $C_j$ with a divisor $D_\omega\subset D_{v_j}$ for $\omega\in\P$ a ray that is a translate of $\rho$. 
We have now built the image $h(\Gamma)$ of a tropical curve containing $P_1,...,P_s$. 
Fig.~\ref{fig-build-trop-curve} illustrates this process.
\begin{figure}[t]
\sidecaption[t]
\includegraphics[width=0.63\textwidth]{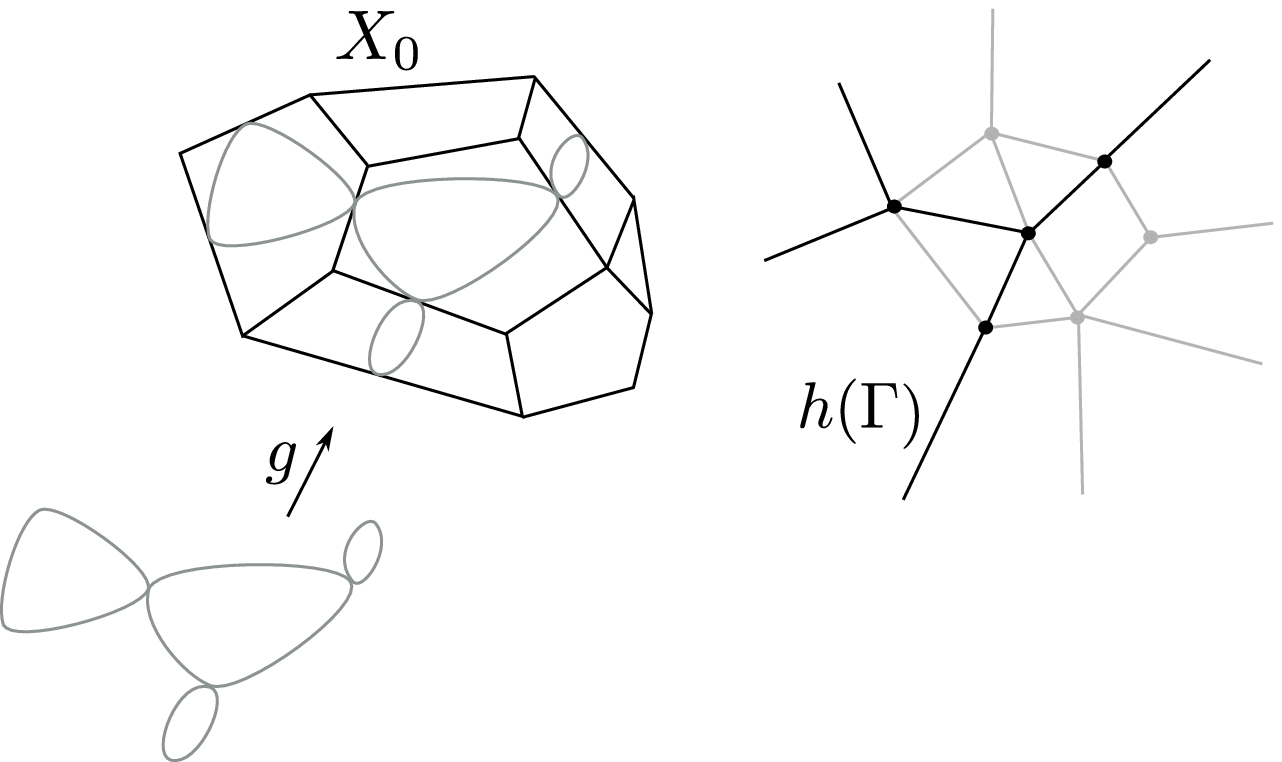}
\caption{Starting from a log curve that maps to $X_0$, we construct the associated tropical curve as part of the one-skeleton of $\P$.}
\label{fig-build-trop-curve}
\end{figure}
It remains to attach weights to edges and rays and to check that the balancing condition holds. 
To then obtain $\Gamma$ is straightforward as it is determined by $h(\Gamma)$ plus weights and the $P_i$.
Indeed, the images of edges of $\Gamma$ under $h$ meet transversely by the assumption of the $P_i$ to be in general position. As $\Gamma$ is trivalent, a higher valency than three of a vertex in $h(\Gamma)$ means a crossing of two edges of $\Gamma$.
Even beyond this, one should note that the set of vertices $v_j$ just given may be larger than the actual set of tropical curve vertices, for instance when a couple of intervals connect to form a longer interval, the midpoints get ignored in the definition of 
$(h,\Gamma,x_i)$ unless they are marked points. For the reverse construction later on, one simply retrieves the midpoints from the knowledge of $\P$.

\subsubsection{The weights}
Let us pick an edge $E$ of $h(\Gamma)$ that we want to associate a weight to. 
If $E$ is a ray with vertex $v$ then we take for its weight the sum of the intersection multiplicities with $D_E$ of the components $C_j$ of $C$ that map to $D_v$. 
A posteriori we will know that there is only one such component meeting $D_E$ and it has intersection multiplicity one with $D_E$ because the tropical curve we produce is going to be simple by Prop.~\ref{shapeM0k} and unbounded edges of simple curves have weight one.

Let now $E$ be a bounded edge, so $D_E$ is the intersection of two components $D_{v_1}, D_{v_2}$ of $X_0$. 
We define the weight of $E$ to be the sum of the intersection multiplicities with $D_E$ of all components of $C$ that map to $D_{v_1}$ and we need that this number coincides with the one where we replace $D_{v_1}$ by $D_{v_2}$. This is guaranteed by the log geometry:

\begin{lemma} 
Let $p$ be an intersection point of two components $C_1,C_2$ of $C$ that map to $D_{v_1}, D_{v_2}$ where $v_1$ and $v_2$ are connected by an edge $E$ and $g(p)\in D_E$. 
The intersection multiplicity of $g(C_1)$ with $D_E$ coincides with the intersection multiplicity of $g(C_2)$ with $D_E$.
\end{lemma}

\begin{proof}
Recall that $S_e$ is the monoid that is given multiplicatively by
$$S_e=\langle x,y,z\mid xy=z^e\rangle.$$
Let $l$ be the integral length of $E$. The log structure of $X_0$ at $g(p)$ is given by the local structure near the origin in the log chart
$$
\begin{array}{rcl}
S_l&\ra& \CC[x,y,u]/(xy)\\
x&\mapsto& x\\
y&\mapsto& y\\
z&\mapsto& 0.
\end{array}
$$
In other words, while the underlying space $X_0$ is ignorant of the length of $E$, its log structure still remembers it.
The local structure of the log map $g:C^\dagger\ra X_0^\dagger$ takes the shape in terms of local charts at $p$ and $g(p)$ given in the following commutative diagram of monoids.
\begin{equation} \label{log-structure-at-node}
\begin{gathered}
\xymatrix@C=30pt
{
\CC[x,y]/(xy)  &  \CC[x,y,u]/(xy) \ar_g^{\substack{x^w\mapsfrom x\\ y^w\mapsto y\\ 0\mapsfrom u}}[l] \\
S_e\ar[u] & S_l\ar[u]  \ar^{\substack{x^w\mapsfrom x\\ y^w\mapsfrom y\\ z\mapsfrom z}}[l]  
}
\end{gathered}
\end{equation}
so there is another integer $e$ that is encoded in the log structure of $C$ (similarly as $w$ is encoded in the log structure of $X_0$) and there is an integer $w$ that comes from the log-structure part of the map $g$.
The well-definedness of this part implies 
$$we=l$$
and $w$ is the intersection multiplicity with $D_E=\Spec\CC[u]$ of either component of $C$.
\end{proof}

\subsubsection{The balancing condition}
Let us now pick a vertex $v\in h(\Gamma)$ that corresponds to a component $C_v$ of $C$ that maps non-constantly into $D_v$ under $g$.
Let $D_{E_1},...,D_{E_r}$ be the toric divisors in $D_v$ that are met by $g(C_v)$ with intersection multiplicities $w_1,...,w_r$ respectively.
Let $\Sigma_v$ denote the fan of $D_v$ with the rays corresponding to $E_1,...,E_r$ generated by the primitive vectors $m_1,...,m_r$.
We want to show that $$\sum_i w_i m_i=0$$
for which it suffices to show that $\sum_i w_i \langle m_i, n\rangle=0$ holds for all $n$ in the dual space. Such an $n$ defines a rational function $z^n$ and $\langle m_i, n\rangle$ is its order of vanishing along $D_{E_i}$, so
$\sum_i w_i \langle m_i, n\rangle$ is the divisor of zeros and poles of the restriction of $z^n$ to $g(C_v)$ which is therefore zero.

\subsection{Turning tropical curves into log curves}
\label{section-trop2log}
The knowledge about Prop.~\ref{shapeM0k} becomes handy for this step. 
It tells us that there are only finitely many tropical curves (that we have already built into the construction of $\Sigma$) and moreover these are all simple.
Let now $(h,\Gamma,x_1,...,x_s)$ be one of them. We want to construct a torically transverse log curve $g:C^\dagger\ra X_0^\dagger$ whose tropical curve under the association in the previous section~\ref{section-log2trop} brings us back to $h$. 
We will need that for an edge $\omega$ of $\P$ contained in $h(E)$ for an edge $E$ of $\Gamma$ the weight $w(E)$ divides the length of $\omega$ because this always holds for the resulting tropical curve obtained from a log curve by the previous section. We can achieve this by replacing $M$ by $\frac1b M$ for a suitable $b$ if necessary.

Let $\widehat\Gamma$ be the graph that results from first removing all marked edges from $\Gamma$ and then removing each resulting bivalent vertex by identifying its adjacent edges respectively. 
We denote by $\widehat\Gamma^{[0]}$ the vertices of $\widehat\Gamma$ (these coincide with those vertices of $\Gamma$ that are not adjacent to a marked edge). 
By $\widehat\Gamma^{[1]}$ we denote the set of edges of $\widehat\Gamma$ and $E_j$ ($1\le j\le s$) refers to the edge of $\widehat\Gamma$ that arises from identifying the edges of $\Gamma$ adjacent to $E_{x_j}$. Note that a priori it could happen that $E_j=E_k$ for $j\neq k$. For $E\in \widehat\Gamma^{[1]}$ we define its weight $w(E)$ as the weight of an edge of $\Gamma$ that is one of its constituents (or coincides with it) which is well-defined by the balancing condition and since $w(E_{x_i})=0$.

For each bounded edge $E$ in $\widehat\Gamma$ let $v^+_E,v^-_E$ be an enumeration of its vertices and for a ray $E$ let $v_E^-$ be its vertex. 
Let $u_E\in M$ be the primitive vector pointing from $h(v_E^-)$ into $h(E)$. 
(In case $E=E_i$, let $u_E$ be the primitive vector pointing from $h(v_E^-)$ into $h(E')$ where $E'$ is the edge of $\Gamma$ adjacent to $v_E^-$ and that got concatenated with other edges of $\Gamma$ to become $E$.)
We set $u_i=u_{E_i}$ and $v_-^i=v_-^{E_i}$.
The crucial gadget in this section is the map of lattices

\begin{eqnarray}
\Phi: \Map(\widehat\Gamma^{[0]},M)&\ra&  \left(\prod_{\substack{E\in \widehat\Gamma^{[1]}\\ E \hbox{ \scriptsize bounded}}} M/\ZZ u_E\right) \oplus\left( \prod_{i=1}^s M/\ZZ u_i\right) \nonumber\\
H&\mapsto& \left(\left(H(v_+^E)-H(v_-^E)\right)_E, H(v_-^{1}),...,H(v_-^{s})\right) \nonumber
\end{eqnarray}
An element $H\in\Map(\widehat\Gamma^{[0]},M_\RR)$ gives a piecewise affine deformation $h_H$ of $h$ (with fewer vertices however) by moving the vertices that are in $\widehat\Gamma^{[0]}$ as prescribed by $H$, i.e.
$$\widehat\Gamma^{[0]} \ni v\mapsto v+H(v)=:h_H(v).$$
One extends this to a map $h_H:\widehat\Gamma\ra M_\RR$ by sending a bounded edge affine linearly the the interval between the images of its vertices and an unbounded unmarked edge $E$ gets mapped to the parallel translate of $h(E)$ so that its vertex is $h_H(v^E_-)$ (If $E$ is an edge concatenated from various edges of $\Gamma$, then we mean by $h(E)$ the union of the images of the individual edges under $h$.)
Let $\Phi_\RR$ be the result of tensoring $\Phi$ by $\RR$. 
The main point is that $h_H:\widehat\Gamma\ra M_\RR$ is a parametrized tropical curve containing the $P_i$ if and only if $H\in\ker\Phi_\RR$.
Since $h$ is rigid, $\ker\Phi_\RR=0$ and thus $\Phi$ is injective. By a rank count one concludes
\begin{lemma} $\Phi$ is an embedding of lattices with finite index.
\end{lemma}
Let $\frak{d}=|\coker\Phi|$ be this index. 

\begin{theorem} 
\begin{enumerate} 
\item The number of stable maps $g:C\ra X_0$ with $\sigma_i(0)\in g(C)$ that give back $h$ under the recipe of the previous section is $\frak{d}$.
\item The number of possibilities of turning a given $g:C\ra X_0$ into a strict log map $g:C^\dagger\ra X_0^\dagger$ is
$$\frak{w}=\left(\prod_{\substack{E\in \Gamma^{[1]}\\ E \hbox{ \scriptsize bounded}}} w(E) \right)\cdot\left( \prod_{i=1}^s w(E_i) \right).$$
\item We have
$$\frak{d}\cdot\frak{w} =\op{Mult}(h).$$
\end{enumerate}
\end{theorem}

\begin{proof} 
We give only the main ideas since details can be found in \cite{kan}.
For a lattice $L$, we denote $\GG(L)=L\otimes_\ZZ\GG_m$ the corresponding group scheme for $\GG_m$ the multiplicative group of $\CC$.
The result of applying $\GG$ to $\Phi$,
\begin{eqnarray}
\GG(\Phi): \GG(\Map(\Gamma^{[0]},M))&\ra&  \left(\prod_{\substack{E\in \Gamma^{[1]}\\ E \hbox{ \scriptsize bounded}}} \GG(M/\ZZ u_E)\right) \times\left( \prod_{i=1}^s \GG(M/\ZZ u_i)\right) \nonumber
\end{eqnarray}
is a surjection with kernel $\Tor_1^\ZZ(\coker(\Phi),\GG_m)\cong \coker(\Phi)$ of size $\frak{d}$ because $\GG_m$ is divisible.
We are going to match choices of stable maps $g$ corresponding to $h$ with elements in the source of $\GG(\Phi)$ that map to the trivial element in the target.
Since $|\ker(\GG(\Phi))|=\frak{d}$, we will conclude item 1 of the Theorem from this.
This matching can be seen through the following steps.
\begin{enumerate}
\item 
\begin{minipage}[t]{0.65\textwidth} 
Given $v\in \widehat\Gamma^{[0]}$, let 
$E_1,...,E_3$ be the adjacent edges of $\widehat\Gamma^{[0]}$ and
$\omega_1,\omega_2,\omega_3\subset h(\Gamma)\cap \P$ be the corresponding three edges in $\P$ containing $h(v)$.
One checks that the standard action of $\GG(\Hom(\{v\},M))\cong \GG_m^2$ on $D_{h(v)}$ induces a transitive and free action on the set of 
maps $g_v:\PP^1\ra D_{h(v)}$ up to automorphism of the domain such that $g_v(\PP^1)$ meets the three divisors $D_{\omega_1}, D_{\omega_2}, D_{\omega_3}$ at order $w(E_1), w(E_2), w(E_3)$.
\end{minipage}\hfill
\begin{minipage}[t][2.5cm][b]{0.26\textwidth}
\includegraphics[width=0.9\textwidth]{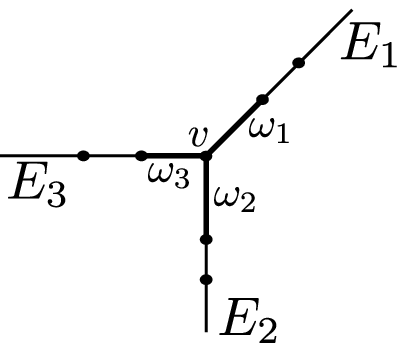}
\end{minipage}\\
\item 
Let $h_0:\widehat\Gamma\ra M_\RR$ refer to the map $h_H$ with $H=0$, i.e. $h_0$ is the adaption of $h$ from $\Gamma$ to $\widehat\Gamma$.
For $E\in \widehat\Gamma^{[1]}$ an edge connecting $v_1$ to $v_2$, we need to connect $g_{v_1}(\PP^1)$ to $g_{v_2}(\PP^1)$ by a chain of $\PP^1$s, one for each $v\in\P$ contained in the relative interior of $h_0(E)$.
The $\PP^1$ corresponding to such a $v\in h_0(E)$ maps into $D_v$. Each such $D_v$ has a natural $\PP^1$-fibration via the map $M\ra M/\ZZ u_E$ and we are looking for a chain of fibres of these fibrations. 
The fibres are parametrized by $\GG(M/\ZZ u_E)$ and the condition that $g_{v_1}(\PP^1)$ connects to $g_{v_2}(\PP^1)$ can be phrased by saying that 
$\GG(\Hom(\{v_1\},M))$ and $\GG(\Hom(\{v_2\},M))$ project to the same element of $\GG(M/\ZZ u_E)$.
\item Eventually $P_i$ lies in the relative interior of $h_0(E_i)$ and $\sigma_i(0)$ lies in the fibration fibre given by some element of $\GG(M/\ZZ u_i)$. That $g(C)$ contains this fibre is encoded in the second factor in the target of $\GG(\Phi)$.
\end{enumerate}

To prove 2. note that by strictness there is only a choice for the log structure at the special points of $C$. 
These are points $p\in C$ such that the log structure of $X_0$ at $g(p)$ is not just the pullback from the base $\Spec\CC^\dagger$. There are three kinds of such points: points of $X_0$ where two components meet, marked points $\sigma_i(0)$ and points in the toric boundary $\partial X_0=\bigcup_{D\subset X_\Sigma\hbox{ \scriptsize a prime divisor not in }X_0} D\cap X_0$. 
One checks that there is only a choice to be made at points where two components of $X_0$ meet. 
The structure there is given by \eqref{log-structure-at-node}. All maps in this diagram are fixed except for the left vertical one that we may twist. 
There is not much of a choice for twisting either in order to keep commutativity. What works for this map is this
$$x\mapsto \zeta x,\quad y\mapsto y$$
for a $w$th$=\frac{l}e$th root of unity $\zeta$.
This gives the same log structure on $C$ abstractly but not the same as a $\Spec\CC^\dagger$-scheme because the product $xy$ changes by $\zeta$ and this is a section coming from of the log structure in the base $\Spec\CC^\dagger$.
On the other hand twisting both $x$ and $y$ yields something that can be shown to be isomorphic to one of the given twists.

Finally, some of these choices are isomorphic by an isomorphism of the underlying scheme $C$. Indeed, we can apply a deck transformation to the source whenever a $\PP^1$ maps to $D_v$ for $v$ in the interior of $h_0(E)$ with $E\in\widehat\Gamma^{[1]}$. There are $w(E)$ sheets that we can permute cyclically. 
This has the effect of that for a given edge $E$ of $\widehat\Gamma$ with vertices $v_1,v_2$ the number of choices for the log structure at the nodes of the chain of $\PP^1$s connecting $g_{v_1}(\PP^1)$ and $g_{v_2}(\PP^1)$ is
$$w(E)\cdot|\{\omega\in\P \hbox{ an edge with } \omega \subset h_0(E)\}|$$
whereas the total of possible deck transformations is
$$w(E)\cdot |\{v\in\P \hbox{ a vertex in the interior of }h_0(E)\}|.$$
The net choice is thus $w(E)$ for each bounded edge $E$ and there is only one choice for unbounded edges. 
If the edge is marked however, there is an additional choice of where to place the marking in the cover, so a marked edge $E$ contributes an additional factor of $w(E)$. This gives item 2 in the assertion.

The proof of item 3 starts with a local argument noting the multiplicity at a (trivalent) vertex $v$ of $\widehat\Gamma$ is defined by
$$w(E_1)w(E_2) |u_{E_1}\wedge u_{E_2}|$$
where $E_1,E_2$ are two of the three outgoing edges at $v$. 
One finds that $|u_{E_1}\wedge u_{E_2}|$ coincides with the rank of the cokernel of
$$\Hom(\{v\},M)\ra M/\ZZ u_{E_1}\oplus M/\ZZ u_{E_2}$$
which is the map given by projection on each component and it is a constituent of the map $\Phi$. 
One can prove item 3 by induction where one removes an unbounded ray with its vertex in each step. One finally uses item 1 and 2, for more details see \cite{kan}.
\end{proof}

In short, we have seen in this section that for a tropical curve $h:(\Gamma,x_1,...,x_2)\ra M_\RR$ there are $\op{Mult}(h)$ many different torically transverse log curves up to isomorphism that match the combinatorics of $h$.

\subsection{From ordinary stable curves to log stable curves}
Assume now we are given a torically transverse stable curve $g_\eta :C_\eta \ra X_\eta$, i.e. mapping in the generic fibre of the degeneration $f:X_\Sigma\ra\AA^1$. 
We require it to contain $\sigma_1(\eta),...,\sigma_s(\eta)$.
A priori, we don't know $\Sigma$ as it was constructed from the tropical curves after choosing $P_i$ and we do neither know the tropical curves nor the $P_i$ yet. 
Instead we start with any $\Sigma$ satisfying properties 3.-5. of section~\ref{section-toric-degen-from-trop}. E.g. the constant family with fibre the toric surface $X$ we started with will do.
A fan satisfying also properties 1. and 2. will come out of the construction in this section.
Starting with the diagram
\begin{equation} \nonumber
\begin{gathered}
\xymatrix@C=30pt
{
C_\eta\ar^{g_\eta}[r]\ar[d] &  X_\Sigma\setminus X_0\ar^f[d]\\
\Spec K\ar[r] & \AA^1\setminus\{0\}
}
\end{gathered}
\end{equation}
we want to fill in the fibres over $\{0\}$. 
This can be done by stable reduction after possibly doing a base change $\AA^1\ra\AA^1, z\mapsto z^k$ which is no problem for us as it just means a rescaling of $\P$.
The resulting stable curve that then maps into $X_\Sigma$ will typically not be torically transverse. 
By a suitable toric blow-up supported on the central fibre $X_0$ and given by a subdivision of $\P$, the map can be made torically transverse, see 
\cite[Thm. 4.24]{kan}. Doing this blow up for each curve in $X_\eta$ will implicitly ensure that $\P$ contains all tropical curves as required in property 1 and 2.
Once one has filled the central fibre by blow-up and semi-stable reduction we obtain a diagram
\begin{equation} \nonumber
\begin{gathered}
\xymatrix@C=30pt
{
C\ar^{g}[r]\ar[d] &  X_\Sigma\ar^f[d]\\
\Spec R\ar[r] & \AA^1
}
\end{gathered}
\end{equation}
with $R$ a discrete valuation ring and the base horizontal map dominant and $C$ a torically transverse stable map. 
We obtain the log curve by restriction of $g$ to $C_0$, the fibre over $\{0\}$, plus pulling back the divisorial log structure $\shM_{(X_\Sigma,X_0)}$ to $C_0$, see Ex.~\ref{ex-div-log-str}.

\subsection{From log curves to ordinary curves}   \label{log-curves-to-curves}
Starting with a torically transverse log stable curve $g_0:C_0^\dagger\ra X_0^\dagger$, we want to deform it to a stable map $g:C\ra X_\Sigma$ so that we can then restrict it to the generic point $\eta$ to obtain an ordinary torically transverse stable curve 
\begin{equation}
\nonumber
g_\eta:C_\eta\ra (X_\Sigma)_\eta=X\times_{\Spec\CC}\Spec K.
\end{equation}
This works by log deformation theory.
The goal is to lift $g_0:C_0^\dagger\ra X_0^\dagger$ order by order to $C_1,C_2,...$ where $g_i:C_i^\dagger\ra X_\Sigma^\dagger$ is defined over $\Spec \CC[t]/(t^{i+1})$. We can then take the projective limit to obtain a curve 
$g_\infty:C_\infty\ra X_\Sigma$ defined over $\CC\llbracket t\rrbracket$ which we then restrict to $\eta$ to get the ordinary curve.
There are four steps
\begin{enumerate}
\item thicken $C_0^\dagger$ to higher orders $C_i^\dagger$,
\item (step 1 plus) extend the map $C_i^\dagger\ra X_\Sigma$,
\item (step 1,2 plus) extend the marked points $x_j\ra C_i$ as sections over $\Spec \CC[t]/(t^{i+1})$,
\item (step 1,2,3 plus) make sure that the sections $x_j$ map under $g$ to the sections $\sigma_j$.
\end{enumerate}
The first item is governed by log smooth deformation theory. The obstruction group is $H^2(C_0,\Theta_{C_0^\dagger/\CC^\dagger})$ where $\Theta_{C_0^\dagger/\CC^\dagger}$ is the relative log tangent sheaf and this cohomology group vanishes because $C_0$ is a curve so any $H^2$ of a coherent sheaf is zero. The lifts from $C^\dagger_{i}$ to $C^\dagger_{i+1}$ form a torsor over 
$$H^1(C_0,\Theta_{C_0^\dagger/\CC^\dagger}).$$
For the second step consider the exact sequence 
\begin{equation}
\label{log-normal-seq}
0 \ra \Theta_{C_0^\dagger/\CC^\dagger}\ra g^*\Theta_{X_0^\dagger/\CC^\dagger}\ra \shN_{g_0}\ra 0
\end{equation}
where $\Theta_{X_0^\dagger/\CC^\dagger}$ is the (relative) log tangent sheaf of $X_0^\dagger$ and $\shN_{g_0}$ is defined by this sequence and can be called the log normal sheaf to $g_0$. 
Obstructions to lifting the map $C_i^\dagger\ra X_\Sigma$ sit in $H^1(C_0,g_0^*\Theta_{X_0^\dagger/\CC^\dagger})$. 
This group is trivial because $\Theta_{X_0^\dagger/\CC^\dagger}$ is a trivial vector bundle by a general fact for the standard log structures on toric varieties and since $C_0$ is a rational stable curve $H^1(C_0,\shO_{C_0})=0$.
One can show that the set of lifts from $C_i^\dagger\ra X_\Sigma^\dagger$ to $C_{i+1}^\dagger\ra X_\Sigma^\dagger$ is a torsor over 
$$H^0(C_0,\shN_{g_0}),$$ 
see \cite[Thm. 3.41]{kan}. 
This connects to step one via the connecting homomorphism in cohomology
$$H^0(C_0,\shN_{g_0})\ra H^1(C_0,\Theta_{C_0^\dagger/\CC^\dagger}).$$
For step 3 consider the embedding $\Theta_{C_0^\dagger/\CC^\dagger}(-\sum_i x_i) \subset \Theta_{C_0^\dagger/\CC^\dagger}$. We can modify \eqref{log-normal-seq} to
\begin{equation}
\label{log-normal-seq2}
0 \ra \Theta_{C_0^\dagger/\CC^\dagger}\left(-\sum_i x_i\right)\ra g^*\Theta_{X_0^\dagger/\CC^\dagger}\ra \shN_{g_0,{\bf x}}\ra 0
\end{equation}
where again $\shN_{g_0,{\bf x}}$ is defined via this sequence. There will then be a surjection $$\shN_{g_0,{\bf x}}\ra \shN_{g_0}$$ whose kernel can be identified with $\bigoplus_{i=1}^s\Theta_{C_0^\dagger/\CC^\dagger}|_{x_i}$ and thus
\begin{equation}
\label{split-Ng0}
\shN_{g_0,{\bf x}}\cong \shN_{g_0}\oplus \bigoplus_{i=1}^s\left.\Theta_{C_0^\dagger/\CC^\dagger}\right|_{x_i}.
\end{equation}
Given $g_i:C_i^\dagger \ra X_\Sigma^\dagger$ with sections $x_i:\Spec\CC[t]/(t^{i+1})\ra C_i$, the set of lifts of this data to order $i+1$ is a torsor over 
$$H^0(C_0,\shN_{g_0,{\bf x}}),$$
 see \cite[Thm. 3.42]{kan}.

Finally for step 4, one considers the map
$$\Xi:H^0(C_0,\shN_{g_0,{\bf x}})\ra \bigoplus_{i=1}^s\left. g^*\Theta_{X_0^\dagger/\CC^\dagger}\right|_{x_i}$$
given by choosing local lifts near the $x_i$ from $\shN_{g_0,{\bf x}}$ to $g^*\Theta_{X_0^\dagger/\CC^\dagger}$ and then restricting these to the $x_i$.
The right hand side records the deformation of the sections $\sigma_i(0)$ and if we want to follow any such deformation with the images $g(x_i)$, the map $\Xi$ needs to be surjective and the set of lifts satisfying item $4$ is then a torsor under $\ker\Xi$, see \cite[Thm. 3.43]{kan}.
It turns out the $\Xi$ is an isomorphism, so there is actually a unique lift for item 4. 
The proof is going to features the map $\Phi$ once more!
Using the splitting \eqref{split-Ng0} we find $\Xi$ is an isomorphism if and only if
$$\Xi':H^0(C_0,\shN_{g_0})\ra \bigoplus_{i=1}^s \frac{\left.g^*\Theta_{X_0^\dagger/\CC^\dagger}\right|_{x_i}}{\left.\Theta_{C_0^\dagger/\CC^\dagger}\right|_{x_i}} $$
is one. The range of $\Xi'$ can be identified with 
$$\prod^s_{i=1} (M/\ZZ u_i)\otimes_\ZZ\CC$$
Via a components-wise calculation and gluing condition, one finds that $H^0(C_0,\shN_{g_0})$ is identified with the kernel of the surjection
$$
\Map(\widehat\Gamma^{[0]},M)\otimes\CC \ra  \left(\prod_{\substack{E\in \widehat\Gamma^{[1]}\\ E \hbox{ \scriptsize bounded}}} (M/\ZZ u_E)\otimes\CC\right)
$$
Hence, $\Xi'$ is an isomorphism if and only if $\Phi\otimes\CC$ is one and we have seen this earlier. For details, consult \cite[\S4.5]{kan}.

\section{Mirror Symmetry for $\mathbb{P}^2$}
\label{sec:p2ms}
We give a sketch of Gross's construction of mirror symmetry for $\mathbb{P}^2$, which can be seen as  a tropical reformulation and expansion of Barannikov's construction \cite{bar}.  We begin with an outline of the relevant details of Barannikov's construction,  touch on the major concepts and tools of Gross's construction, and end with a statement of the theorem.  This exposition should be viewed as an attempt to give an abridged summary of \cite{kan} \cite{MSP2} with a few explanatory notes.
\subsection{Introduction}
\label{subsec:msintro}

In the case of Calabi-Yau threefolds, mirror symmetry relates the moduli space of K\"ahler structures on one manifold $X$ (the so-called A-model) with the moduli space of complex structures on another manifold, $\check{X}$ (the B-model).  Our picture is a bit different, as we'll be examining a mirror symmetry construction for $\mathbb{P}^2$, which is not Calabi-Yau.  

The A-model structure we'll be discussing on $X:=\mathbb{P}^2$ is relatively straightforward to describe; it concerns (roughly) counts of rational curves on $X$ satisfying certain intersection and genus requirements. These  ``counts," called Gromov-Witten invariants, can be used to perturb the usual cup product on the cohomology of $X$ into something called quantum cohomology, a construction whose operations can then be compiled into a particularly nice object called a Frobenius manifold.

When $X$ is Fano, as it is in our case, the mirror object is not a manifold but rather a \emph{Landau-Ginzburg model}. In the context of our discussion, this consists of a pair $(\hat{X}, W)$, where $\hat{X}$ is a variety and $W : \hat{X}\rightarrow \mathbb{C}$ a regular function called a \textit{Landau-Ginzburg potential}. Through Barannikov's technique of semi-infinite variation of Hodge structures \cite{bar}, one can again recover a Frobenius manifold. Mirror symmetry dictates that the Frobenius manifolds arising in the A- and B-model constructions should be the same.
 
In the case of $X = \mathbb{P}^2$, Gross has shown that both sides of the mirror are intrinsically susceptible to analysis by tropical geometry \cite{kan} \cite{MSP2}.  In his pioneering work, Mikhalkin demonstrated its descriptive power for the A-model by showing it possible to compute certain Gromov-Witten invariants for toric surfaces (including, of course, $\mathbb{P}^2$) by counting tropical curves in $\mathbb{R}^2$ \cite{mik}. The ease with which these invariants could now be computed and the conceptual insight yielded by the tropical point of view has inspired many attempts to generalize the result. Gathmann, Markwig, Kerber, Rau and others have made significant progress in this regard, establishing not only methods for the tropical computation of certain \textit{descendant} Gromov-Witten invariants, but also an intersection theory on a relevant moduli space \cite{GKM} \cite{M+R}. 

The tropical interpretation of the Landau-Ginzburg model is more recent.  The content of Gross's version of mirror symmetry for $\mathbb{P}^2$ is a simple, tropical description of the Landau-Ginzburg potential such that the mirror relationship can be easily described in terms of combinatorial objects.  This should be seen as a proof-of-concept for the Gross-Siebert program, exhibiting mirror symmetry via by expressing either side of the picture using the same tropical data.  For discussion on the generalization of these ideas and a better sense of their context, please see \cite{kan} and especially \cite{gsyz}.  
\section{Barannikov's construction}
\subsection{A model}
We will assume basic knowledge of Gromov-Witten theory.  For more information, consult the relevant chapter in this volume.  We'll confine our discussion to the concrete example of $X:=\mathbb{P}^2$.
Define $\mathcal{M}:=\spec \mathbb{C}[[y_0,y_1,y_2]]$.
Let $T_i$ be a positive generator of $H^{2i}(\mathbb{P}^2, \mathbb{Z})$ and let
$$\gamma:=y_0T_0+y_1T_1+y_2T_2$$
With this data, we are able to define the \textit{Gromov-Witten potential of $\mathbb{P}^2$}. 
$$\Phi := \sum_{k=0}^\infty \sum_{\beta\in H_2(X,\mathbb{Z})} \frac{1}{k!} \langle \gamma^k \rangle_{0,\beta}.$$
This function encodes much of the enumerative information of $\mathbb{P}^2$.
Define a constant metric $g$ on $\mathcal{M}$ with
$$g(\partial_{y_i}, \partial_{y_j}):=\int_{\mathbb{P}^2}T_i\cup T_j$$
and the connection $\nabla$ given by the flat sections $\partial_{y_i}$.  Define a product structure on the tangent bundle of $\mathcal{M}$ given by 
$$\partial_{y_i}\ast \partial_{y_j}:=\sum_{a,l} (\partial_{y_i} \partial_{y_j} \partial_{y_a} \Phi)g^{al} \partial_{y_l}.$$ 
This data defines a \emph{Frobenius manifold}.  For much more on these objects, see \cite{Ma99}.

Identifying $T_i$ with  $\partial_{y_i}$,  one can think of $\ast$ as giving a product structure on $H^*(\mathbb{P}^2, \mathbb{C}[[y_0, y_1, y_2]])$.  This is known as the \emph{big quantum cohomology ring}. 
The A-model data encoded in this manifold can be arranged into a function that will arise naturally on the other side of the mirror.  To define this function, we'll need a slight upgrade of the Gromov-Witten invariant, known as the \emph{descendent} Gromov-Witten invariant.  
\begin{definition}[Descendent Gromov-Witten invariants]
For $\alpha_i\in H^*(X, \mathbb{C})$, define
$$\langle \psi^{j_1} \alpha_1, \ldots  \psi^{j_n} \alpha_n\rangle_{g,\beta}:=\int_{[\bar{\mathcal{M}}_{g,n}(X,\beta)]^{vir}}\psi_1^{j_1}\cup\ldots \cup \psi_n^{j_n}\cup ev^*(\alpha_1\times \cdots \times \alpha_n).$$
Here we've attached a natural line bundle $\mathcal{L}_i$ to $\bar{\mathcal{M}}_{g,n}(X,\beta)$ associated to each marked point $x_i$.  The fiber of $\mathcal{L}_i$ at a point $[(C,x_1,\ldots, x_n)]$ is the cotangent line $\mathfrak{m}_{x_i}/\mathfrak{m}_{x_i}^2$, where $\mathfrak{m}_{x_i}\subseteq \mathcal{O}_{C,x_i}$ is the maximal ideal.  Then  $\psi_i:=c_1(\mathcal{L}_i)\in H^2(\bar{\mathcal{M}}_{g,n}(X,\beta), \mathbb{Q})$.
\end{definition}

\begin{definition}[Givental's J-function for $\mathbb{P}^2$]
$J_{\mathbb{P}^2}:\mathcal{M}\times \mathbb{C}^\times \rightarrow H^*(\mathbb{P}^2, \mathbb{C})$ is defined as follows:
\begin{eqnarray*}
J_{\mathbb{P}^2}(y_0, y_1, y_2, \hbar)&:=&e^\frac{y_0T_0+y_1T_1}{\hbar}\cup\left(T_0+\sum_{i=0}^2\left(y_2\hbar^{-1}\delta_{2,i}\vphantom{\frac12}\right.\right.\\
&&\left.\left.\sum_{d\geq 1}\sum_{\nu\geq 0} \langle T_2^{3d+i-2-\nu}, \psi^\nu T_{2-i}\rangle_{0,d}\hbar^{-(\nu+2)}e^{dy_1}\frac{y_2^{3d+i-2-\nu}}{(3d+i-2-\nu)!}\right)T_i\right)
\end{eqnarray*}
\end{definition}
We can define functions $J_i:\mathcal{M}\times \mathbb{C}^\times \rightarrow H^{2i}(\mathbb{P}^2, \mathbb{C})$ by the decomposition of $J$:
$$J_{\mathbb{P}^2}=\sum_{i=0}^n J_iT_i$$

\subsection{B model}
\label{bmodel}
Here we follow the summary of Barannikov's results \cite{bar} as given in \cite{MSP2}.  The mirror of $\mathbb{P}^2$ is the Landau-Ginzburg model $(\hat{X}, W)$, where $\hat{X}:=V(x_0x_1x_2-1)\subseteq \spec\CC[x_0,x_1, x_2]$ and $W=x_0+x_1+x_2$. 

We consider the universal unfolding of $W$ parametrized by the moduli space $\Specf\mathbb{C}[[t_0, t_1, t_2]]$
$$W_{\bf t}:=\sum_{i=0}^2 W^it_i,$$
and the local system $\mathcal{R}$ on $\mathcal{M}\times \mathbb{C}^\times$ whose fiber at a point $({\bf t}, \hbar)$  is the relative homology group $H_n(\hat{X}, Re(W_{\bf t}/\hbar)\ll 0)$.  With this setup, Barannikov uses semi-infinite variation of Hodge parameters to show the following result.  See Chapter 2 of \cite{kan} for a discussion of how these structures arise in our particular example.  First, there is a unique choice of the following data:
\begin{itemize}
\item A (multi-valued) basis of sections of $\mathcal{R}$, $\Xi_0,\Xi_1,\Xi_2$, with $\Xi_i$ uniquely defined modulo $\Xi_0,\ldots, \Xi_{i-1}$.
\item A section $s$ of $\mathcal{R}^\vee \otimes_\mathbb{C} \mathcal{O}_{\mathcal{M}\times\mathbb{C}}$ defined by integration of a family of holomorphic forms on $\hat{X} \times \mathcal{M} \times\mathbb{C}^\times$ of the form $$e^{W_{\bf t}/\hbar}f {\rm dlog}x_1\wedge{\rm dlog} x_2$$
where $\hbar$ is the coordinate on $\mathbb{C}$ and $f$ is a regular function on $\hat{X}\times \mathcal{M}\times\mathbb{C}^\times$ with $f|_{\hat{X}\times \{0\} \times \mathbb{C}^\times}=1$ and which extends to a regular function on $\hat{X}\times \mathcal{M}\times(\mathbb{C}^\times\cup\{\infty\})$.
\item  The monodromy associated with $\hbar\rightarrow \hbar e^{2\pi i}$ in $\mathcal{R}$ is given, in the constructed basis, by $\exp(6\pi iN)$, where
\begin{eqnarray*}
N=\left( \begin{array}{ccc}
0&1&0\\
0&0&1\\
0&0&0\\
\end{array}\right)
\end{eqnarray*}
\item A fiber of $\mathcal{R}^\vee$ is identified with the ring $\mathbb{C}[\alpha]/(\alpha^3)$, with $\alpha^i$ dual to $\Xi_i$.
The selected section $s$ of $\mathcal{R}^\vee\otimes \mathcal{O}_{\mathcal{M}\times \mathbb{C}^\times}$ gives us an element of each fiber of $\mathcal{R}^\vee$, which we write as
$$s({\bf t}, \hbar)=\sum_{i=0}^2 \alpha^i \int_{\Xi_i}e^{W_{\bf t}/\hbar}f{\rm dlog}x_1\wedge {\rm dlog} x_2$$
We require that we can write
$$s({\bf t}, \hbar)=\hbar^{-(3\alpha)}\sum_{i=0}^2 \phi_i({\bf t}, \hbar)(\alpha\hbar)^i$$
for functions $\phi_i$ satisfying
$$\phi_i({\bf t}, \hbar)= \delta_{0,i}+\sum_{j=1}^\infty \phi_{i,j}({\rm t}) \hbar^{-j}$$
for $0\leq i \leq 2$.  
These conditions place a restriction on the function $f$.  In the above,
$$\hbar^{-3\alpha}=\sum_{i=0}^2 \frac{(3)^i}{i!}(-\log \hbar)^i\alpha^i,$$
which absorbs the multi-valuedness of the integrals.\\
\end{itemize}
As a result of these conditions, if we set $y_i({\bf t})=\phi_{i,1}({\bf t})$, the functions $y_i$ form a set of coordinates on $\mathcal{M}$, $\lim_{\hbar\rightarrow \infty} \hbar^i\phi_i(0,\hbar)=\delta_{0,i}$, and we are able to state the following:

\begin{proposition}[Mirror symmetry for $\mathbb{P}^2$]
\label{msp2}
Given the above setup, on the $\mathbb{C}$ vector space $\mathbb{C}[[y_0,y_1,y_2, \hbar^{-1}]]$,
$$J_i=\phi_i$$
\end{proposition}
See \cite{bar} for the part of the statement not involving descendent invariants, and \cite{iritani} for a more direct proof. The functions $\phi_{i,t}({\bf t})$ can be thought of as specifying a new set of coordinates on the moduli space; it is this change of coordinates that gives the isomorphism of the B-model Frobenius manifold with that arising in the A-model.  In Barannikov's formulation, this change of coordinates is difficult to make explicit and not immediately meaningful.  We will see that Gross's tropical methods make the transition very natural and explicit, providing a tropical interpretation of mirror symmetry.

\subsection{Tropical A-model}
The story here is the relatively long and extensive history of the tropical computation of Gromov-Witten invariants.  See Section \ref{section:NS}.  It's important to note that not all of the invariants appearing in the $J$ function have \emph{a priori} tropical interpretations.  In particular, tropical versions of descendent invariants of the type $\langle \psi^\nu T_i,  T_2, \ldots, T_2\rangle _{0,d}$ are, for $i\neq 2$, a result of the mirror symmetry construction outlined here.  The case where $i=2$ was previously treated by Markwig and Rau \cite{M+R}.

\section{Tropical B-model}
\subsection{Family of tropical Landau-Ginzburg potentials}
Recalling the role of the Landau-Ginzburg potential as discussed in \ref{bmodel}, we now outline the tropical version given in \cite{kan}.  The idea is to replace Barannikov's universal unfolding of $W$ with one that naturally relates to the flat coordinates $y_i$ on the A-model side.  Fukaya, Oh, Ohta, and Ono have shown that it is possible to construct a universal unfolding in terms of Maslov index 2 holomorphic disks \cite{FOOO}; there is a relationship between tropical disks and holomorphic disks \cite{Nishinou}.  Gross's construction defines a  universal deformation of $W$ in terms of Maslov index 2 tropical disks; the process of integration glues these disks together to form tropical curves (appearing on the A-model side of the picture).  In this process, the flat coordinates arise naturally and the mirror statement is a transparent combinatorial relationship.

Fix $k$ points $P_1, \ldots, P_k$ and a single point $Q$ in general position in $M_\mathbb{R}$.  In this context, general position can be achieved by choosing points for which the line connecting any pair is of irrational slope.    For the definitions of tropical curves, disks, and trees, see Section \ref{sec:tropics}.
\begin{definition}[$R_k$]
For each $P_i\in\{P_1, \ldots, P_k\}$ associate the variable $u_i$ in the ring:
$$R_k:=\frac{\mathbb{C}[u_1, \ldots, u_k]}{(u_1^2, \ldots, u_k^2)}$$
\end{definition}
For a tropical disk or tree $h$ in $(X_\Sigma, P_1, \ldots, P_k)$, define $I(h)\subseteq \{1, \ldots, k\}$ by 
$$I(h):=\{i | h(p_j)=P_i \text{ for some } j\}$$
\begin{definition}[$u_{I(h)}$]
Let $h$ be a tropical disk or tree in $(X_\Sigma, P_1, \ldots,  P_k)$.
Then
$$u_{I(h)}:=\prod_{i\in I(h)}u_{i}$$
\end{definition}
\begin{definition}[$Mono(h)$]
Let $h$ be a Maslov index 2 tropical disk with boundary $Q$ or Maslov index 0 tropical tree.  Then
$$Mono(h):=Mult(h)u_{I(h)}z^{\Delta(h)}\in \mathbb{C}[T_\Sigma]\otimes_\mathbb{C}R_k[[y_0]]$$
where $z^{\Delta(h)}\in\mathbb{C}[T_{\Sigma}]$ is the monomial associated to $\Delta(h)$.  We will often write $x_i$ for $z^{v_{\rho_i}}$.  See Figure \ref{fan}.
\end{definition}
\begin{figure}[h!]
\sidecaption[t]
\includegraphics[width=0.3\textwidth]{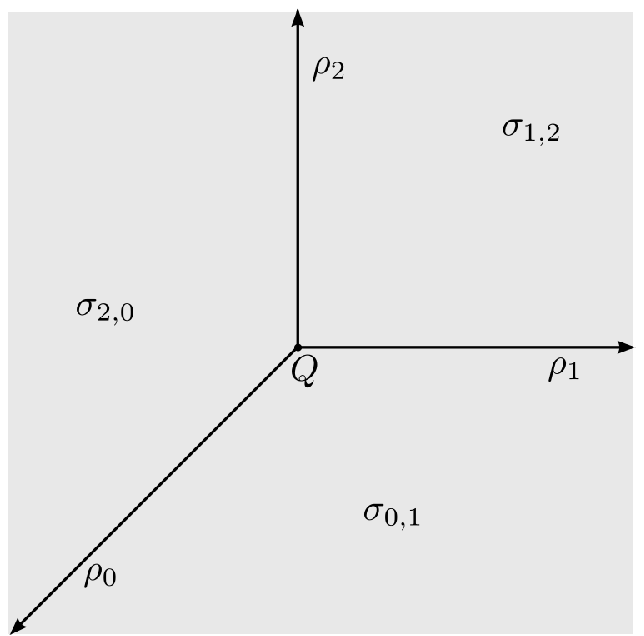} 
\caption{Toric fan for $\mathbb{P}^2$} 
\label{fan}
\end{figure}
\begin{definition}[$W_k(Q)$]
We define the $k$-pointed $n$-descendent Landau Ginzburg potential as 
$$W_{k}(Q):=y_0+\sum_h Mono(h)$$
where the sum is over all Maslov index 2 disks $h \in (X_\Sigma, P_{1}, \ldots,  P_{k})$ with boundary $Q$.
\end{definition}

\subsection{B-model tropical moduli}
Here we define Givental's B-model moduli space \cite{giv}, closely following the presentation in \cite{MSP2}.  

Fix a complete fan $\Sigma$ in $M_\mathbb{R}$ with $X_\Sigma$ a non-singular toric variety.

As the assumption of non-singularity implies the surjectivity of $r$, we have the following exact sequence:
$$0\rightarrow K_\Sigma\rightarrow T_\Sigma \rightarrow M\rightarrow 0$$
with the third arrow given by $r$ and $K_\Sigma$ its kernel.  
Dualizing over $\mathbb{Z}$ gives
$$0\rightarrow N \rightarrow {\rm Hom}_\mathbb{Z}(T_\Sigma, \mathbb{Z})\rightarrow {\rm Pic} X_\Sigma \rightarrow 0$$
Tensoring with $\mathbb{C}^\times$ gives the sequence

 $$0\rightarrow N\otimes\mathbb{C}^\times \rightarrow {\rm Hom}(T_\Sigma, \mathbb{C}^\times)\rightarrow {\rm Pic} X_\Sigma\otimes\mathbb{C}^\times \rightarrow 0$$
with the third arrow defining the map $\kappa$, providing the family of mirrors to $X_\Sigma$.  
 Set $$\check{\mathcal{X}}:=\Hom(T_\Sigma, \mathbb{C}^\times)=\spec\mathbb{C}[T_\Sigma].$$

The \emph{K\"ahler moduli space} of $X_\Sigma$ is defined to be 
$$ \mathcal{M}_\Sigma:={\rm Pic}X_\Sigma\otimes \mathbb{C}^\times=\spec \mathbb{C}[K_\Sigma]$$
Of course, this is very simple in our case with $K_\Sigma\cong\mathbb{Z}$.  Note that $\kappa$, by definition, is now a map:
$$\kappa:\spec \mathbb{C}[T_\Sigma]\rightarrow \mathcal{M}_\Sigma$$
A fiber of $\kappa$ over a closed point of $\mathcal{M}_\Sigma$ is isomorphic to $ \spec\mathbb{C}[M]$.  

Define the {\textit k-order thickening} of the K\"ahler moduli space $\mathcal{M}_\Sigma$ to be the ringed space
$$\mathcal{M}_{\Sigma, k}:=(\mathcal{M}_\Sigma, \mathcal{O}_{\mathcal{M}_{\Sigma, k}})$$
where  $\mathcal{O}_{\Sigma, k}(U)$ for $U\subseteq \mathcal{M}_\Sigma$ given by expressions of the form
$$\sum_{\substack{n=0\\
I\subseteq\{1,\ldots, k\}}}^\infty f_{n, I}y_0^nu_I$$
where $u_I\in R_k$, $f_{n,I}$ is a holomorphic function on $U$ for each $n$ and $I$ and there are only a finite number of terms for each $n$.\\
The {\textit k-order thickening of the mirror family} $\check{\mathcal{X}}_{\Sigma, k}:=(\check{\mathcal{X}}_\Sigma, \mathcal{O}_{\check{X}_{\Sigma, k}})$ is defined similarly, giving us a family
$$\kappa:\check{X}_{\Sigma, k}\rightarrow \mathcal{M}_{\Sigma, k}$$
In our particular example, writing $x_i$ for the monomial $z^{v_\rho}\in \mathbb{C}[T_\Sigma]$, it's easy to see that $\kappa$ is a map $\kappa:(\mathbb{C}^\times)^3\rightarrow \mathbb{C}^\times$ with 
$$\kappa(x_0, x_1, x_2)=x_0x_1x_2.$$
The relevance of this discussion to our earlier constructions is clear;  $W_{k}(Q)$ is, by construction, a regular function on $\check{X}_{\Sigma, k}$.  We can think of this map as providing a family of Landau-Ginzburg potentials.

The sheaf of relative differentials $\Omega^1_{\check{X}_{\Sigma, k}/\tilde{M}_{\Sigma, k}}$ is canonically isomorphic to the trivial locally free sheaf $M\otimes_\mathbb{Z} \mathcal{O}_{\check{X}_{\Sigma, k}}$, with $m\otimes 1$ corresponding to the differential
$${\rm dlog} m := \frac{d(z^{\overline{m}})}{z^{\overline{m}}}$$
where $\overline{m}$ is any lift of $m\in M$ to $T_\Sigma$ under the map $r$ and ${\rm dlog} z^{\overline{m}}$ is well defined as a relative differential independent of the choice of the lift.  Thus, a choice of generator $\wedge^2 M\cong \mathbb{Z}$ determines a nowhere-vanishing relative holomorphic two-form $\Omega$, which is, up to sign, canonical.  Explicitly, if $e_1$, $e_2$ is a positively oriented basis of $M$, we choose
$$\Omega := {\rm dlog} e_1\wedge {\rm dlog} e_2$$
\subsection{Automorphisms}
There is an obvious dependence on the potential $W_k$ on the position of the points $Q, P_1, \ldots, P_k$; significantly, the changes induced by different choices of points are restricted to those given by the action of a particularly nice group. 
\begin{definition}[$\mathbb{V}_{\Sigma, k}$]
$\mathbb{V}_{\Sigma, k}$ is the group of automorphisms of  $\mathbb{C}[T_\Sigma]\otimes_\mathbb{C}R_k[[y_0]]$ generated by elements of the form ${\rm exp}(cu_Iz^m\otimes n )$, whose action is given by:
$${\rm exp}(cu_Iz^m\otimes n )(z^{m'})=z^{m'}(1+cu_I\langle n, r(m')\rangle z^m)$$
\end{definition}
The generators of this group preserve our choice of $\Omega$; in fact, the original version of this group was defined as a group of Hamiltonian symplectomorphisms. 
 \subsection{Scattering diagrams}
The essential tool for understanding the dependence of $W_{k}(Q)$ on $Q\in M_\mathbb{R}$ is the \emph{scattering diagram}.  The definition we shall give, from \cite{kan}, has broad generalizations, but in this situation the underlying idea is very concrete and intuitively appealing.  One defines a collection of rays and lines (\emph{walls}) in the plane, each with an attached function in  $\mathbb{C}[T_\Sigma]\otimes_\mathbb{C}R_k[[y_0]]$.  Given the data of a wall and an attached function, one one can give an automorphism in $\mathbb{V}_{\Sigma, k}$ defined by crossing the wall in either of the possible directions.
\begin{definition} \cite{kan}
Fix $k\geq 0$.
\begin{enumerate}
\item A \emph{ray} or \emph{line} is a pair $(\mathfrak{d}, f_\mathfrak{d})$ such that
\begin{itemize}
\item $\mathfrak{d}\subseteq M_\mathbb{R}$ is given by 
$$\mathfrak{D}=m_{0} ^{'}-\mathbb{R}_{\geq 0} r(m_0)$$
if $\mathfrak{d}$ is a ray and 
$$\mathfrak{d} = m_{0}^{'} -\mathbb{R} r(m_0)$$
if $\mathfrak{d}$ is a line, for some $m_{0}^{'}\in M_\mathbb{R}$ and $m_0\in T_\Sigma$ with $r(m_0)\neq0$.  The set $\mathfrak{d}$ is called the \emph{support} of the line or ray.  If $\mathfrak{d}$ is a ray, $m_0^{'}$ is called the \emph{initial point} of the ray, written as $Init(\mathfrak{d})$.
\item $f_\mathfrak{d} \in \mathbb{C}[z^{m_0}] \otimes _\mathbb{C} R_k \subseteq \mathbb{C}[T_\Sigma]\otimes_\mathbb{C} R_k[[y_0]]$.
\end{itemize}
\item A \emph{scattering diagram} $\mathfrak{D}$ is a finite collection of lines and rays.
\end{enumerate}
\end{definition}
If $\mathfrak{D}$ is a scattering diagram, we write
$$Supp(\mathfrak{D}):=\cup_{\mathfrak{d}\in \mathfrak{D}} \mathfrak{d} \subseteq M_\mathbb{R}$$
and
$$Sing(\mathfrak{D}):=\bigcup_{\mathfrak{d}\in \mathfrak{D}} \partial \mathfrak{d} \cup \bigcup_{\substack{\mathfrak{d}_1, \mathfrak{d}_2\\ {\rm dim} \mathfrak{d}_1\cap \mathfrak{d}_2=0}} \mathfrak{d}_1\cap \mathfrak{d}_2$$
where $\partial\mathfrak{d}=\{Init(\mathfrak{d})\}$ if $\mathfrak{d}$ is a ray, and empty if it is a line.

\begin{definition}[$\theta_{\gamma, \mathfrak{D}}\in \mathbb{V}_{\Sigma, k}$]
Given a scattering diagram $\mathfrak{D}$ and smooth immersion $\gamma:[0,1]\rightarrow M_\mathbb{R}\setminus Sing(\mathfrak{D})$ whose endpoints are not in $Supp(\mathfrak{D})$, with $\gamma$ intersecting $Supp(\mathfrak{D})$ transversally,  this information defines a ring automorphism $\theta_{\gamma,\mathfrak{D}}\in \mathbb{V}_{\Sigma, k}$.
First, find numbers
$$0<t_1\leq t_2\leq \ldots \leq t_s <1$$
and elements $\mathfrak{d}_i$ such that $\gamma(t_i)\in \mathfrak{d}_i$, $\mathfrak{d}_i\neq \mathfrak{d}_j$ if $i\neq j $ and $s$ is taken to be as large as possible to account for all elements of $\mathfrak{D}$ that are crossed by $\gamma$.  For each $i\in \{1,\ldots, s\}$, define $\theta_{\gamma,\mathfrak{d}_i}\in \mathbb{V}_{\Sigma, k}$ to be the automorphism with action
\begin{align*}
\theta_{\gamma, \mathfrak{d}_i}(z^m)&= z^mf_{\mathfrak{d}_i}^{\langle n_0, r(m) \rangle}\\
\theta_{\gamma, \mathfrak{d}_i}(d)&= d
\end{align*}
for $m\in T_\Sigma$, $d\in R_k [[y_0]]$, where $n_0\in N$ is chosen to be primitive, annihilating the tangent space to $\mathfrak{d}_i$ and satisfying
$$\langle n_0, \gamma'(t_i)\rangle <0$$
Then $\theta_{\gamma, \mathfrak{D}}:=\theta_{\gamma, \mathfrak{d}_s}\circ\cdots \circ \theta_{\gamma, \mathfrak{d}_1}$, where composition is taken from right to left.

\end{definition}

In our particular example, we construct our walls from the outgoing edges of Maslov index 0 trees and attach functions determined by the degree, multiplicity, and marked points of the corresponding tree.  Given a general choice of $P_1, \ldots, P_k$, there should be a finite set $Trees(\Sigma, P_1, \ldots, P_k)$ of Maslov index zero trees in $X_\Sigma$ with the property that each maps its marked points to some subset of $\{P_1, \ldots, P_k\}$. 
\begin{definition}\cite{kan}
We define $\mathfrak{D}(\Sigma, P_1, \ldots, P_k)$. to be the scattering diagram which contains one ray for each element $h$ of $Trees(\Sigma, P_1, \ldots, P_k)$,  The ray corresponding to $h$ is of the form $(\mathfrak{d}, f_\mathfrak{d})$, where
\begin{itemize}
\item $\mathfrak{d} = h(E_{out})$.
\item $f_\mathfrak{d} = 1 + w_\Gamma(E_{out})Mono(h)$, where $w_\Gamma(E_{out})$ is the weight of the outgoing edge $E_{out}$.
\end{itemize}
\end{definition}

When the outgoing edges of two trees meet, one can construct a new tree by gluing them together and attaching an appropriate outgoing edge. This outgoing edge corresponds to a ray in the scattering diagram $\mathfrak{D}$ (see the lower left wall in Figure \ref{fig:scat0}).  It is this process that inspired the term ``scattering."  This property automatically induces a very nice feature of $\mathfrak{D}$:  the automorphism defined by going around a loop of any (unmarked) vertex in our scattering diagram is the identity.  In other examples of scattering diagrams, walls will need to be added at intersection points to ensure this phenomenon \cite{KS06}.
\begin{proposition}\cite{kan}
Let $P_1, \ldots, P_k$ be chosen generally.  If 
$$P\in {\rm Sing}(\mathfrak{D}(\Sigma, P_1, \ldots, P_k))$$ is a singular point with $P\notin \{P_1, \ldots, P_k\}$, and $\gamma_P$ is a small loop around $P$, then $\theta_{{\gamma_P}, \mathfrak{D}(\Sigma, P_1, \ldots, P_k)}={\rm Id}.$
\end{proposition}
\subsection{Broken lines}
Once we have assembled a scattering diagram, the Maslov index 2 disks with a particular endpoint $Q$ can be found by analyzing objects called broken lines.  The precise definition (given in Section 5.4.4 of \cite{kan}) is not necessary for this exposition, but the idea is quite simple.  One begins with a line of slope equal to one of elements of $\Sigma^{[1]}$ in $M_\mathbb{R}$ far away from our chosen points in the plane. Label the line with the monomial associated to its element of $T_\Sigma$, and begin traveling along the line (in the direction opposite that specified by the monomial) until reaching a wall of the scattering diagram.  At this point, you can either choose to bend the line in a fashion dictated by the wall while appropriately adjusting the attached monomial or continue on undisturbed.  If you end up hitting $Q$ after some time, you've discovered a \emph{broken line with endpoint Q}.  Recalling that each of the walls of our scattering diagram correspond to a set of Maslov index 0 disks, the process of constructing a broken line can be thought of as taking a stem (the broken line) and attaching a set of disks corresponding to the walls at which the line bends.  It turns out that each Maslov index 2 disk can be decomposed in such a fashion, giving us the following useful result.
\begin{proposition} If $Q\notin {\rm Supp}(\mathfrak{D}(\Sigma, P_1, \ldots, P_k))$ is general, then there is a one-to-one correspondence between broken lines with endpoint $Q$ and Maslov index 2 disks with boundary $Q$.  In addition, if $\beta$ is a broken line corresponding to a disk $h$, and $cz^m$ is the monomial associated to the last segment of $\beta$, then
$$cz^m={\rm Mono}(h)$$
\end{proposition}
\subsubsection{Examples}
See Figures \ref{fig:scat0}, \ref{msdisk0}, and \ref{fig:scat1}.

\begin{figure}
\sidecaption[t]
\includegraphics[width=.6\textwidth]{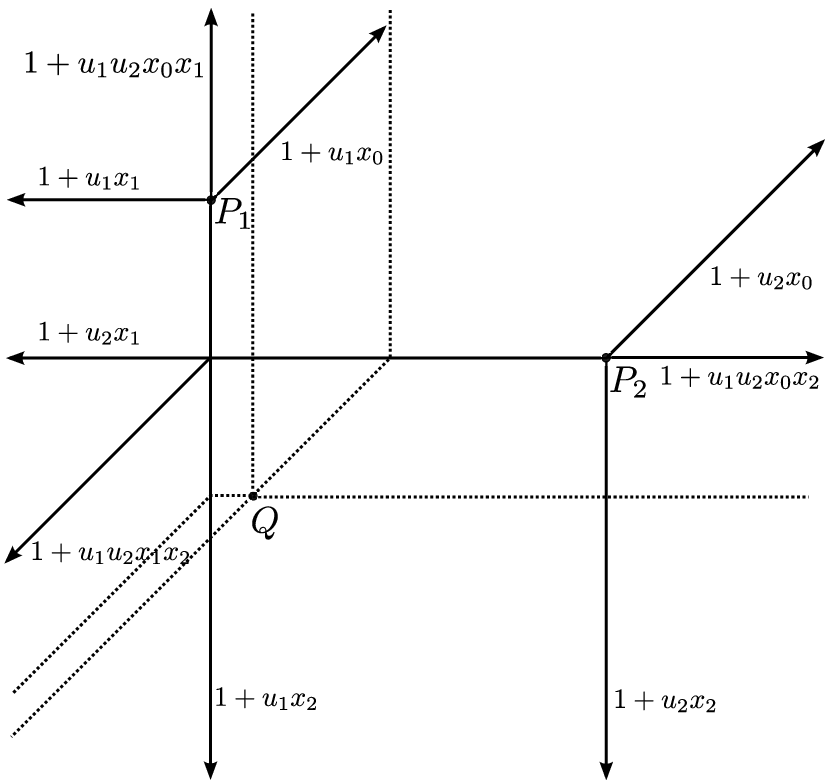}
 \caption{The scattering diagram for $W_2(Q)$ and this particular arrangement of points.  Broken lines are shown dotted.  The monomials corresponding to the broken lines are (beginning with that in the 12 o'clock position and proceeding clockwise):  $x_2$, $u_2x_1x_2$, $x_1$, $x_0$, $u_1x_0x_2$.}
\label{fig:scat0}
\end{figure}

\begin{figure}
\centering
\begin{tabular}{ccc}
\includegraphics[scale=.5]{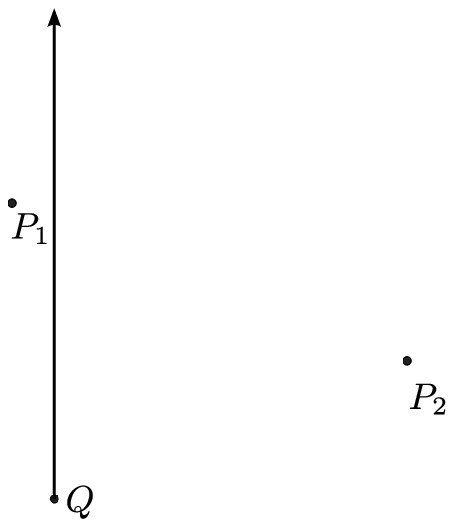}&
 \includegraphics[scale=.5]{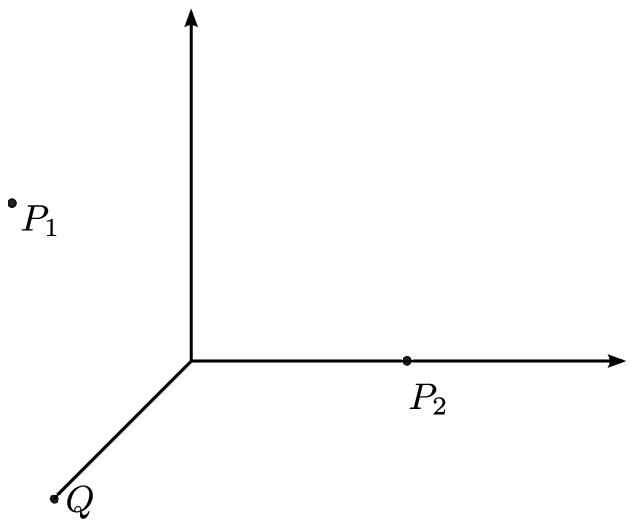}
 \includegraphics[scale=.5]{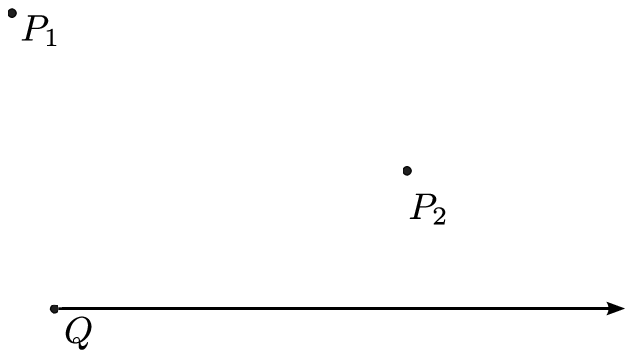}
\end{tabular}
\begin{tabular}{cc}
\includegraphics[scale=.5]{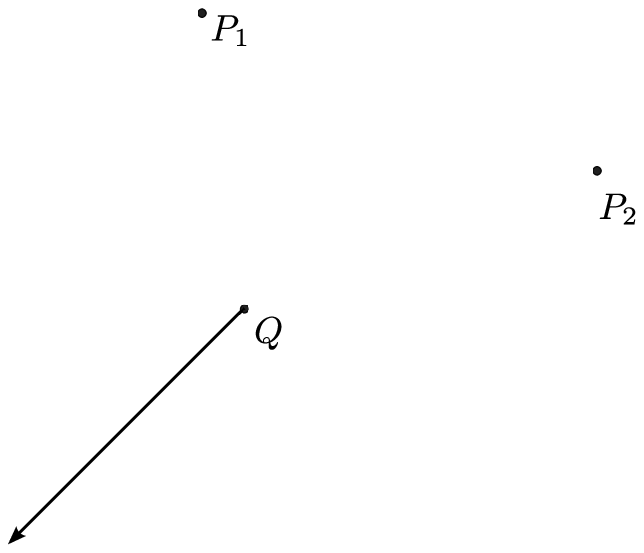}&
\includegraphics[scale=.5]{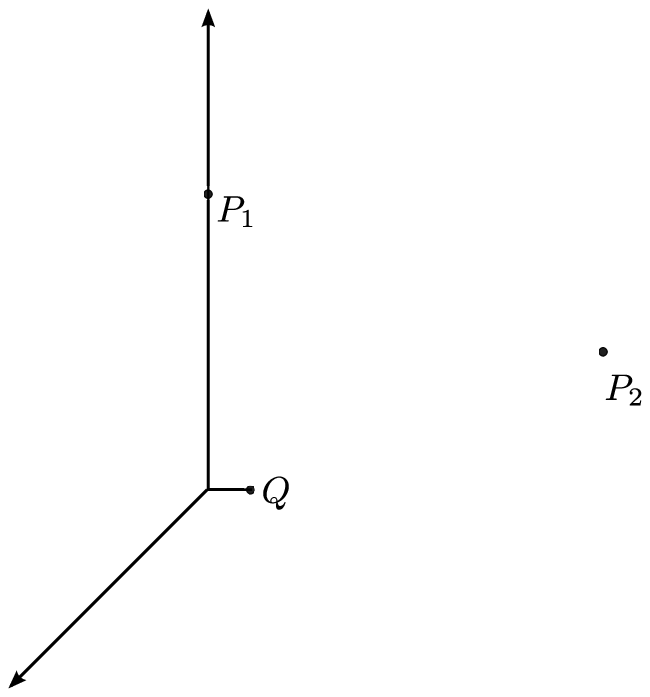}
\end{tabular}
\caption{Maslov index two disks corresponding to the broken lines in Figure \ref{fig:scat0}.}
\label{msdisk0}
\end{figure}

\begin{figure}
\sidecaption[t]
\includegraphics[width=.6\textwidth]{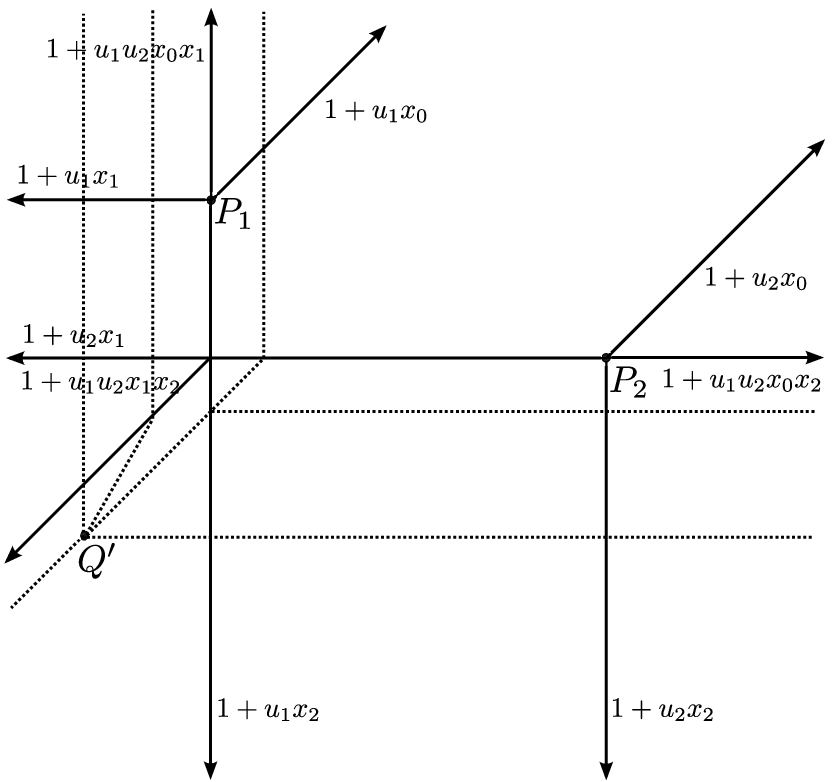}
\caption{The scattering diagram for $W_2(Q')$ and the same choice of $P_i$ as in Figure \ref{fig:scat0}. The monomials corresponding to the broken lines are (beginning with that in the 12 o'clock position relative to $Q'$ and proceeding clockwise):  $x_2$, $u_1u_2x_1x_2$, $u_2x_1$, $u_1x_2$, $x_1$, $x_0$.  For an explanation of the functions attached to the walls, see \cite{kan}, Section 5.4.3.}
\label{fig:scat1}
\end{figure}

\subsection{Tropical invariants}
In order to discuss the results of the period integrals, we must first give a notion of the tropical versions of the Gromov-Witten invariants involved in Givental's $J$-function for $\mathbb{P}^2$.  The exact definitions are not particularly illuminating, but the basic idea is essential to understanding our mirror symmetry construction.  In order to understand what type of curves contributing to these invariants, we must define a slightly different moduli space of parametrized tropical curves than was explored in Section \ref{sec:enum}. 

\begin{definition}[$\mathcal{M}_{0,k+1}(\Sigma, \Delta, P_1, \ldots, P_k, \psi^{\nu} S)$]
Let $P_1, \ldots, P_k \in M_\mathbb{R}$ be general.  Let $S\subseteq M_\mathbb{R}$.  Define
$$\mathcal{M}_{0,k}(\Sigma, \Delta, P_1, \ldots, P_k, \psi^{\nu} S)$$
to be the moduli space of rational $(k+1)$-pointed tropical curves in $X_\Sigma$, $$h:(\Gamma , p_1, \ldots, p_k, x) \rightarrow M_\mathbb{R}$$ of degree $\Delta$ such that
\begin{itemize}
\item $h(p_j)=P_{j}$, $1\leq j\leq k$.
\item $h(x)\in S$.
\item If $E_{x}$ shares a vertex $V_{j}$ with $E_{p_j}$, then
$$Val(V_j)=3+\nu$$
and the valency of the vertex $V_i$  attached to $E_{p_i}$ for $i\neq j$ is given by
$$Val(V_j)=3$$
\item Otherwise, the valency of the vertex $V_{x}$ attached to $E_{x}$ is given by 
$Val(V_{x})=\nu+3$ and $Val(V_j)=3$ for $1\leq j \leq k$.   
\item The weight of each unbounded edge of $\Gamma$ is either $0$ or $1$.  Note that all unmarked, unbounded edges must have weight 1 and be translates of elements of $\Sigma_{[1]}$.
\end{itemize}
\end{definition}

For compactness of notation, we depart slightly from the notation of \cite{kan}.  Let $S_0\subseteq M_\mathbb{R}$ be the set $\{Q\}$, $S_1=L\subseteq M_\mathbb{R}$ the tropical line with vertex $Q$ (the tropical curve given by attaching unbounded rays in the direction of $(-1-1)$, $(1,0)$, and $(0,1)$ to $Q$), and $S_2=M_\mathbb{R}$ .  Gross defines tropical invariants of the form
\begin{eqnarray*}
\langle P_1, \ldots, P_k, \psi^\nu S_i\rangle _{0,d}^{trop}
\end{eqnarray*}
with $3d-\nu - k +(2-i)=0$.  These are meant to be (and, as we shall see, are) equal to the corresponding classical Gromov-Witten invariants of the form $\langle \overbrace{T_2, \ldots, T_2}^k, \psi^\nu T_{2-i}\rangle_{0,d}$ for $\mathbb{P}^2$.   The tropical invariants are defined by summing the contributions of curves in $\mathcal{M}_{0,k+1}(\Sigma, \Delta, P_1, \ldots, P_k, \psi^{\nu-j} S_{i-j})$ for $0\leq j \leq i$ with the appropriate (and quite complicated) multiplicities.  For the precise definitions, see Section 5.2 of \cite{kan}.  Each of the tropical curves contributing to these invariants are glued from tropical disks and trees,  objects with a close correspondence to terms appearing in the tropical Landau-Ginzburg potential.  This is the connection that binds the A- and B-models in this construction.  See Figures \ref{curve1} and \ref{curve2} for examples of tropical curves relevant to these invariants.

\begin{figure}[h!]
\centering
\def\svgwidth{.2\textwidth}
\begin{tabular}{cc}
\includegraphics[scale=.4]{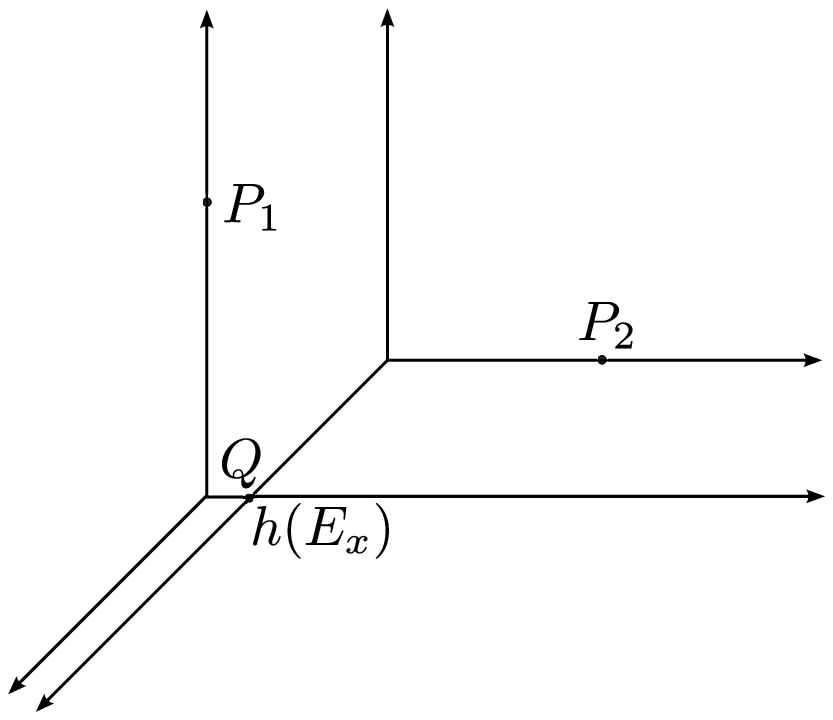}&
\includegraphics[scale=.4]{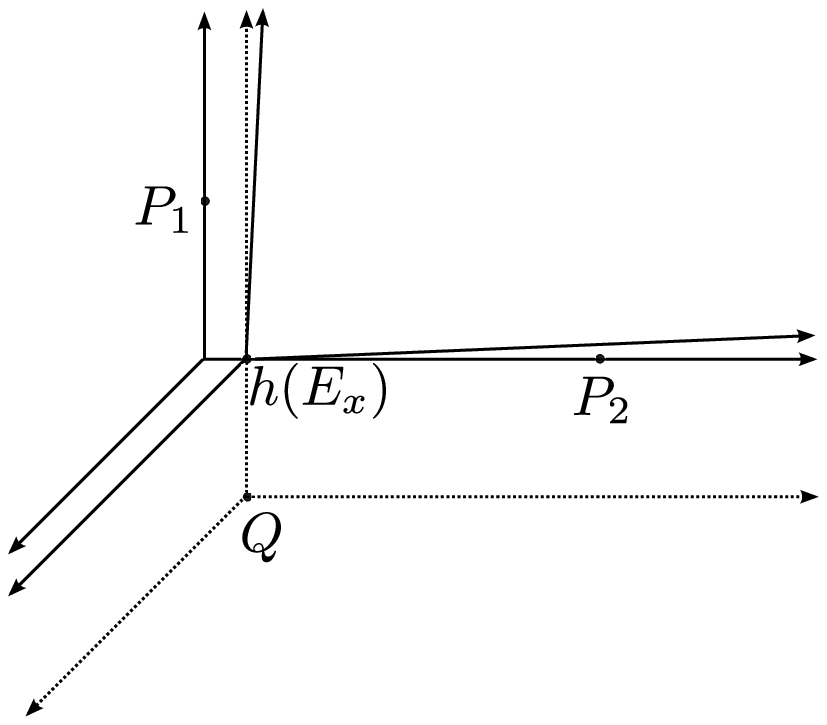}\\
\includegraphics[scale=.4]{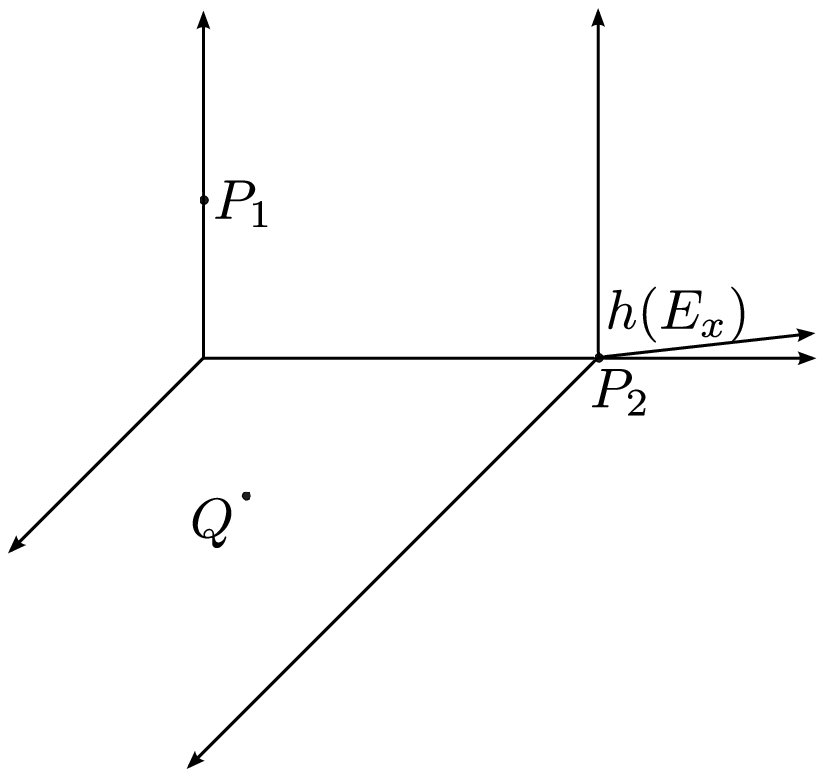}&
\includegraphics[scale=.4]{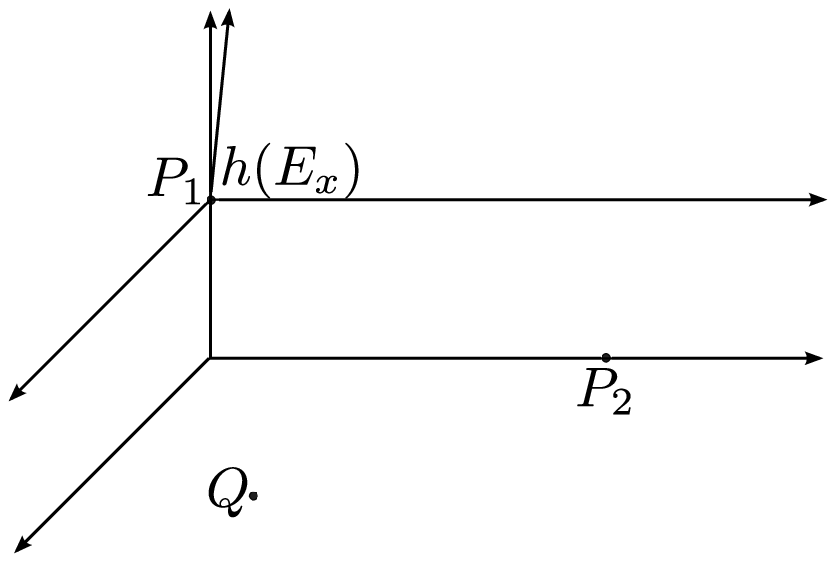}
\end{tabular}
\caption{Tropical curves \ contributing to $\langle P_1, P_2, \psi^2 S_0\rangle_{0,2}^{trop}$, $\langle P_1, P_2, \psi^3 S_1\rangle_{0,2}^{trop}$, and $\langle P_1, P_2, \psi^4 S_2\rangle_{0,2}^{trop}$.  Edges have been drawn as perturbed from their true direction when necessary for clarity.}
\label{curve1}
\end{figure}
\begin{figure}[h!]
\centering
\def\svgwidth{.2\textwidth}
\begin{tabular}{cc}
\includegraphics[scale=.4]{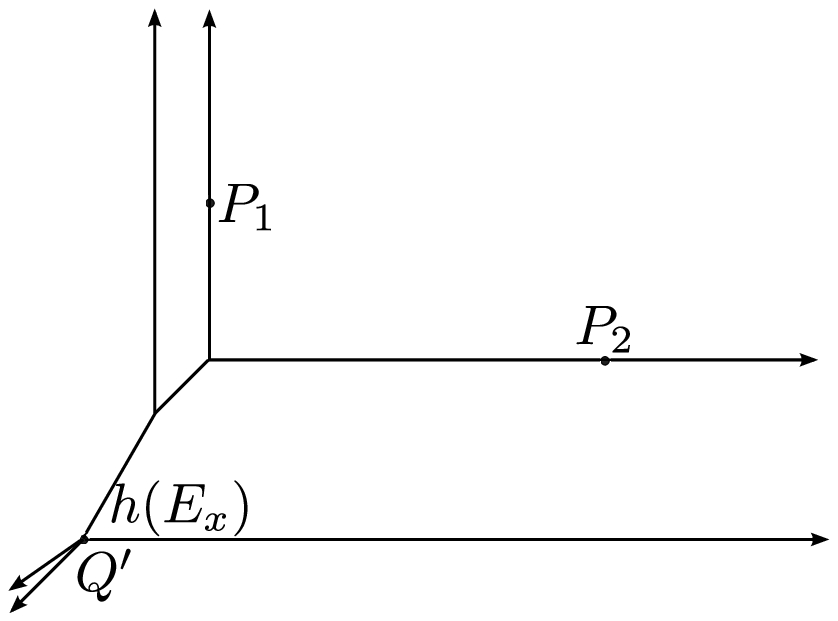}&
\includegraphics[scale=.4]{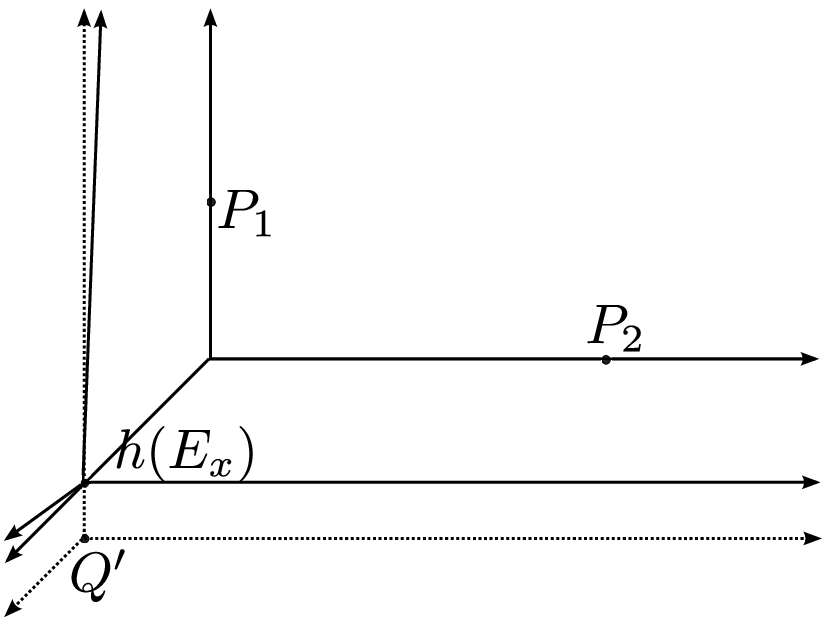}\\
\includegraphics[scale=.4]{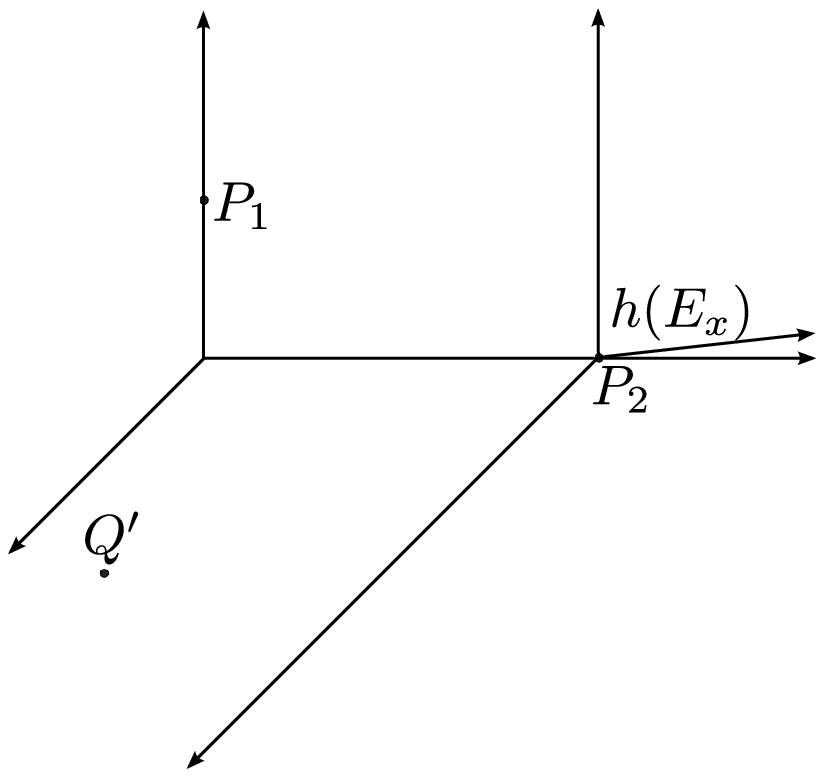}&
\includegraphics[scale=.4]{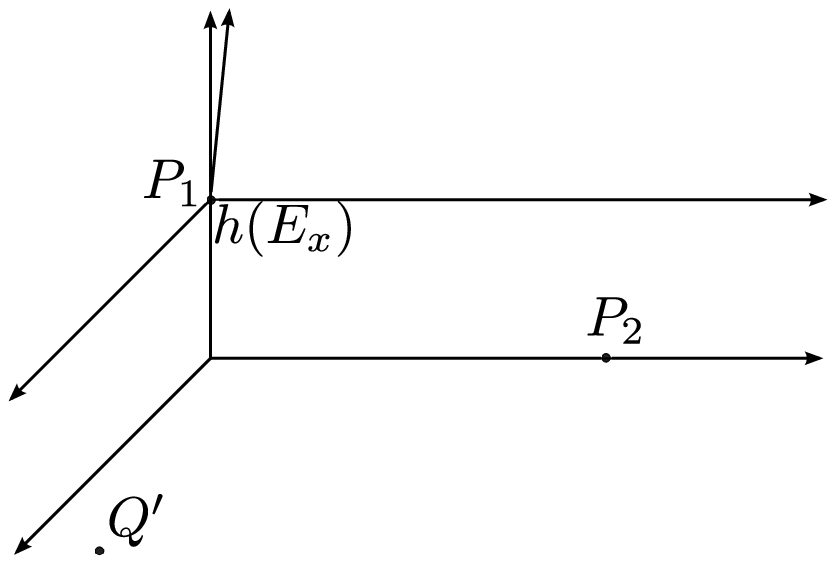}
\end{tabular}
\caption{Tropical curves contributing to $\langle P_1, P_2, \psi^2 S_0\rangle_{0,2}^{trop}$, $\langle P_1, P_2, \psi^3 S_1\rangle_{0,2}^{trop}$, and $\langle P_1, P_2, \psi^4 S_2\rangle_{0,2}^{trop}$ with basepoint $Q'$.}
\label{curve2}
\end{figure}
\subsection{Evaluation of integrals}
Through the evaluation of period integrals, the tropical objects controlling the Landau-Ginzburg model are assembled into tropical curves representing A-model invariants.  This is the punchline of the construction.
Here we return to the setup of language of \ref{bmodel}.  Let $\mathcal{R}$ be the local system on $\mathcal{M}_{\Sigma,k}\times \mathbb{C}^*$ whose fiber over $(u,\hbar)$ is given by 
$$H_2(\kappa^{-1}(u), {\rm Re}(W_0(Q)/\hbar)\ll 0).$$
Note that this local system is unconcerned with our thickening by the ring $R_k$.  Gross shows that it's possible to find a local basis $\Xi_0, \Xi_1, \Xi_2$ of $\mathcal{R}$ satisfying Barannikov's conditions such that the integrals $\int_\Xi e^{W_0(Q)/\hbar} \Omega$ take on a particular form.  We proceed by writing
$$\exp(W_k(Q)/\hbar)=\exp(W_0(Q))\exp((W_k(Q)-W_0(Q))/\hbar)$$
and expanding the latter part into a finite power series.  This term corresponds to gluings of the finite number of Maslov index 2 tropical disks which pass through marked points.  The variables $u_i$ square to zero, so only a finite number of such gluings result in nonzero monomials.
Using the properties of the sections $\Xi_i$, one can show that 
$$\sum_{i=0}^2 \alpha^i \int_{\Xi_i} e^{(x_0+x_1+x_2)/\hbar} x_0^{n_0} x_1^{n_1} x_2^{n_2}\Omega = \hbar^{-3\alpha}e^{\alpha y_1} \sum_{i=0}^2 \psi_i(n_0, n_1, n_2)\alpha^i,$$
where
$$\psi_i(n_0, n_1, n_2)=\sum_{d=0}^\infty D_i(d, n_0, n_1, n_2) \hbar^{-(3d-n_0-n_1-n_2)}e^{dy_1}$$
and the $D_i$ are some explicit numerical quantities.  With this result and the explicit dependence of $W_k$ on the scattering, the problem becomes combinatorial in nature.  The key to understanding the integral is to first break the finite expansion of $\exp((W_k(Q)-W_0(Q))/\hbar)$ into several sums and showing that, selecting one of these sums, we can make the resulting contribution to the integral be zero if we move $Q$ out toward infinity in an appropriate direction.  The structure of the scattering diagram is used to study how these contributions change as $Q$ moves back in from infinity.  The resulting terms can be interpreted as tropical curves.  
As can be seen by comparing Figures \ref{fig:scat0} and \ref{fig:scat1}, there is a clear dependence $W_{k}(Q)$ on $Q$.   As the choices of $Q$ and $P_i$ vary, Gross shows that $W_k$ is transformed by elements of $\mathbb{V}_{\Sigma, k}$; this results from the combinatorial properties of the scattering diagrams used to define the potential.  It's easy to show that the action of such an element on $W_k$ preserves the result of our desired integral.  The result of this analysis, as given in  \cite{MSP2}, is the following direct relationship between A-model and B-model data:
\section{Mirror symmetry}
\begin{theorem}
\label{msthm}
A choice of general points $P_1, \ldots, P_k$ and $Q$ gives rise to a function $W_k(Q)\in \mathbb{C}[T_\Sigma]\otimes_\mathbb{C} R_k[[y_0]]$ and hence  a family of Landau-Ginzburg potentials on the family $\kappa:\check{\mathcal{X}}_{\Sigma, k}\rightarrow \mathcal{M}_{\Sigma, k}$ with a relative nowhere-vanishing two form $\Omega$ as defined before.  This data gives rise to a local system $\mathcal{R}$ on $\mathcal{M}_{\Sigma, k}\otimes \spec \mathbb{C}[\hbar, \hbar^{-1}]$ whose fiber over $(\kappa, \hbar)$ is given by $H_2((\check{\mathcal{X}}_{\Sigma, k})_\kappa, \Real(W_0/\hbar)\ll 0)$.  There exists a multi-valued basis $\Xi_0$, $\Xi_1$, $\Xi_2$ of sections of $\mathcal{R}$ satisfying the conditions of the introduction such that
$$\sum_{i=0}^2 \alpha^i \int_{\Xi_i} e^{W_k(Q)/\hbar}\Omega = \hbar^{-3\alpha} \sum_{i=0}^2 \phi_i (\alpha \hbar)^i$$
with
$$(\phi_i(y_0, y_1, u_1, \ldots, u_k, \hbar^{-1})=\delta_{0,i}+\sum_{j=1}^\infty \phi_{i,j}(y_0,y_1`, u_1, \ldots, u_k)\hbar^{-j}$$
for $0\leq i \leq 2$, with
\begin{eqnarray*}
\phi_{0,1}&=y_0\\
\phi_{1,1}&=y_1:=\log(\kappa)\\
\phi_{2,1}&= y_2 :=\sum_{i=1}^k u_i.
\end{eqnarray*}
Furthermore,
$$\phi_i=J_i^{trop} (y_0, y_1, y_2).$$
Where
\begin{eqnarray*}
J^{trop}_{\mathbb{P}^2}(y_0, y_1, y_2, \hbar)&:=&\exp\left(\frac{y_0T_0+y_1T_1}{\hbar}\right)\cup\Bigg(T_0+\sum_{i=0}^2\bigg(y_2\hbar^{-1}\delta_{2,i}\\
&+&\sum_{d\geq 1}\sum_{\nu\geq 0} \langle T_2^{3d+i-2-\nu}, \psi^\nu T_{2-i}\rangle^{trop}_{0,d}\hbar^{-(\nu+2)}e^{dy_1}\frac{y_2^{3d+i-2-\nu}}{(3d+i-2-\nu)!}\bigg)T_i\Bigg)\\
&=:&\sum_{i=0}^2 J_i^{trop}T_i
\end{eqnarray*}
\end{theorem}
There is an immediate corollary.
\begin{corollary}
Let $\mathfrak{M}_{\Sigma, k}$ be the formal spectrum of the completion of $\mathbb{C}[K_\Sigma]\otimes_\mathbb{C} R_{k} [[y_0]]$ at the maximal ideal $(y_0, \kappa-1, \{u_i\})$.  The completion is isomorphic to $\mathbb{C}[[y_0, y_1]]\otimes_\mathbb{C} R_{k}$ with $y_1:=\log\kappa$, the latter expanded in a power series at $\kappa=1$.  Let
$$\check{\mathfrak{X}}_{\Sigma, k}=\check{\mathcal{X}}_{\Sigma, k}\times _{\mathcal{M}_{\Sigma, k}} \mathfrak{M}_{\Sigma, k}.$$
The function $W_{k}(Q)$ is regular on $\mathfrak{X}_{\Sigma, k}$ and restricts to $W_0(Q)=x_0+x_1+x_2$ on the closed fiber of $\check{\mathfrak{X}}_{\Sigma, k}\rightarrow  \mathfrak{M}_{\Sigma, k}$ and hence gives a deformation of this function over $ \mathfrak{M}_{\Sigma, k}$.  Thus we have a morphism from  $\mathfrak{M}_{\Sigma, k}$ to the universal unfolding moduli space $\spec\mathbb{C}[[y_0, y_1, y_2]]$.  This map is given by:
\begin{eqnarray*}
y_0&\mapsto y_0\\
y_1&\mapsto \log(\kappa)\\
y_2& \mapsto \sum_{i} u_i
\end{eqnarray*}
\end{corollary}
Furthermore, we have the equivalence between the classical accuracy of Gross's tropical descendent invariants and Proposition \ref{msp2} (mirror symmetry for $\mathbb{P}^2$).  More precisely, consider the following proposition:
\begin{proposition}
\label{eq}
$J_{\mathbb{P}^2}^{trop} = J_{\mathbb{P}^2}$.
\end{proposition}
We have the following as a corollary of Theorem \ref{msthm}.
\begin{corollary}
Proposition \ref{msp2} and Proposition \ref{eq} are equivalent.
\end{corollary}

\section{Further reading}
\label{sec:further}
This chapter has given mention to topics appearing in a  wide swath of literature, and there are many connected works for the interested reader to explore.  As mentioned in the introduction, an excellent survey of the relationship between the Strominger-Yau-Zaslow conjecture and the Gross-Siebert program can be found in \cite{gsyz}.  This article serves as a helpful reading guide for much of the literature surrounding this topic.  Another valuable source of insight into the philosophy of the program can be found in the article giving its announcement \cite{GS03}.  

For a more in depth treatment of log geometry, the reader is recommended the relevant chapter in the book \cite{kan} by Gross. 
This source has the advantage to be tailored towards the Gross-Siebert program. Log differential forms in the Gross-Siebert program are treated in \cite{GS10,Ru10}.
Concerning logarithmic Gromov-Witten invariants, the foundational paper \cite{loggw} by Gross and Siebert defines the relevant moduli space.

There are many good introductions to tropical geometry.  For an entertaining and insightful overview, see the lecture of Maxim Kontsevich given at the Fields Institute \cite{Ko13}.  The application of the field to enumerative geometry was spearheaded by Grigory Mikhalkin \cite{mik}; our exposition is based on \cite{NS06} and \cite{kan}. Welschinger Invariants are treated in \cite{IKS03,Sh06}. Significant further progress has been made by Allerman, Markwig, and Rau, among others \cite{AR10} \cite{M+R}.  The latter works establish a tropical intersection theory whose analysis significantly expands the range of Gromov-Witten theory invariants calculable via tropical methods. 

Another application of tropical geometry to mirror symmetry, in this case the elliptic curve, is given by Boehm, Bringmann, Buchholz, and Markwig  in \cite{BBBM13}.  
As repeatedly mentioned, a much more comprehensive source for the material given in Section \ref{sec:p2ms} can be found in Gross's book \cite{kan}, while the author gives a more concise description in an article \cite{MSP2}.  Chapter 6 of the book also contains very explicit and concrete description of the details of the Gross-Siebert program in dimension two.  Some of the tools used in this construction, specifically scattering diagrams and broken lines, seem to have a very rich structure with a number of deep connections beyond this particular context.  For a discussion of the relationship with the so-called ``wall crossing structures" of Kontsevich and Soibelman, see Section 10 of \cite{KS13}.  An application to cluster algebras is forthcoming in work by Gross, Hacking, Keel and Kontsevich.

\bibliography{bib}

\begin{thebibliography}{10}

\bibitem{Chen-Abram-DF}
D.~Abramovich and Q.~Chen.
\newblock Stable logarithmic maps to {D}eligne-{F}altings pairs {II}.
\newblock {\em To appear in Asian J. Math.}, arXiv:1102.4531 (2011).

\bibitem{AR10}
Lars Allermann and Johannes Rau.
\newblock First steps in tropical intersection theory.
\newblock {\em Mathematische zeitschrift}, 264(3):633--670, 2010.

\bibitem{A07}
Denis Auroux.
\newblock Mirror symmetry and {T}-duality in the complement of an anticanonical
  divisor.
\newblock {\em arXiv preprint arXiv:0706.3207}, 2007.

\bibitem{bar}
S~Barannikov.
\newblock Semi-infinite hodge structures and mirror symmetry for projective
  spaces.
\newblock {\em arXiv preprint math/0010157}, 2000.

\bibitem{BB96}
Victor~V Batyrev and Lev~A Borisov.
\newblock On {C}alabi-{Y}au complete intersections in toric varieties.
\newblock {\em Higher-dimensional complex varieties (Trento, 1994)}, pages
  39--65, 1996.

\bibitem{BBBM13}
Janko Boehm, Kathrin Bringmann, Arne Buchholz, and Hannah Markwig.
\newblock Tropical mirror symmetry for elliptic curves.
\newblock {\em arXiv preprint arXiv:1309.5893}, 2013.

\bibitem{Chen-DF}
Q.~Chen.
\newblock Stable logarithmic maps to {D}eligne-{F}altings pairs {I}.
\newblock {\em Ann. of Math.}, 180(2):455--521, 2014.

\bibitem{F05}
Kenji Fukaya.
\newblock Multivalued morse theory, asymptotic analysis and mirror symmetry.
\newblock {\em Graphs and patterns in mathematics and theoretical physics},
  73:205--278, 2005.

\bibitem{FOOO}
Kenji Fukaya, Yong-Geun Oh, Hiroshi Ohta, and Kaoru Ono.
\newblock Lagrangian {F}loer theory on compact toric manifolds {II}: bulk
  deformations.
\newblock {\em Selecta Mathematica}, 17(3):609--711, 2011.

\bibitem{fulton}
William Fulton.
\newblock {\em Introduction to toric varieties}.
\newblock Number 131. Princeton University Press, 1993.

\bibitem{gat}
Andreas Gathmann.
\newblock Tropical algebraic geometry.
\newblock {\em arXiv preprint math/0601322}, 2006.

\bibitem{GKM}
Andreas Gathmann, Michael Kerber, and Hannah Markwig.
\newblock Tropical fans and the moduli spaces of tropical curves.
\newblock {\em Compositio Mathematica}, 145(01):173--195, 2009.

\bibitem{giv}
Alexander~B. Givental.
\newblock Equivariant {G}romov-{W}itten invariants.
\newblock {\em Internat. Math. Res. Notices}, (13):613--663, 1996.

\bibitem{G05}
Mark Gross.
\newblock Toric degenerations and {B}atyrev-{B}orisov duality.
\newblock {\em Mathematische Annalen}, 333(3):645--688, 2005.

\bibitem{MSP2}
Mark Gross.
\newblock Mirror symmetry for $\mathbb{P}^2$ and tropical geometry.
\newblock {\em Advances in Mathematics}, 224(1):169--245, 2010.

\bibitem{kan}
Mark Gross.
\newblock {\em Tropical geometry and mirror symmetry}, volume 114 of {\em CBMS
  Regional Conference Series in Mathematics}.
\newblock Published for the Conference Board of the Mathematical Sciences,
  Washington, DC, 2011.

\bibitem{gsyz}
Mark Gross.
\newblock Mirror symmetry and the {S}trominger-{Y}au-{Z}aslow conjecture.
\newblock {\em arXiv preprint arXiv:1212.4220}, 2012.

\bibitem{GS03}
Mark Gross and Bernd Siebert.
\newblock Affine manifolds, log structures, and mirror symmetry.
\newblock {\em Turk J Math}, 27:33--60, 2003.

\bibitem{GS10}
Mark Gross and Bernd Siebert.
\newblock Mirror symmetry via logarithmic degeneration data, {II}.
\newblock {\em Journal of Algebraic Geometry}, 19(4):679--780, 2010.

\bibitem{GS11}
Mark Gross and Bernd Siebert.
\newblock From real affine geometry to complex geometry.
\newblock {\em Annals of mathematics}, 174(3):1301--1428, 2011.

\bibitem{loggw}
Mark Gross and Bernd Siebert.
\newblock Logarithmic {G}romov-{W}itten invariants.
\newblock {\em Journal of the American Mathematical Society}, 26(2):451--510,
  2013.

\bibitem{GS06}
Mark Gross, Bernd Siebert, et~al.
\newblock Mirror symmetry via logarithmic degeneration data {I}.
\newblock {\em Journal of Differential Geometry}, 72(2):169--338, 2006.

\bibitem{GW00}
Mark Gross, Pelham~MH Wilson, et~al.
\newblock Large complex structure limits of {K}3 surfaces.
\newblock {\em Journal of Differential Geometry}, 55(3):475--546, 2000.

\bibitem{Hi97}
Nigel Hitchin.
\newblock The moduli space of special {L}agrangian submanifolds.
\newblock {\em arXiv preprint dg-ga/9711002}, 1997.

\bibitem{Illusie_log_spaces}
L.~Illusie.
\newblock {\em Logarithmic spaces (according to {K}. {K}ato)}, volume~15 of
  {\em Perspect. Math.}
\newblock Barsotti Symposium in Algebraic Geometry (Abano Terme, 1991),
  Academic Press, San Diego, CA, 1994.

\bibitem{iritani}
Hiroshi Iritani.
\newblock Quantum {D}-modules and generalized mirror transformations.
\newblock {\em Topology}, 47(4):225--276, 2008.

\bibitem{IKS03}
Ilia Itenberg, Viatcheslav Kharlamov, and Eugenii Shustin.
\newblock Welschinger invariant and enumeration of real rational curves.
\newblock {\em International Mathematics research notices},
  2003(49):2639--2653, 2003.

\bibitem{kato-log-curves}
F.~Kato.
\newblock Log smooth deformation and moduli of log smooth curves.
\newblock {\em Internat. J. Math}, 11(2):215--232, 2000.

\bibitem{Kato_log_struct}
K.~Kato.
\newblock {\em Logarithmic structures of {F}ontaine-{I}llusie}.
\newblock Algebraic analysis, geometry, and number theory (Baltimore, MD,
  1988). Johns Hopkins Univ. Press, Baltimore, MD, 1989.

\bibitem{local_B_model}
Y.~Konishi and S.~Minabe.
\newblock Local {B}-model and mixed {H}odge structure.
\newblock {\em Adv. Theor. Math. Phys.}, 14(4):1089--1145, 2010.

\bibitem{Ko13}
Maxim Kontsevich.
\newblock What is tropical mathematics?, 10 2013.

\bibitem{KS06}
Maxim Kontsevich and Yan Soibelman.
\newblock Affine structures and non-{A}rchimedean analytic spaces.
\newblock In {\em The unity of mathematics}, pages 321--385. Springer, 2006.

\bibitem{KS13}
Maxim Kontsevich and Yan Soibelman.
\newblock Wall-crossing structures in {D}onaldson-{T}homas invariants,
  integrable systems and mirror symmetry.
\newblock {\em arXiv preprint arXiv:1303.3253}, 2013.

\bibitem{li-stable-relative}
J.~Li.
\newblock Stable morphisms to singular schemes and relative stable morphisms.
\newblock {\em J. Differential Geom.}, 57(3):509--578, 2001.

\bibitem{Ma99}
IU~I Manin.
\newblock {\em Frobenius manifolds, quantum cohomology, and moduli spaces},
  volume~47.
\newblock American Mathematical Soc., 1999.

\bibitem{M+R}
Hannah Markwig and Johannes Rau.
\newblock Tropical descendant {G}romov-{W}itten invariants.
\newblock {\em manuscripta mathematica}, 129(3):293--335, 2009.

\bibitem{MikAm}
Grigory Mikhalkin.
\newblock Amoebas of algebraic varieties and tropical geometry.
\newblock In {\em Different faces of geometry}, pages 257--300. Springer, 2004.

\bibitem{mik}
Grigory Mikhalkin.
\newblock Enumerative tropical algebraic geometry in $\mathbb{R}^2$.
\newblock {\em Journal of the American Mathematical Society}, 18(2):313--377,
  2005.

\bibitem{milne_et_book}
J.~Milne.
\newblock {\em \'Etale cohomology}, volume~33 of {\em Princeton Mathematical
  Series}.
\newblock Princeton University Press, Princeton, N.J., 1980.

\bibitem{Nishinou}
Takeo Nishinou.
\newblock Disc counting on toric varieties via tropical curves.
\newblock {\em arXiv preprint math/0610660}, 2006.

\bibitem{NS06}
Takeo Nishinou and Bernd Siebert.
\newblock Toric degenerations of toric varieties and tropical curves.
\newblock {\em Duke Mathematical Journal}, 135(1):1--51, 2006.

\bibitem{Ru10}
Helge Ruddat.
\newblock Log {H}odge groups on a toric {C}alabi-{Y}au degeneration.
\newblock {\em Mirror Symmetry and Tropical Geometry, Contemp. Mathematics},
  527:113--164, 2008.

\bibitem{RS14}
Helge Ruddat and Bernd Siebert.
\newblock Canonical coordinates in toric degenerations, 2014.

\bibitem{Sh06}
Eugenii Shustin.
\newblock A tropical calculation of the {W}elschinger invariants of real toric
  {D}el {P}ezzo surfaces.
\newblock {\em arXiv preprint math/0406099}, 2004.

\bibitem{lim_hodge_str}
J.~Steenbrink.
\newblock Limits of {H}odge structures.
\newblock {\em Invent. Math.}, 31(3):229--257, 1975/76.

\bibitem{SYZ96}
Andrew Strominger, Shing-Tung Yau, and Eric Zaslow.
\newblock Mirror symmetry is {T}-duality.
\newblock {\em Nuclear Physics B}, 479(1):243--259, 1996.

\bibitem{welsch}
J.-Y. Welschinger.
\newblock Invariants of real rational symplectic 4-manifolds and lower bounds
  in real enumerative geometry.
\newblock {\em C. R. Math. Acad. Sci. Paris}, (336(4)):341--344, 2003.

\end{thebibliography}

\end{document}